\documentclass[12pt]{amsart} 
\usepackage{fullpage}
\usepackage{mathtools}
\usepackage[all]{xy}
\usepackage{hyperref}
\usepackage{amsfonts,amsmath,amsthm,amssymb}
\usepackage{enumerate}
\usepackage{graphicx}
\usepackage{bbm}
\usepackage{color}

\usepackage{newtxtext,newtxmath}

\newtheorem{thm}{Theorem}
\newtheorem{lma}[thm]{Lemma}

\newtheorem{lmacomp}[thm]{Lemma-Computation}
\newtheorem{cly}[thm]{Corollary}
\newtheorem*{cly*}{Corollary}
\newtheorem*{lma*}{Lemma}
\newtheorem*{thm*}{Theorem}

\newcommand{\GL}{\mathrm{GL}}

\newcommand{\AAA}{\mathcal{A}}

\newcommand{\CCC}{\mathcal{C}}
\newcommand{\DDD}{\mathcal{D}}
\newcommand{\MMM}{\mathcal{M}}
\newcommand{\OOO}{\mathcal{O}}

\newcommand{\LLL}{\mathcal{L}}

\renewcommand{\AA}{\mathbb{A}}

\newcommand{\CC}{\mathbb{C}}

\newcommand{\RR}{\mathbb{R}}

\newcommand{\ZZ}{\mathbb{Z}}

\newcommand{\PP}{\mathbb{P}}
\newcommand{\QQ}{\mathbb{Q}}

\newcommand{\tto}[1]{\xrightarrow{#1}}

\newcommand{\pq}[1]{\left(#1\right)}
\newcommand{\aq}[1]{\left|#1\right|}

\newcommand{\cq}[1]{\left\{#1\right\}}
\newcommand{\bq}[1]{\left[#1\right]}

\newcommand{\ol}[1]{\overline{#1}}
\newcommand{\del}{\partial}

\newcommand{\III}{\mathcal{I}}
\newcommand{\inv}{^{-1}}
\newcommand{\f}[2]{\frac{#1}{#2}}

\newcommand{\eps}{\varepsilon}
\newcommand{\sm}{\setminus}

\newcommand{\du}{^{*}}

\DeclareMathOperator{\Hom}{Hom}

\DeclareMathOperator{\Ext}{Ext}

\DeclareMathOperator{\Spec}{Spec}

\newcommand{\CLie}{\mathcal{CL}}

\newcommand{\simplecap}[2]{\refstepcounter{figure}\vskip1emFigure \ref{#1}. #2\label{#1}}
\newcommand{\Cor}{\text{\rm Cor}}
\newcommand{\QSh}{{\rm QSh}}

\newcommand{\Sh}{{\rm Sh}}

\newcommand{\pg}[2]{\big(#1\;|\;#2\big)}
\newcommand{\pqg}[2]{\left(#1\mid #2\right)}

\newcommand{\ldbr}{\left[\hspace{-0.02in}\left[\,}
\newcommand{\rdbr}{\,\right]\hspace{-0.02in}\right]}
\newcommand{\bbq}[1]{\ldbr#1\rdbr}

\newcommand{\HHH}{\mathcal{H}}

\DeclareMathOperator{\wt}{wt}

\newcommand{\HGF}{\mathbf{\Lambda}}
\newcommand{\GF}{\mathbf{\Lambda}^*}
\newcommand{\nil}{\text{\rm nil}}
\newcommand{\Hod}{\text{\rm Hod}}
\newcommand{\HT}{\text{\rm HT}}
\newcommand{\MT}{\text{\rm MT}}
\newcommand{\gr}{\text{\rm gr}}
\DeclareMathOperator{\sgn}{sgn}
\newcommand{\Li}{\text{\rm Li}}
\newcommand{\Lie}{\text{\rm Lie}}
\newcommand{\Der}{\text{\rm Der}}

\newcommand{\reg}{\text{\rm reg}}

\newcommand{\Mot}{\text{\rm Mot}}
\newcommand{\End}{\text{\rm End}}
\newcommand{\MiM}{\mathcal{MM}}
\newcommand{\PM}{\mathcal{PM}}
\newcommand{\MTM}{\mathcal{MTM}}

\newcommand{\MTF}{{\text{\rm MT}/F}}
\newcommand{\HS}{\mathrm{HS}}
\newcommand{\HTS}{\mathrm{HT}}
\newcommand{\MHS}{\mathrm{MH}}
\newcommand{\MHTS}{\mathrm{MHT}}

\newcommand{\todo}[1]{\relax}

\title{Shuffle relations for Hodge and motivic correlators}
\author{Nikolay Malkin}

\begin{document}

\maketitle

 \begin{abstract}
    The Hodge correlators $\Cor_\HHH(z_0,z_1,\dots,z_n)$ are functions of several complex variables, defined by (Goncharov, 2008) by an explicit integral formula. They satisfy some linear relations: dihedral symmetry relations, distribution relations, and the shuffle relations. 
    
    We found new \emph{second shuffle relations}. When $z_i\in\cq0\cup\mu_N$, where $\mu_N$ are the $N$-th roots of unity, they are expected to give almost all relations.
    
    When $z_i$ run through a finite subset $S$ of $\CC$, the Hodge correlators describe the real mixed Hodge-Tate structure on the pronilpotent completion of the fundamental group $\pi_1^\nil(\CC\PP^1\sm(S\cup\cq\infty),v_\infty)$. The latter is a Lie algebra in the category of mixed $\QQ$-Hodge-Tate structures. The Hodge correlators are lifted to canonical elements $\Cor_\Hod(z_0,\dots,z_n)$ in the Tannakian Lie coalgebra $\Lie_\HT^\vee$ of this category. We prove that these elements satisfy the second shuffle relations.
    
    Let $S\subset\ol\QQ$. The pronilpotent fundamental group is the Betti realization of the motivic fundamental group, which is a Lie algebra in the category of mixed Tate motives over $\ol\QQ$. The Hodge correlators are lifted to elements $\Cor_\Mot(z_0,\dots,z_n)$ in the Tannakian Lie coalgebra $\Lie_\MT^\vee$ of the category of mixed Tate motives. We prove the second shuffle relations for these motivic elements.
    
    The universal enveloping algebra of $\Lie_\MT^\vee$ was described by Goncharov via motivic multiple polylogarithms, which obey a similar yet different set of double shuffle relations. Motivic correlators have several advantages: they obey dihedral symmetry relations at all points, not only at roots of unity;  they are defined for any curve, and the double shuffle relations admit a generalization to elliptic curve; and they describe elements of the motivic Lie coalgebra rather than its universal enveloping algebra.

\end{abstract}

\setcounter{tocdepth}{2}
\tableofcontents

\section{Introduction and main results}
\label{sec:intro}

\subsection{Summary}

The Hodge correlators $\Cor_\HHH(z_0,z_1,\dots,z_n)$ are functions of several complex variables, defined by an explicit integral formula in~\cite{goncharov-hodge-correlators}. They satisfy some linear relations: the dihedral symmetry relations, the distribution relations, and the shuffle relations.

We found new relations, called \emph{second shuffle relations}. When $z_i\in\cq0\cup\mu_N$, where $\mu_N$ are the $N$-th roots of unity, they should give almost all relations: the results of \cite{goncharov-motivic-modular} suggest that the other relations are sporadic, i.e., cannot be described by universal formulae.

When $z_i$ run through a finite subset $S$ of $\CC$, the Hodge correlators are the canonical real periods of the mixed Hodge-Tate structures on the pronilpotent completion of the fundamental group $\pi_1^\nil(\CC\PP^1\sm(S\cup\cq\infty),v_\infty)$, with the tangential base point at $\infty$. The latter is a Lie algebra in the category of mixed $\QQ$-Hodge-Tate structures. The Hodge correlators describe the real mixed Hodge structure on this Lie algebra tensored over $\QQ$ by $\RR$.

The category of mixed $\QQ$-Hodge-Tate structures is canonically equivalent to the category of representations of a graded Lie algebra over $\QQ$. Let us take its image in the representation defining $\pi_1^\nil(\CC\PP^1\sm(S\cup\cq\infty),v_\infty)$, and consider the graded dual Lie coalgebra $\Lie_\HT^\vee(S)$. The Hodge correlators were lifted in~\cite{goncharov-hodge-correlators} to canonical elements \begin{equation}\Cor_\Hod(z_0,\dots,z_n)\in\Lie_\HT^\vee(S).\label{eqn:cor_hod}\end{equation} The real numbers $\Cor_\HHH$ are the canonical real periods of these elements. We prove that our new relations can be lifted to relations on the elements~(\ref{eqn:cor_hod}).

Let $S\subset\ol\QQ\subset\CC$. The Lie algebra $\pi_1^\nil(\CC\PP^1\sm (S\cup\cq\infty),v_\infty)$ is the Betti realization of the motivic fundamental group $\pi_1^\Mot(\PP^1\sm(S\cup\cq\infty),v_\infty)$. The latter is a Lie algebra in the category of mixed Tate motives over $\ol\QQ$, defined in \cite{deligne-goncharov}. This category is identified with the category of representations of the motivic Galois Lie algebra. Just like in the Hodge case, we take the image of this Lie algebra in the representation provided by the motivic fundamental group, and consider the graded dual Lie coalgebra $\Lie_\MT^\vee(S)$. In~\cite{goncharov-hodge-correlators}, the elements (\ref{eqn:cor_hod}) were lifted to elements \begin{equation}\Cor_\Mot(z_0,\dots,z_n)\in\Lie_\MT^\vee(S).\label{eqn:cor_mot}\end{equation} We prove that our relations can be upgraded to linear relations on these elements.

The universal enveloping algebra for the Lie coalgebra $\Lie_\MT^\vee(S)$ was described in \cite{goncharov-polylogs-tate} via motivic multiple polylogarithms. The motivic double shuffle relations for them were proved in \cite{goncharov-periods-mm}. The explicit relation between motivic correlators and multiple polylogarithms is an interesting open problem.

The multiple polylogarithms obey a similar system of double shuffle relations, but the dihedral symmetry relation holds only at roots of unity. The combinatorics of those relations, originally described by \cite{goncharov-polylogs-modular}-\cite{goncharov-polylogs-tate}, were studied further by \cite{racinet}.

The motivic correlator description of $\pi_1^\Mot(\PP^1\sm(S\cup\cq\infty),v_\infty)$ has several advantages. Most importantly, motivic correlators are defined for any algebraic curve, not only $\AA^1\sm S$, and the double shuffle relations admit a generalization to elliptic curves \cite{malkin-ec}. The motivic correlators obey double shuffle and cyclic symmetry relations at all points. Motivic correlators describe elements of the Lie coalgebra rather than its universal enveloping algebra. Finally, they give the best way to describe the mysterious connection between the Lie coalgebra $\Lie_\MT^\vee(\cq0\cup\mu_N)$ and modular manifolds \cite{goncharov-motivic-modular}.

\subsubsection*{Acknowledgements} I am grateful to A.B.\ Goncharov for introducing me to this problem, for many helpful discussions and explanations, and for comments on a draft of this paper.

This material is based upon work supported by the National Science Foundation under grant
DMS-1440140 while the author was in residence at the Mathematical Sciences Research Institute in Berkeley, California, during the Fall 2019 semester. The author also acknowledges support from NSF grants DMS-1107452, 1107263, 1107367 ``RNMS: Geometric Structures and Representation Varieties'' (the GEAR Network).

\subsection{Hodge correlators and shuffle relations}

We describe a family of functions of several complex variables, the Hodge correlators (\cite{goncharov-hodge-correlators}).\footnote{In this paper, ``Hodge correlators'' will refer only to Hodge correlators associated to the curve $\PP^1$.} Our main result is a set of functional equations on the Hodge correlators and the Hodge-theoretic and motivic upgrades of these relations.

\subsubsection{Definition} 

Let $z_0,\dots,z_n\in\CC$. We define the Hodge correlator of weight $n$, $\Cor_\HHH(z_0,\dots,z_n)$.

Draw a disc in the plane with a sequence of points $V^\del=\cq{v_0,\dots,v_n}$ placed counterclockwise around the boundary, and label $v_i$ by the value $z_i$. Choose a plane trivalent tree $T$ inside the disc with leaves at the labeled boundary vertices.  Such a tree has $n-1$ interior vertices $V^\circ$ and $2n-1$ edges $E=\cq{E_0,\dots,E_{2n-2}}$. The embedding into the plane gives a canonical orientation $\rm{Or}_T\in\cq{\pm1}$ (a choice of component of $\RR^{\wedge E}$, i.e., ordering of the edges up to even permutation). 

Let us assign to each edge $E_j$ a function $f_j$ on \[{\mathbf X}:=\CC^{V^\circ}\times \CC^{V^\del}.\] Precisely, to an edge $E_i=(u,v)$, assign $f_i=(2\pi i)\inv\log\aq{x_u-x_v}$, where $x_u$ is the coordinate on $\mathbf{X}$ corresponding to a vertex $u$. Then fix the coordinate at each boundary vertex $v_i$ to be $z_i$. Abusing notation, also denote by $f_j$ the restriction of $f_j$ to $\CC^{V^\circ}$ with the boundary coordinates fixed.

Setting $d^\CC=\del-\ol\del$, we define:
\begin{equation}
c_T(z_0,\dots,z_n)=(-4)^{n-1}\binom{2n-2}{n-1}\inv\mathrm{Or}_T\int_{\CC^{V^\circ}} f_0\,d^\CC f_1\wedge\dots\wedge d^\CC f_{2n-2},
\label{eqn:integral_onetree_p1}
\end{equation}
This expression is independent of the numbering of the edges. The Hodge correlator is defined as the sum of these integrals over all plane trivalent trees $T$:
\begin{equation*}
    \Cor_\HHH(z_0,\dots,z_n)
    =\sum_Tc_T(z_0,\dots,z_n).
\end{equation*}
It takes values in $(2\pi i)^{-n}\RR$. The simplest example, in weight 1, is shown in Fig.~\ref{fig:w1_hc}.

\begin{figure}[h]
\begin{center}
\includegraphics{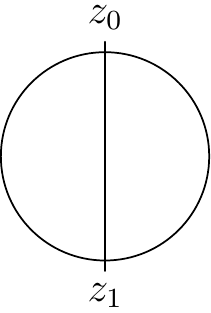}
\end{center}
\simplecap{fig:w1_hc}{$\Cor_\HHH(z_0,z_1)=(2\pi i)^{-1}\log\aq{z_0-z_1}$.}
\end{figure}

In weight 2, the Hodge correlators are given by
\begin{align*}
    \Cor_\HHH(z_0,z_1,z_2)&=-\f18\int_x(2\pi i)^{-3}\log\aq{x-z_0}\,d^\CC\log\aq{x-z_1}\wedge d^\CC\log\aq{x-z_2}.
\end{align*}
This integral is described by the Feynman diagram in Fig.~\ref{fig:w2_hc}.
\begin{figure}[h]
\begin{center}
    \includegraphics{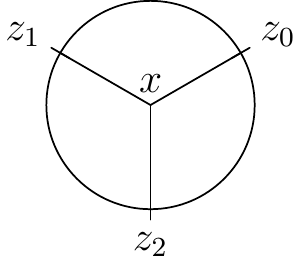}
\end{center}
\simplecap{fig:w2_hc}{}
\end{figure}

Recall the single-valued version of the dilogarithm, called the Bloch-Wigner function:
\[\LLL_2(z)=\Im\pq{\Li_2(z)}+\log\aq z\arg(1-z),\quad\Im(a+bi):=b.\]
The weight 2 Hodge correlator integral can be calculated explicitly as
\begin{equation}
    \Cor_\HHH(z_0,z_1,z_2)=(2\pi i)^{-2}\LLL_2\pq{\f{z_1-z_0}{z_2-z_0}}.
    \label{eqn:h_w2_l2}
\end{equation}

\subsubsection{Properties}

The Hodge correlators satisfy \emph{dihedral symmetry} relations:
\begin{align*}
    \Cor_\HHH(z_0,z_1,\dots,z_n)&=\Cor_\HHH(z_1,\dots,z_n,z_0)\\&=(-1)^{n+1}\Cor_\HHH(z_n,\dots,z_1,z_0).
\end{align*}
One can show using (\ref{eqn:integral_onetree_p1}) that the Hodge correlators are invariant under an additive shift of the arguments. In weight $>1$, they are also invariant under a multiplicative shift:
\begin{align*}
    \Cor_\HHH(z_0,\dots,z_n)&=\Cor_\HHH(z_0+a,\dots,z_n+a),\\
    \Cor_\HHH(z_0,\dots,z_n)&=\Cor_\HHH(az_0,\dots,az_n)\quad(a\in\CC\du, n>1).
\end{align*}

Furthermore, the Hodge correlators satisfy \emph{shuffle relations}: for $r,s\geq1$ and $z_0,\dots,z_{r+s}\in\CC$,
\begin{equation}
    \sum_{\sigma\in\Sigma_{r,s}}\Cor_\HHH(z_0,z_{\sigma\inv(1)},z_{\sigma\inv(2)},\dots,z_{\sigma\inv(r+s)})=0,
    \label{eqn:first_shuffle_cyc}
\end{equation}
where $\Sigma_{r,s}\subset S_{r+s}$ is the set of \emph{$(r,s)$-shuffles}, consisting of the permutations $\sigma$ such that \[\sigma(1)<\dots<\sigma(r),\quad\sigma(r+1)<\dots<\sigma(r+s).\]

\newcommand{\zzone}{\textcolor{blue}{z_1}}
\newcommand{\zztwo}{\textcolor{red}{z_2}}
For example, the $(1,1)$-shuffle relation states:
\[\Cor_\HHH(z_0,\zzone,\zztwo)+\Cor_\HHH(z_0,\zztwo,\zzone)=0;\]
the $(2,1)$-shuffle relation is:
\renewcommand{\zztwo}{\textcolor{blue}{z_2}}
\newcommand{\zzthr}{\textcolor{red}{z_3}}
\[\Cor_\HHH(z_0,\zzone,\zztwo,\zzthr)+\Cor_\HHH(z_0,\zzone,\zzthr,\zztwo)+\Cor_\HHH(z_0,\zzthr,\zzone,\zztwo)=0.\]
The shuffle relations may be considered ``easy'' because they hold on the level of the sum over trees of the \emph{integrands} in (\ref{eqn:integral_onetree_p1}).

\subsubsection{Second shuffle relation}

We found another relation on the Hodge correlators. Together, the two relations form the \emph{double shuffle relations}. To state the new relations, we must introduce some notation.

Because of the multiplicative invariance (in weight $>1$) of Hodge correlators, it is possible and convenient to introduce an inhomogeneous notation for them, where the arguments are represented by the quotients between successive nonzero values and the number of 0s between them. Precisely, given $w_0,\dots,w_k\in\CC\du$ such that $w_0w_1\dots w_k=1$, define
\begin{align*}
&\Cor_\HHH\du(w_0|n_0,w_1|n_1,\dots,w_k|n_k):=\\&=\Cor_\HHH(\underbrace{0,\dots,0}_{n_0},1,\underbrace{0,\dots,0}_{n_1},w_1,\underbrace{0,\dots,0}_{n_2},w_1w_2,\dots,\underbrace{0,\dots,0}_{n_k},w_1\dots w_k).
\end{align*}
This definition is illustrated in Fig.~\ref{fig:homo}.
\begin{figure}[h]
\begin{center}
\begin{tabular}{ccc}
    \includegraphics[width=0.3\textwidth]{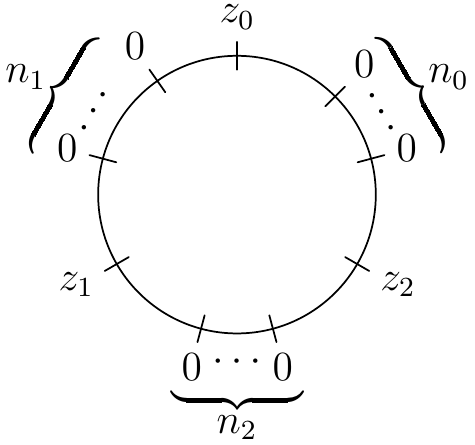}&
    \includegraphics[width=0.3\textwidth]{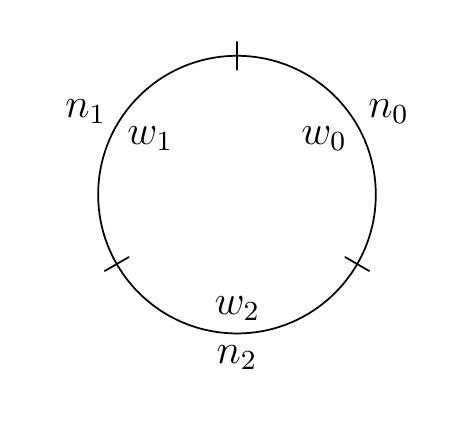}
    &\parbox{2in}{\vspace{-1.5in}$\displaystyle w_i=\f{z_i}{z_{i-1}}$}

\end{tabular}
\end{center}
\simplecap{fig:homo}{$\Cor_\HHH(z_0,\underbrace{0,\dots,0}_{n_1},z_1,\underbrace{0,\dots,0}_{n_2},z_2,\underbrace{0,\dots,0}_{n_0})\equiv\Cor_\HHH\du(w_1|n_1,w_2|n_2,w_3|n_3)$.}
\end{figure}

Define the \emph{depth} of an expression $\Cor_\HHH(z_0,\dots,z_n)$ to be one less than the number of arguments in the multiplicative notation, that is, $k$ in the formula above.

Our new shuffle relation states:
\begin{equation}
    \sum_{\sigma\in\Sigma_{r,s}}\Cor_\HHH\du(w_{\sigma\inv(1)}|n_{\sigma\inv(1)},\dots,w_{\sigma\inv(r+s)}|n_{\sigma\inv(r+s)},w_0|n_0)+\text{lower-depth terms}=0.
    \label{eqn:second_shuffle_cyc}
\end{equation}
That is, we shuffle two ordered sets of expressions $(w_i|n_i)$, while leaving the segment $(w_0|n_0)$ fixed. For example the $(1,1)$-shuffle relation begins:
\begin{center}
\begin{tabular}{cc}
\includegraphics{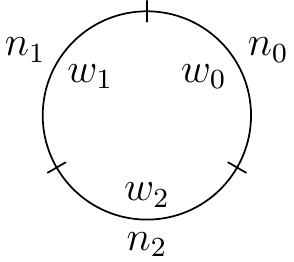}
&\includegraphics{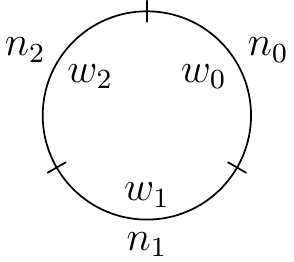}\\\,\\
$\Cor_\HHH\du(w_1|n_1,w_2|n_2,w_0|n_0)$&$+\:\Cor_\HHH\du(w_2|n_2,w_1|n_1,w_0|n_0)$\\\,
\end{tabular}
\end{center}

To describe the lower-depth terms, we need the notion of \emph{quasishuffle}. Let $A=\cq{a_1<\dots<a_r}$ and $B=\cq{b_1<\dots<b_s}$ be two ordered sets. A quasishuffle of $A$ and $B$ is a sequence of slots $\cq{1,\dots,M}$ and a placement of each element of $A\cup B$ in a slot, such that each slot is filled with one of:
\begin{itemize}
    \item some $a_i\in A$,
    \item some $b_j\in B$, 
    \item a pair $\cq{a_i,b_j}$,
\end{itemize} 
and the sequence of slots containing the $a_1,\dots,a_r$ and the sequence of slots containing the $b_1,\dots,b_s$ are ordered left to right. If $a_i$ and $b_j$ share a slot, they are said to \emph{collide}. If no elements collide, the quasishuffle is said to be a shuffle.

Let $A=\cq{1,\dots,r}$ and $B=\cq{r+1,\dots,r+s}$ with the natural orders. Then, equivalently, the quasishuffles are the surjective maps $\cq{1,\dots,r+s}\tto\sigma\cq{1,\dots,M_\sigma}$ that are strictly increasing on $1,\dots,r$ and $r+1,\dots,r+s$. 

Indices $i\in\cq{1,\dots,r}$ collide with indices $j\in\cq{r+1,\dots,r+s}$ whenever $\sigma(i)=\sigma(j)$. Let $\ol\Sigma_{r,s}$ be the set of such quasishuffles.

A quasishuffle $\sigma$ is a shuffle if $M_\sigma=r+s$. Recall the set of $(r,s)$-shuffles $\Sigma_{r,s}\subset S_{r+s}$. We naturally identify $\Sigma_{r,s}$ with the subset of the shuffles in $\ol\Sigma_{r,s}$.

The lower-depth terms in (\ref{eqn:second_shuffle_cyc}) come in two kinds:
\begin{enumerate}[(1)]
\item Terms coming from the $(r,s)$-\emph{quasishuffles} that are not proper shuffles. Whenever the segments $(w_i|n_j)$ and $(w_j|n_j)$ collide, we get a new segment $(w_iw_j|n_i+n_j+1)$ in their place -- a 0 is inserted -- and the term picks up a negative sign.

For the $(1,1)$-shuffle relation, there is only one quasishuffle that is not a shuffle. In this quasishuffle, the two segments $(w_1|n_1)$ and $(w_2|n_2)$ collide:
\begin{center}
    \begin{tabular}{c}
    \includegraphics{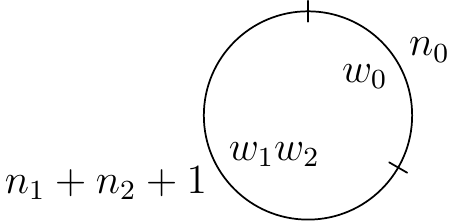}\\\,\\
    $-\Cor_\HHH\du(w_1w_2|n_1+n_2+1,w_0|n_0)$\\\,
    \end{tabular}
\end{center}
\item Two extra terms: one where the segments $w_1,\dots,w_r$ appear in order and the remaining segments $w_{r+1},\dots,w_{r+s},w_0$ collapse; another where the segments $w_{r+1},\dots,w_{r+s}$ appear in order and $w_1,\dots,w_r,w_0$ collapse.   These terms come with a negative sign.

For the $(1,1)$-shuffle relation:
\begin{center}
    \begin{tabular}{cc}
    \includegraphics{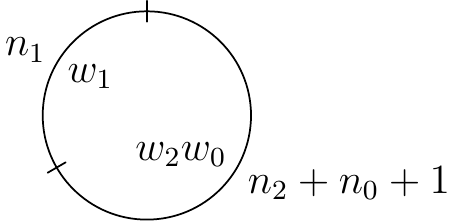}
    &\includegraphics{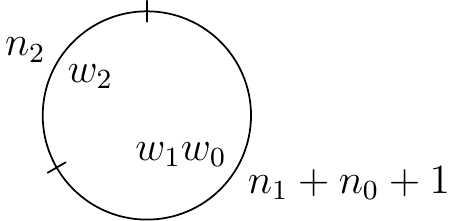}\\\,\\
    $-\Cor_\HHH\du(w_1|n_1,w_2w_0|n_2+n_0+1)$&$-\:\Cor_\HHH\du(w_2|n_2,w_1w_0|n_1+n_0+1)$\\\,
    \end{tabular}
\end{center}
\end{enumerate}

In summary, the $(1,1)$-shuffle relation states, for $w_0,w_1,w_2\in\CC\du$ and $w_0w_1w_2=1$,
\begin{align*}
    \Cor_\HHH\du(w_1|n_1,w_2|n_2,w_0|n_0)&+\Cor_\HHH\du(w_2|n_1,w_1|n_1,w_0|n_0)\\&-\Cor_\HHH\du(w_1w_2|n_1+n_2+1,w_0|n_0)\\&-\Cor_\HHH\du(w_1|n_1,w_2w_0|n_2+n_0+1)\\&-\Cor_\HHH\du(w_2|n_2,w_1w_0|n_1+n_0+1)&=0.
\end{align*}
It is already a nontrivial relation, which is not easy to prove from the definition (\ref{eqn:integral_onetree_p1}) even for $n_0=n_1=n_2=0$.

By formula (\ref{eqn:h_w2_l2}), Hodge correlators in weight 2 are expressed in a simple way in terms of the Bloch-Wigner function $\LLL_2$. The $(1,1)$-shuffle relation with $n_0=n_1=n_2=0$ is equivalent to the five-term relation,
\[\LLL_2\pq{\f{1-w_1}{1-w_1w_2}}+\LLL_2\pq{\f{1-w_2}{1-w_1w_2}}+\LLL_2(1-w_1w_2)+\LLL_2(w_1)+\LLL_2(w_2)=0.\]
According to \cite{bloch-irvine}, this is essentially the only functional equation for $\LLL_2$. It follows that the dihedral symmetry and shuffle relations are the \emph{only} relations between the Hodge correlators in weight 2.

\newcommand{\xxxx}{\color{blue}{w_1}}
\newcommand{\yyyy}{\color{blue}{w_2}}
\newcommand{\zzzz}{\color{red}{w_3}}
For further illustration, let us write out the $(2,1)$-shuffle relation for the Hodge correlator \[\Cor_\HHH\du(w_1|0,w_2|1|w_3|1,w_4|0),\] where $\xxxx$ and $\yyyy$ will be shuffled with $\zzzz$:

\newcommand{\bbbb}[2]{\begin{bmatrix}#1\\#2\end{bmatrix}}
\begin{enumerate}[(1)]
\item[(0)] There are three terms from the shuffles:
\begin{center}
\begin{tabular}{ccc}
$\xxxx\,\yyyy\,\zzzz$&$\xxxx\,\zzzz\,\yyyy$&$\zzzz\,\xxxx\,\yyyy$\\
\includegraphics{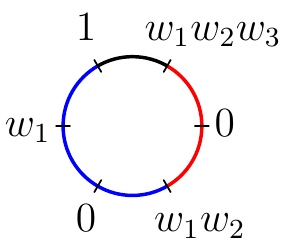}&\includegraphics{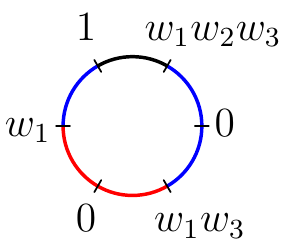}&\includegraphics{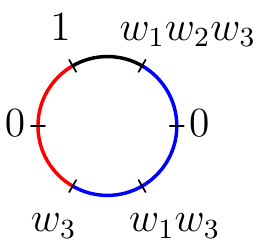}\\

\end{tabular}
\end{center}
\item There are two terms from the quasishuffles that are not shuffles:
\begin{center}
\begin{tabular}{ccc}
$\bbbb{\xxxx}{\zzzz}\,\yyyy$&$\xxxx\,\bbbb{\yyyy}{\zzzz}$\\
\includegraphics{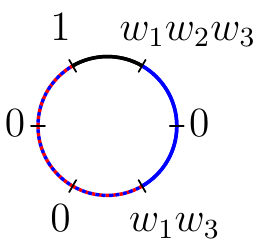}&\includegraphics{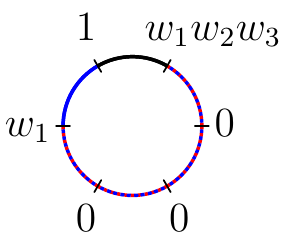}\\

\end{tabular}
\end{center}
\item There are two additional terms:
\begin{center}
\begin{tabular}{ccc}
\includegraphics{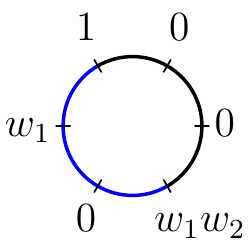}&\includegraphics{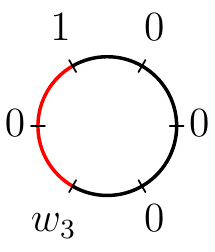}\\

\end{tabular}
\end{center}
\end{enumerate}
The full relation is then
\begin{align*}
    \Cor_\HHH\du(w_1|0,w_2|1,w_3|1,w_4|0)&+\Cor_\HHH\du(w_1|0,w_3|1,w_2|1,w_4|0)+\Cor_\HHH\du(w_3|1,w_1|0,w_2|1,w_4|0)\\
    &-\Cor_\HHH\du(w_1w_3|2,w_2|1,w_4|0)-\Cor_\HHH\du(w_1|0,w_2w_3|3,w_4|0)\\
    &-\Cor_\HHH\du(w_1|0,w_2|1,(w_1w_2)\inv|2)-\Cor_\HHH\du(w_3|1,w_3\inv|3)&=0,
\end{align*}
where the $3+2+2$ terms in the three rows match the $3+2+2$ pictures above.

We now write out the general relation:
\begin{thm}
    \begin{enumerate}[(a)]
    \item
    Suppose that $r,s>1$ and that not all $n_i=0$ or not all $w_i=1$. Then the Hodge correlators satisfy the relation:
    \begin{align*}
        &\sum_{\sigma\in\ol\Sigma_{r,s}}(-1)^{r+s-M_\sigma}\Cor_\HHH\du(w_{\sigma\inv(1)}|n_{\sigma\inv(1)},\dots,w_{\sigma\inv(M_\sigma)}|n_{\sigma\inv(M_\sigma)},w_0|n_0)\\
        &-\Cor_\HHH\du(w_1|n_1,\dots,w_r|n_r,w_{\cq{r+1,\dots,r+s,0}}|n_{\cq{r+1,\dots,r+s,0}})\\
        &-\Cor_\HHH\du(w_{r+1}|n_{r+1},\dots,w_{r+s}|n_{r+s},w_{\cq{1,\dots,r,0}}|n_{\cq{1,\dots,r,0}})&=0,
    \end{align*}
    where
    \[n_S=\sum_{i\in S}(n_i+1)-1,\quad w_S=\prod_{i\in S}w_i.\]
    \item The Hodge correlators satisfy all specializations of this relation as any subset of the $w_i$ $(1\leq i\leq n)$ approaches 0.
\end{enumerate}
    \label{thm:period_main}
\end{thm}

\subsubsection{Applications}

Theorem~\ref{thm:period_main} gives simple proofs of certain results of \cite{goncharov-rudenko}.

\begin{cly}[\cite{goncharov-rudenko}, Proposition 2.8]
    For $n>2$, every Hodge correlator of weight $n$ is a linear combination of Hodge correlators of weight $n$ and depth at most $n-2$.

    Precisely, for $z_1,\dots,z_n\in\CC\du$, we have
    \begin{align}
        \Cor_\HHH(z_1,\dots,z_n,0)
        &=\sum_{i=1}^n\Cor_\HHH\pq{z_1,\dots,z_{i-1},z_i,z_i\f{z_1}{z_n},\dots,z_{n-1}\f{z_1}{z_n},z_n\f{z_1}{z_n}}\nonumber\\
        &\quad-\sum_{i=2}^n\Cor_\HHH\pq{z_1,\dots,z_{i-1},0,z_i\f{z_1}{z_n},\dots,z_{n-1}\f{z_1}{z_n},z_n\f{z_1}{z_n}}\nonumber\\
        &\quad-\Cor_\HHH\pq{z_1,z_1\f{z_1}{z_n},0,\dots,0}.\label{eqn:lower_depth_reduction}
    \end{align}
    \label{cly:gr28}
\end{cly}

In weight 3, we deduce the Hodge correlator version of relations (27) and (29) from \cite{goncharov-rudenko}. 
\begin{cly}
    The Hodge correlators in weight 3 satisfy the relations:
\begin{align}
    \Cor_\HHH(1,0,0,x)&+\Cor_\HHH(1,0,0,1-x)+\Cor_\HHH(1,0,0,1-x\inv)=\Cor_\HHH(1,0,0,1),\\
\Cor_\HHH(0,x,1,y)&=
-\Cor_\HHH(1,0,0,1-x\inv)
-\Cor_\HHH(1,0,0,1-y\inv)
-\Cor_\HHH\pq{1,0,0,\f yx}\nonumber\\\label{eqn:corr_29}
&\quad-\Cor_\HHH\pq{1,0,0,\f{1-y}{1-x}}
+\Cor_\HHH\pq{1,0,0,\f{1-y\inv}{1-x\inv}}
+\Cor_\HHH(1,0,0,1).
\end{align}
\label{cly:gr2729}
\end{cly}

We have noted that the double shuffle and dihedral symmetry relations give all relations between Hodge correlators in weight 2. 

In weight 3, the Hodge correlators of depth 1 are expressed in terms of the single-valued trilogarithm $\LLL_3$ (see \S\ref{sec:rel_hodge_polylog}). By the results of \cite{goncharov-rudenko}, the relations (\ref{eqn:corr_29}) imply the general functional equation for $\LLL_3$ (\cite{goncharov-conf-polylogs}). We conclude that the double shuffle relations for Hodge correlators imply all functional equations for $\LLL_2$ and $\LLL_3$.

\subsection{Quasidihedral Lie coalgebras}

Let $G$ be an abelian group. We use the multiplicative notation for $G$; the identity element is $1\in G$. Typically, $G$ will be the multiplicative group of a field $F^\times$ or the group of $N$-th roots of unity $\mu_N$. We adjoin to $G$ a formal element 0, where $0\cdot g=0$ for $g\in G\cup\cq0$.

We define the \emph{quasidihedral Lie coalgebra} $\DDD(G)$. It generalizes the dihedral Lie coalgebra of \cite{goncharov-dihedral}; the latter is the associated graded for the depth filtration of of $\DDD(G)$. The aim of the construction of $\DDD(G)$ is twofold:
\begin{enumerate}[(1)]
    \item It is the main combinatorial ingredient in the proof of the double shuffle relations for correlators.
    \item The Lie coalgebra $\DDD(G)$ describes the coproduct of motivic correlators.
\end{enumerate}

\subsubsection{Cyclic Lie coalgebra}

Let $V$ be the $\QQ$-vector space with basis indexed by $G\cup\cq0$

Let $T(V)=\bigoplus_{n\geq0}V^{\otimes n}$ be the tensor algebra of $V$ over $\QQ$. We impose a grading by weight, where $V^{\otimes n}$ has weight $n-1$. Then define the cyclic Lie coalgebra, as a vector space, by

\[\CCC(G)=\f{T(V)}{\text{cyclic symmetry}}.\]
It is positively graded and generated in weight $n$ by elements $x_0\otimes\dots\otimes x_n$ modulo the relation $x_0\otimes\dots\otimes x_n=x_1\otimes\dots\otimes x_n\otimes x_0$. We can represent these elements by elements of $G\cup\cq0$ written counterclockwise at marked points on a circle.

The coproduct on $\CCC(G)$ is defined on such a generator by splitting the circle into two arcs that share exactly one point. That is, consider a line inside the circle, starting at a marked point and ending between two marked points. It splits the circle into two parts, representing generators $x'$ and $x''$, and the coproduct of $x_0\otimes\dots\otimes x_n$ is the sum of $x'\wedge x''$ over all such cuts.
\begin{center}
\includegraphics{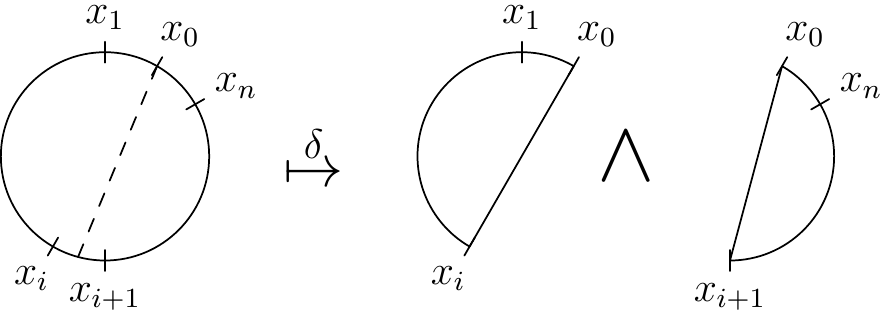}
\end{center}
Precisely, the coproduct is defined by
\begin{equation}
    \delta\pq{x_0\otimes\dots\otimes x_n}=\sum_{\rm cyc}\sum_{i=1}^{n-1}\pq{x_0\otimes x_1\otimes\dots\otimes x_i}\wedge\pq{x_0\otimes x_{i+1}\otimes\dots\otimes x_n}.
    \label{eqn:cyc_coproduct}
\end{equation}
It respects the weight grading and satisfies the co-Jacobi identity.

We will write elements of $\CCC(G)$ as 
\[C(x_0,\dots,x_n)=x_0\otimes\dots\otimes x_n.\]
Also introduce a notation, analogous to that for Hodge correlators, for $w_0,\dots,w_k\in G$ with $w_0\dots w_k=1$:
\begin{align*}
    &C\du(w_0|n_0,w_1|n_1,\dots,w_k|n_k):=\\&=C(\underbrace{0,\dots,0}_{n_0},1,\underbrace{0,\dots,0}_{n_1},w_1,\underbrace{0,\dots,0}_{n_2},\dots,w_1\dots w_{k-1},\underbrace{0,\dots,0}_{n_k},w_1\dots w_k).
\end{align*}

\subsubsection{Relations}

A \emph{first shuffle} in $\CCC(G)$ is an element of the form
\[\sum_{\sigma\in\Sigma_{r,s}}C(x_0,x_{\sigma\inv(1)},x_{\sigma\inv(2)},\dots,x_{\sigma\inv(r+s)}).\] 
Define
\[\widetilde\DDD(G)=\f{\CCC(G)}{\text{first shuffles, scaling relations, distribution relations}}.\]
The scaling relations we impose are:
\begin{enumerate}[(1)]
\item In weight 1, we have $C(0,0)=0$ and $C(ab,ac)=C(0,a)+C(b,c)$ for $a\in G$.
\item In weight $>1$, multiplicative invariance: 
\[C(x_0,\dots,x_n)=C(ax_0,\dots,ax_n),\quad a\in G.\]
\end{enumerate}
The distribution relations are the following. For $l\in\ZZ_{>0}$, let $G_l$ denote the $l$-torsion of $G$. Suppose that $G_l$ is finite and $l$ divides $\aq{G_l}$, and suppose $x_0,\dots,x_n\in G\cup\cq{0}$ are divisible by $l$ (note $0$ is always divisible by $l$). Let $m$ be the number of 0s among the $x_i$. Then the relation is
\begin{equation}
    C(x_0,\dots,x_n)=\f{l^m}{\aq{G_l}}\sum_{y_i^l=x_i}C(y_0,\dots,y_n),
    \label{eqn:distr_rel}
\end{equation}
except in the case that $n=1$ and $x_0=x_1$.

The following is immediate from the constructions of \cite{goncharov-dihedral} (Theorem 4.3).
\begin{thm}
The first shuffles, scaling relations, and distribution relations generate a coideal in $\CCC(G)$. The coproduct on $\CCC(G)$ descends to a well-defined coproduct on $\widetilde\DDD^\vee(G)$.
\end{thm}
Abusing notation, denote also by $C$ and $C\du$ the images in $\widetilde\DDD(G)$ of the elements $C,C\du$ in $\CCC(G)$. 

A \emph{second shuffle} in $\CCC(G)$ is an element of the form suggested by Theorem~\ref{thm:period_main}:
\begin{align*}
    &\sum_{\sigma\in\ol\Sigma_{r,s}}(-1)^{r+s-M_\sigma}C\du(w_{\sigma\inv(1)}|n_{\sigma\inv(1)},\dots,w_{\sigma\inv(M_\sigma)}|n_{\sigma\inv(M_\sigma)},w_0|n_0)\\
    &-C\du(w_1|n_1,\dots,w_r|n_r,w_{\cq{r+1,\dots,r+s,0}}|n_{\cq{r+1,\dots,r+s,0}})\\
    &-C\du(w_{r+1}|n_{r+1},\dots,w_{r+s}|n_{r+s},w_{\cq{1,\dots,r,0}}|n_{\cq{1,\dots,r,0}}),
\end{align*}
where
\[n_S=\sum_{i\in S}(n_i+1)-1,\quad w_S=\prod_{i\in S}w_i.\]
Define the \emph{quasidihedral Lie coalgebra}
\[\DDD(G)=\f{\widetilde\DDD(G)}{\text{second shuffles}}.\]
Then we prove:
\begin{thm}
    The second shuffles form a coideal in $\widetilde\DDD(G)$. The coproduct on $\widetilde\DDD(G)$ descends to a well-defined coproduct on $\DDD(G)$.
    \label{thm:main_qdih_const}
\end{thm}
Theorem~\ref{thm:main_qdih_const} provides us with a Lie coalgebra generated by sequences of elements of $G\cup\cq0$ that satisfies dihedral symmetry, scaling, and the two shuffle relations. 

Let $\CCC^\circ(G)$ the subspace of $\CCC(G)$ generated by elements $C(x_0,\dots,x_n)$ where not all $x_i$ are equal. It is a subcoalgebra, which we call the \emph{restricted cyclic Lie coalgebra}. The image of $\CCC^\circ(G)$ in $\DDD(G)$ is the \emph{restricted quasidihedral Lie coalgebra}, denoted  $\DDD^\circ(G)$.

The Hodge correlators satisfy cyclic symmetry, first shuffle, distribution, and scaling relations. Equivalently, the function $\Cor_\HHH\du$ factors through $\widetilde\DDD(\CC\du)$ and a map
\[C\du(w_0|n_0,\dots,w_k|n_k)\mapsto\Cor_\HHH\du(w_0|n_0,\dots,w_k|n_k).\]
An equivalent form of Theorem \ref{thm:period_main} is that, restricted to the set of arguments where not all $w_i=1$ or not all $n_i=0$, this function factors through the quotient $\DDD^\circ(\CC\du)$.

\subsubsection{Depth filtration} 

The Lie coalgebra $\DDD(G)$ is filtered by the \emph{depth}, where a generator has depth $d$ if it includes $d+1$ elements of $G$ (not counting 0s). Consider $\gr^D\DDD(G)$. In this coalgebra, the second shuffle relations lose their lower-depth terms.

\subsection{Relations for motivic correlators: Hodge realization}

We present the construction of motivic correlators of \cite{goncharov-hodge-correlators} and state our main result in this setting.

This section concerns the Hodge realization of motivic correlators. They are objects in the fundamental Lie coalgebra of the category of $\RR$-mixed Hodge structures, and are Hodge-theoretic upgrades of the Hodge correlator functions.

\subsubsection{Summary}

In \cite{goncharov-hodge-correlators}, given any collection of complex numbers $z_0,\dots,z_n$, the Hodge correlators $\Cor_\HHH(z_0,\dots,z_n)$ were upgraded to elements of the Tannakian Lie coalgebra $\Lie_\HT^\vee$ of the category of real mixed Hodge structures: \begin{equation}\Cor_\Hod(z_0,\dots,z_n)\in\Lie_\HT^\vee.\label{eqn:elem_cor_hod}\end{equation} Furthermore, if follows easily from the construction of the upgraded Hodge correlators (\ref{eqn:elem_cor_hod}) that they satisfy the dihedral and first shuffle relations, and that their coproduct in the coalgebra $\Lie_\HT^\vee$ is given precisely by the formula (\ref{eqn:cyc_coproduct}).

One of the main results of this paper is that the elements (\ref{eqn:elem_cor_hod}) satisfy the second shuffle relations. In other words, they provide a map of Lie coalgebras $\DDD^\circ(\CC\du)\to\Lie_\HT^\vee$.

\subsubsection{Hodge-theoretic setup}

Let $\MHTS_\RR$ of be the tensor category of $\RR$-mixed Hodge-Tate structures and $\HTS_\RR$ the category of $\RR$-pure Hodge-Tate structures. Every object of $\MHTS_\RR$ is filtered by weight, and  $\MHTS_\RR$ is generated by the simple objects $\RR(n)$, the pure Hodge-Tate structures of weight $-n$. The cohomology of a punctured projective line is a mixed Hodge-Tate structure, nontrivial in weights 0 and 1.

The \emph{Galois Lie algebra} of the category of mixed Hodge-Tate structures, $\Lie_\HT$, is the algebra of tensor derivations of the functor $\gr^W:\MHTS_\RR\to\HTS_\RR$. It is a graded Lie algebra in the category $\HTS_\RR$, and $\MHTS_\RR$ is equivalent to the category of graded $\Lie_\HT$-modules in $\HTS_\RR$. Let $\Lie_\HT^\vee$ be its graded dual. A canonical \emph{period map} \[p:\Lie_\HT^\vee\to\RR\] was defined in \cite{goncharov-hodge-correlators}.

Let $X=\PP^1(\CC)$, $S\subset X$ a finite set of punctures containing $\infty$, and $v_\infty=\f{-1}{z^2}\f{\del}{\del z}$ a distinguished tangent vector at $\infty$. The pronilpotent completion $\pi_1^\nil(X\sm(S\cup\cq\infty),v_\infty)$ of the fundamental group $\pi_1(X\sm S,\infty)$ carries a mixed Hodge-Tate structure, depending on $v_\infty$, and thus there is a map
\[\Lie_\HT\to\Der\pq{\gr^W\pi_1^\nil(X\sm S,v_\infty)}.\]

\subsubsection{Hodge correlator coalgebra}

\label{sec:hodge_cor_coalg}

The \emph{Hodge correlator coalgebra} is defined by \cite{goncharov-hodge-correlators} as 
\begin{equation*}
\CLie_{X,S,v_\infty}^\vee:=\f{T(\CC\bq{S\sm\cq\infty}^\vee)}{\text{relations}}\otimes H_2(X).
\end{equation*}
Note that $H_2(X)\cong\RR(1)$. If $[h]\in H_2(X)$ is the fundamental class, we write $x(1)$ for $x\otimes[h]$.

The relations are the following:
\begin{enumerate}[(1)]
\item Cyclic symmetry: $x_0\otimes\dots\otimes x_n=x_1\otimes\dots\otimes x_n\otimes x_0$.
\item (First) shuffle relations:
\begin{equation*}
\sum_{\sigma\in\Sigma_{p,q}}x_0\otimes x_{\sigma\inv(1)}\otimes\dots\otimes x_{\sigma\inv(p+q)}=0.
\end{equation*}
\item Take the quotient by the weight $-1$ elements $(x_0)$.
\end{enumerate}
There is a Lie coalgebra structure on $\CLie_{X,S,v_0}^\vee$, defined by the same formula as for the cyclic Lie coalgebra:
\begin{equation}
    \delta\pq{(x_0\otimes\dots\otimes x_n)(1)}=\sum_{\rm cyc}\sum_{i=1}^{n-1}\pq{(x_0\otimes x_1\otimes\dots\otimes x_i)(1)}\wedge\pq{(x_0\otimes x_{i+1}\otimes\dots\otimes x_n)(1)}.
    \label{eqn:clie_coproduct}
\end{equation}

An action of the graded dual Lie algebra $\CLie_{X,S,v_\infty}$ by derivations on $L_{X,S,s_0}$ was constructed by \cite{goncharov-hodge-correlators}. The action
\begin{equation*}
    \CLie_{X,S,v_\infty}\to\Der\pq{L_{X,S,v_\infty}}
\end{equation*} 
is injective. Its image consists of the \emph{special derivations} $\Der^S\pq{L_{X,S,v_\infty}}$, those which act by 0 on the loop around $\infty$ and preserve the conjugacy classes of all the loops $s\in S\sm\cq\infty$.

Dualizing this map composed with the action of $\Lie_\HT$, we get the \emph{Hodge correlator map} of Lie coalgebras:
\[\Cor_\Hod:\CLie^\vee_{X,S,v_{\infty}}\to\Lie_\HT^\vee.\]
We will also write $\Cor_\Hod(x_0,\dots,x_n)$ for $\Cor_\Hod\pq{(x_0\otimes\dots\otimes x_n)(1)}$, and similarly define \[\Cor_\Hod\du(w_0|n_0,\dots,w_k|n_k).\]

\subsubsection{Period map and Hodge correlator functions}

Recall that the Hodge correlator functions $\Cor_\HHH(x_0,\dots,x_n)$ satisfy cyclic symmetry and shuffle relations, so we may also denote by $\Cor_\HHH$ the function
\begin{align*}
    \Cor_\HHH:\CLie^\vee_{X,S,v_\infty}&\to\CC,\\
    (x_0\otimes\dots\otimes x_n)(1)&\mapsto\Cor_\HHH(x_0,\dots,x_n).
\end{align*}
The dual to the Hodge correlator $\Cor_\HHH:\CLie_{X,S,v_\infty}^\vee\to\CCC$, an element of $\CLie_{X,S,v_\infty}$, is called the \emph{Green operator} $\mathbf{G}_{v_\infty}$. It can be viewed as a special derivation of $\gr^W\pi_1^\nil(X\sm S,v_\infty)\otimes\CC$, and defines a real mixed Hodge structure on $\pi_1^\nil(X\sm S,v_\infty)$. An element $x\in\CLie^\vee_{X,S,v_\infty}$ provides a framing $\RR(n)\to\gr_{2n}^W\pi_1^\nil(X\sm S,v_\infty)$, and $\Cor_\Hod(x)$ is the element of $\Lie_\HT^\vee$ induced by this framing.

As made precise by a main result of \cite{goncharov-hodge-correlators}, $\Cor_\HHH$ factors through the Hodge correlator map to $\Lie_\HT^\vee$ and the period map $\Lie_\HT^\vee\to\CC$, and the resulting mixed Hodge structure on $\pi_1^\nil$ coincides with the standard one.

\begin{thm}[\cite{goncharov-hodge-correlators}, Theorem 1.12]
    \begin{enumerate}[(a)]
    \item Let $x\in\CLie^\vee_{X,S,v_\infty}$ be homogeneous of weight $n$. Then $\Cor_\HHH(x)=(2\pi i)^{-n}p(\Cor_\Hod(x))$, where $p$ is the canonical period map $\Lie_\HT^\vee\to\RR$.
    
    \item The mixed Hodge structure on $\pi_1^\nil$ determined by the dual Hodge correlator map coincides with the standard mixed Hodge structure on $\pi_1^\nil$.
    \end{enumerate}
    \label{thm:hc_main_point}
\end{thm}

\subsubsection{Second shuffle relations}

We state the version of the main result for the Hodge correlators, on the level of the map $\Cor_\Hod$.

\begin{thm}
    \begin{enumerate}[(a)]
    \item Restricted to the subspace of $\CLie_{X,S,v_\infty}^\vee$ generated by elements $(x_0\otimes\dots\otimes x_n)(1)$ with not all $x_i$ equal, the map $\Cor_\Hod$ factors through $\DDD^\circ(\CC\du)$.
    \item
    Suppose that $r,s>1$ and that not all $n_i=0$ or not all $w_i=1$. Then the Hodge correlators satisfy the relation:
    \begin{align*}
        &\sum_{\sigma\in\ol\Sigma_{r,s}}(-1)^{r+s-M_\sigma}\Cor_\Hod\du(w_{\sigma\inv(1)}|n_{\sigma\inv(1)},\dots,w_{\sigma\inv(M_\sigma)}|n_{\sigma\inv(M_\sigma)},w_0|n_0)\\
        &-\Cor_\Hod\du(w_1|n_1,\dots,w_r|n_r,w_{\cq{r+1,\dots,r+s,0}}|n_{\cq{r+1,\dots,r+s,0}})\\
        &-\Cor_\Hod\du(w_{r+1}|n_{r+1},\dots,w_{r+s}|n_{r+s},w_{\cq{1,\dots,r,0}}|n_{\cq{1,\dots,r,0}})&=0,
    \end{align*}
    where
    \[n_S=\sum_{i\in S}(n_i+1)-1,\quad w_S=\prod_{i\in S}w_i.\]
    \item The Hodge correlators satisfy all specializations of this relation as any subset of the $w_i$ $(1\leq i\leq n)$ approaches 0.
\end{enumerate}
    \label{thm:hodge_main}
\end{thm}

While Theorem~\ref{thm:period_main} was an equality between functions, Theorem~\ref{thm:hodge_main} is a relation in the fundamental Lie coalgebra of mixed Hodge-Tate structures. Theorem~\ref{thm:period_main} follows immediately from Theorem~\ref{thm:hodge_main} by applying the period map.

\subsection{Relations for motivic correlators over a number field}

We now state the most general version of the result by upgrading the constructions of the previous section from mixed Hodge structures to mixed motives over a number field.

\subsubsection{Motivic setup}

Let $F$ be a number field and $\MTM_F$ the category of mixed Tate motives over $F$. It is generated by objects $\QQ(n)=\QQ(1)^{\otimes n}$ for $n\in\ZZ$, where $\QQ(1)$ is the Tate motive, pure of weight $-1$. This induces a canonical weight filtration on objects of $\MTM_F$. There is a functor $\gr^W:\MTM_F\to\PM_F$, where $\PM_F$ is the category of pure motives over $F$.

The \emph{fundamental (motivic Tate) Lie algebra} $\Lie_\MTF$ is the algebra of tensor derivations of the functor $\gr^W$, a graded Lie algebra in the category $\PM_F$, and $\MTM_F$ is equivalent to the category of graded $\Lie_\MTF$-modules.

An embedding $\sigma:F\to\CC$ induces a \emph{realization functor} $r:\MTM_F\to\MHTS_\RR$ and a map $r:\Lie_\MTF^\vee\to\Lie_\HT^\vee$. 

Let $X=\PP^1$, $S\subset X(F)$ a finite set of punctures containing $\infty$, and $v_\infty$ the distinguished tangent vector at $\infty$. Deligne and Goncharov's \emph{motivic fundamental group} (\cite{deligne-goncharov}) $\pi_1^\Mot(X\sm S,v_\infty)_{\rm un}$ is a  prounipotent group scheme in the category $\MTM_F$. The Hodge realization of its Lie algebra is $\pi_1^\nil(X\sm S,v_\infty)$. As it is an object in $\MTM_F$, there is an action $\Lie_\MTF\to\Der\pq{\gr^W\pi_1^\Mot}$.

\subsubsection{Motivic correlator coalgebra}

The construction of the Hodge correlator coalgebra $\CLie^\vee_{X,S,v_\infty}$ can be upgraded to the motivic setting. 
The definition of the \emph{motivic correlator coalgebra} mimics that of its Hodge realization:
\begin{equation*}
\pq{\CLie^\Mot_{X,S,v_\infty}}^\vee:=\f{T\pq{(\QQ(1)^{S\sm\cq\infty})^\vee}}{\text{relations}}\otimes H_2(X),
\end{equation*}
a graded Lie coalgebra in the category of pure motives over $F$, where the relations imposed are the cyclic symmetry, first shuffles, and quotient by weight $0$. Then $\CLie^\Mot_{X,S,v_0}$ is isomorphic to the algebra of special derivations of $\gr^W\pi_1^\Mot(X-S,v_\infty)$, and there is a map
\begin{equation*}
\Cor_\Mot:\pq{\CLie^\Mot_{X,S,v_\infty}}^\vee\to\Lie_\MTF^\vee.
\end{equation*}
We will write $\Cor_\Mot(x_0,\dots,x_n)$ for $\Cor_\Mot(\pq{x_0\otimes\dots\otimes x_n}(1))$, and likewise $\Cor_\Mot(w_0|n_0,\dots,w_k|n_k)$.

Let us describe how motivic correlators are related to Hodge correlators. Fix an embeding $r:F\to\CC$. The Hodge realization provides coalgebra maps $\Lie_\MTF^\vee\to\Lie_\HT^\vee$ and \[r:\pq{\CLie^\Mot_{X,S,v_\infty}}^\vee\otimes\CC\to\CLie^\vee_{X,S,v_\infty}\otimes\CC,\] and thus a period map 
\[
    \Cor_\HHH\circ r:\pq{\CLie^\Mot_{X,S,v_\infty}}^\vee\otimes\CC\to\CLie_{X,S,v_\infty}^\vee\otimes\CC\to\CC.
\]

By Theorem~\ref{thm:hc_main_point}, it coincides with the composition 
\[
    \pq{\CLie^\Mot_{X,S,v_\infty}}^\vee\to\Lie_\Mot^\vee\to\Lie_\HT^\vee\to\CC.   
\]
We can summarize the objects and maps defined thus far as follows:
\[
\xymatrix{
    \Der^S(\gr^W\pi_1^\Mot(X\sm S,v_\infty))^\vee\ar@{-}[r]
    &(\CLie_{X,S,v_\infty}^\Mot)^\vee\ar[r]^{\quad\Cor_\Mot}\ar[d]^r
    &\Lie_\MTF^\vee\ar[d]^r
    \\
    \Der^S(\gr^W\pi_1^\nil(X\sm S,v_0))^\vee\ar@{-}[r]
    &(\CLie_{X,S,v_\infty}^\vee)\ar[r]^{\quad\Cor_\Hod}\ar[dr]_{(2\pi i)^w\Cor_\HHH}&\Lie_\HT^\vee\ar[d]^p
    \\
    &&\RR.
}    
\]

Under certain conditions, relations on motivic correlators hold can be proven by showing that they hold in the Hodge realization under any complex embedding. This is a key fact in the proof of the motivic upgrade of our relations on Hodge correlators:
\begin{lma}
    Let $X\sm S$ be a rational curve over $F$. Suppose $x\in\pq{\CLie_{X,S,v_\infty}^\Mot}^\vee$ has weight $>1$, $\delta\Cor_\Mot(x)=0$, and $\Cor_\HHH(r(x))=0$ for every embedding $r:F\to\CC$. Then $\Cor_\Mot(x)=0$.
    \label{lma:d0h0_rational}
\end{lma}

\subsubsection{Dependence on $S$} If $S\subseteq S'$, there is an induced inclusion $\iota:(\CLie_{X,S,v_0}^\Mot)^\vee\to(\CLie_{X,S',v_0}^\Mot)^\vee$.

The following diagram commutes:
\[
    \xymatrix{
        &(\CLie_{X,S,v_\infty}^\Mot)^\vee\ar[d]^\iota\ar[ddl]_{\Cor_\Mot}\ar[ddr]^{\Cor_\HHH\circ r}\\
        &(\CLie_{X,S',v_\infty}^\Mot)^\vee\ar[dl]^{\Cor_\Mot}\ar[dr]_{\Cor_\HHH\circ r}\\
        \Lie_\MTF^\vee\ar[rr]_{p\circ r}&&\CC.        
    }
\]
This allows us to write down elements of $(\CLie_{X,S,v_0}^\Mot)^\vee$ without explicitly specifying $S$.

\subsubsection{Second shuffle relations}

We are ready to state the most general version of the main result.

\begin{thm}
    Let $F$ be a number field.
    \begin{enumerate}[(a)]
    \item Restricted to the subspace of $\pq{\CLie_{X,S,v_\infty}^\Mot}^\vee$ generated by elements $(x_0\otimes\dots\otimes x_n)(1)$ with not all $x_i$ equal, the map $\Cor_\Mot$ factors through $\DDD^\circ(F^\times)$.
    \item
    Suppose that $r,s>1$ and that not all $n_i=0$ or not all $w_i=1$. Then the motivic correlators satisfy the same relation as in Theorem~\ref{thm:hodge_main}, with $\Cor_\Hod\du$ replaced by $\Cor_\Mot\du$.
    \item The motivic correlators satisfy all specializations of this relation as any subset of the $w_i$ $(1\leq i\leq n)$ approaches 0.
\end{enumerate}
    \label{thm:mot_main}
\end{thm}

\section{Background: Hodge and motivic correlators}

\subsection{Hodge realization of motivic correlators}

\subsubsection{Mixed Hodge theory}

We recall the relevant definitions from~\cite{deligne-mixed-hodge}. A real mixed Hodge structure consists of the following data:
\begin{enumerate}[(1)]
\item A real vector space $V$;
\item An increasing weight filtration $W_\bullet$ on $V$;
\item A decreasing Hodge filtration $F^\bullet$ on its complexification $V_\CC=V\otimes_\RR\CC$, with conjugate $\ol F^\bullet$,
\end{enumerate}
such that $F^\bullet$ and $\ol F^\bullet$ induce a pure real Hodge structure of weight $n$ on $\gr_n^WV_\CC$, i.e.,
\[
\gr^W_nV_\CC=\bigoplus_{p+q=n}F_{(n)}^p\cap\ol F_{(n)}^q,\quad
F_{(n)}^p=\f{F^p\cap(W_n)_\CC+(W_{n-1})_\CC}{(W_{n-1})_\CC},\quad\ol F_{(n)}^q=\dots.
\]

A mixed Hodge structure is a \emph{mixed Hodge-Tate structure} if $V_\CC^{p,q}=0$ for $p\neq q$. For the real mixed Hodge structures that are Tate, which are the ones we consider, the associated graded pure Hodge-Tate structures are trivial in odd weight. Therefore, we reindex the filtration by semiweight (so $\RR(1)$ has weight $-1$, rather than $-2$).

Mixed Hodge-Tate structures are iterated extensions of the one-dimensional pure mixed Hodge-Tate structures of weight $-n$, denoted $\RR(n)$. Equivalently, in the category $\MHS_\RR$ of real mixed Hodge-Tate structures, the subcategory of mixed Hodge-Tate structures $\MHTS_\RR$ is the full subcategory generated by the simple objects $\RR(n)$.

The map $\gr^W$ provides a fiber functor from mixed to pure real Hodge-Tate structures: \[\gr^W:\MHTS_\RR\to\HTS_\RR.\] The Tannakian reconstruction theorem implies that there is a graded Lie algebra $\Lie_\HT$ in the category $\HTS_\RR$ such that $\MHTS_\RR$ is equivalent to the category of finite-dimensional graded $\Lie_\HT$-modules in $\HTS_\RR$. Specifically, $\Lie_\HT=\Der^\otimes\gr^W$, the graded Lie algebra in $\HTS_\RR$ of tensor derivations of the functor $\gr^W$. That is, every mixed Hodge-Tate structure $X$ determines an action \[\Lie_\HT\to\Der\pq{\gr^WX}.\]  Let $\Lie_\HT^\vee$ be the graded dual of $\Lie_\HT$. 

The simple objects of the category $\HS_\RR$ are $\RR(n)$, and $\Lie_\HT$ is free on
\[\bigoplus_{n<0}\Ext^1_{\MHTS_\RR}(\RR(0),\RR(n))^\vee\otimes_{\End\RR(n)}\RR(n).\]

A \emph{framing} of a mixed Hodge-Tate structure $V$ of weight $n$ consists of a pair of morphisms $\RR(0)\to\gr_0^WV$, $\gr_{-2n}^WV\to\RR(n)$. The isomorphism classes of framed real mixed Hodge-Tate structures generate a Hopf algebra $\HHH_\bullet$, with the structure defined by \cite{bgsv}, which is canonically isomorphic to the dual to the universal enveloping algebra of $\Lie_\HT$. An element of $\Lie_\HT^\vee$ of weight $n$ is represented by a framed real mixed Hodge-Tate structure of weight $n$, modulo products in $\HHH$, that is, 
\begin{equation}
    \Lie_{\HT/B}^\vee\cong\f{\HHH}{\HHH_{>0}\cdot\HHH_{>0}}.
    \label{eqn:lie_quot_hopf} 
\end{equation}

The $\Ext^1(\RR(0),\RR(n))$ are trivial for $n\geq0$ and 1-dimensional for $n<0$, in which case \[\Ext^1(\RR(0),\RR(n))=(\RR(n)\otimes\CC)/\RR(n)=\RR(n)\otimes_\RR i\RR.\] 

According to \cite{goncharov-hodge-correlators}, a choice of  generators $n_w$ of $\Lie_\HT\otimes\CC$ satisfying $n_w=-\ol n_w$ amounts to a map
\begin{equation*}
\Lie_\HT^\vee\to\bigoplus_{n<0}\Ext^1_{\MHTS_\RR}(\RR(0),\RR(n))\otimes\RR(n)^\vee=\bigoplus_{n<0}\RR(n)\otimes_\RR i\RR,
\end{equation*}
and thus defines a canonical period map 
\begin{equation*}
    p:\Lie_\HT^\vee\to\RR.
\end{equation*}
Such generators were originally defined by Deligne for the larger category of real mixed Hodge structures (\cite{deligne-mixed-hodge}). However, we use the different set of generators proposed by Goncharov (\cite{goncharov-hodge-correlators}), the \emph{Green's operators} $G_{w}$. They have the property that, for Hodge structures varying over a base, the Griffiths transversality condition needed to define variations of Hodge structures is expressed by a Maurer-Cartan differential equation on the $G_{w}$, which is essential for the construction of Hodge correlators. Contrary to this, the differential equations for Deligne's generators are difficult to write.

A \emph{variation} of real mixed Hodge-Tate structures on a complex variety $B$ is a variation of the linear data of real mixed Hodge-Tate structure that satisfies the Griffiths transversality condition. Precisely, it is a real vector bundle with flat connection $(V,\nabla)$ with a weight filtration $W_\bullet$ on $V$ and a Hodge filtration $F^\bullet$ on $V\otimes_\RR\OOO_B$ such that $F$ and $V$ induce a real mixed Hodge-Tate structure over each point of $X$ and $\nabla^{0,1}F^p\subseteq F^{p-1}\otimes\Omega_B^1$.

A consequence of the transversality condition is that for $n>1$, $\Ext^1(\RR(0),\RR(n))$ is rigid in the category of variations of mixed Hodge-Tate structures over $B$: if the coproduct of a variation of Hodge-Tate structures of weight $w>1$ is 0, then the variation is isomorphic to a constant one.

\subsubsection{Pronilpotent fundamental group}

Let $X=\PP^1(\CC)$, $S\subset X$ a finite set of punctures containing $\infty$, and $v_\infty=\f{-1}{z^2}\f{d}{dz}$ a distringuished tangent vector at $\infty$. Let $\pi_1=\pi_1(X\sm S,\infty)$ be the classical fundamental group. The group algebra $A=\QQ[\pi_1]$ is a free group generated by loops around the points of $S\sm\cq{\infty}$. Let $\III=\ker(A\to\QQ)$ be the augmentation ideal. Then form a Hopf algebra
\begin{equation*}
A^\nil(X\sm S,v_\infty):=\lim_{\leftarrow}\pq{\dots\to A/\III^{n+1}\to A/\III^n\to\dots\to A},
\end{equation*}
with coproduct defined by $g\to g\otimes g$ for $g\in\pi_1$. The subset of primitive elements is denoted $\pi_1^\nil(X\sm S,v_\infty)$. It is actually a pronilpotent Lie algebra, the Mal'cev completion of $\pi_1$.

There is a canonical weight filtration on $H_1(X\sm S,\QQ)$, where the loops around punctures lie in weight $-1$. This induces a weight filtration $W$ on $A^\nil$, and we have
\begin{equation*}
\gr^WA^\nil(X\sm S,v_\infty)=T(\gr^WH_1(X\sm S,\QQ)).
\end{equation*}
Furthermore, let $L_{X,S,v_\infty}$ be the free Lie algebra generated by $\CC\bq{S\sm\cq\infty}$. Then there is a canonical isomorphism
\begin{equation*}
L_{X,S,v_\infty}\cong\gr^W\pi_1^\nil(X\sm S,v_\infty)\otimes\CC.
\end{equation*}

There is a real mixed Hodge-Tate structure on $\pi_1^\nil(X\sm S,v_\infty)\otimes\RR$, which depends on the choice of the tangent vector $v_\infty$, and thus an action $\Lie_\HT\to\Der\pq{\gr^W\pi_1^\nil(X\sm S,v_\infty)}$.

\subsubsection{Correlators in families}

The construction of the Hodge correlator coalgebra (\S\ref{sec:hodge_cor_coalg}) can be performed over a base. Let $X\to B$ be a smooth family of genus 0 curves. Generalizing from the case of $B$ a point, one simply replaces the punctures $S$ by \emph{nonintersecting} sections $s:B\to X$ and the tangential base point by a nonvanishing section $v_\infty:B\to T^1_{X/B}$ factoring through a distinguished section $s_\infty:B\to X$. This construction yields a family of coalgebras
\begin{equation}
    \pq{\CLie^\vee_{X_t,\cq{(s_i)_t},(v_\infty)_t}}_{t\in B}.
    \label{eqn:coalg_base}
\end{equation}
We will denote this coalgebra by $\CLie^\vee_{X/B,S,v_\infty}$ when the objects $X,S,v_\infty$ vary over $B$.

The Green's function $(2\pi i)\inv\log\aq{x-y}$, used in the definition of the Hodge correlator, becomes a distribution on $X\times_BX$ with logarithmic singularities along the relative divisors $x=s_\infty$, $y=s_\infty$, and $x=y$. As we explain below, the higher-weight correlators also determine smooth variations over the base. In particular, the period map $Cor_\HHH:\CLie^\vee_{X,S,v_\infty}\to\CC$ is upgraded to a map \[\Cor_\HHH:\CLie^\vee_{X/B,S,v_\infty}\to\AAA^0_B,\] and the map $\Cor_\Hod$ to a map \[\Cor_\Hod:\CLie^\vee_{X/B,S,v_\infty}\to\Lie_{\HT/B}^\vee\] to the fundamental Lie coalgebra of the category of variations of real mixed Hodge-Tate structures.

The case of specialization at intersecting sections, as well as degeneration to nodal curves, is related to the behavior of the Hodge structure on $\pi_1^\nil$ at the boundary of the moduli space of Riemann surfaces with $n$ punctures. We will examine this question in \S\ref{sec:nodal}.

As $X,S,v_0$ vary over the moduli space $\MMM_{0,n}'$ of Riemann surfaces of genus $0$ with $n$ distinct marked points and a tangential base point $v_0$, we get a family $\mathbf{V}$ of framed $\RR$-mixed Hodge structures on $\pi_1^\nil(X\sm S,s_0)$. Theorem~\ref{thm:hc_main_point} is generalized to the following.

\begin{thm}[\cite{goncharov-hodge-correlators}, Theorem 1.12]
\begin{enumerate}[(a)]
    \item There is a flat connection on $\mathbf{V}$ making it a variation of mixed Hodge structures over $\MMM_{0,n}'$. 
    \item This variation coincides with the standard variation of mixed Hodge structures on $\pi_1^\nil$.
\end{enumerate}
\label{thm:hc_main}
\end{thm}

A consequence of Theorem~\ref{thm:hc_main} is that the coalgebra structure on $\CLie_{X,S,s_0}^\vee$ should translate into differential equations on the periods over $\MMM_{0,n}'$. We now describe these equations.

Extend the period map $\Cor_\HHH$ to a map 
\begin{align*}
    \Cor_\HHH:\wedge^2\CLie_{X/B,S,v_\infty}&\to\AAA^1_B,\\
    C_1\wedge C_2&\mapsto\f{2w_2-1}{2(w-1)}\Cor_\HHH(C_2)\,d_B^\CC \Cor_\HHH(C_1)\\&\quad-\f{2w_1-1}{2(w-1)}\Cor_\HHH(C_1)\,d_B^\CC \Cor_\HHH(C_2),
\end{align*}
where $w_i=\wt C_i$ and $w=w_1+w_2$.
Then we have a diagram that commutes in weight $>1$:
\begin{equation}
\xymatrix{
    \CLie_{X/B,S,v_\infty}\ar[r]^\delta\ar[d]_{\Cor_\HHH}&\bigwedge_B^2\CLie_{X/B,S,v_\infty}\ar[d]^{\Cor_\HHH}\\\AAA^0_B\ar[r]^{d_B}&\AAA^1_B.
}
\label{eqn:diff_eq_p1}
\end{equation}

For the simplest example, consider the Hodge correlator $\Cor_\HHH(1,0,z)$ as $z$ varies over $\PP^1\sm\cq{0,1,\infty}$. Noting that $\Cor_\HHH(1,0)=0$, we have
\begin{align*}
d\Cor_\HHH(1,0,z)
&=\Cor_\HHH\pq{C(1,z)\wedge C(0,z)}\\
&=(2\pi i)^{-2}\pq{\log\aq{z}\,d^\CC\log\aq{z-1}-\log\aq{z-1}\,d^\CC\log\aq{z}}&=-\f12(2\pi i)^{-2}d\LLL_2(z),
\end{align*}
and indeed, by (\ref{eqn:h_w2_l2}), $\Cor_\HHH(1,0,z)=-\f12(2\pi i)^{-2}\LLL_2(z)$.

We emphasize that the sections have so far required to be nonintersecting. 
In \S\ref{sec:nodal} we will prove a specialization theorem, which allows to pass to the boundary of $\MMM_{0,n}'$. It will imply the statement about periods:
\begin{thm}
The Hodge correlators $\Cor_\HHH(z_0,\dots,z_n)$ are continuous on $\CC^{n+1}\sm\cq{z_0=\dots=z_n}$.
\label{thm:corrh_cts}
\end{thm}

\subsubsection{Distribution relations}
\label{sec:distr_rel}

The formula expressing how the Hodge correlators transform under endomorphisms of $X$ appears in \cite{goncharov-hodge-correlators}, Lemma 12.3. We translate this result to our setting, showing that it gives a relation of the form (\ref{eqn:distr_rel}). 

Consider the map $[l]:\PP^1\to\PP^1$, $z\mapsto z^l$ ($l\in\ZZ_{>0}$). Let $S'=[l]\inv(S)$. Then there is an induced map
\begin{align*}
    [l]\du:\CLie_{X,S,v_\infty}&\to\CLie_{X,S',v_\infty},\\
    (z_0\otimes\dots\otimes z_n)(1)&\mapsto\f1l(z_0'\otimes\dots\otimes z_n')(1),
\end{align*}
where
\[z_i'=\begin{cases}\sum_{y_i^l=z_i}(y_i)&z_i\neq0\\l\cdot(0)&z_i=0\end{cases}.\]
That is, each point is pulled back to the sum of its preimages, counted with multiplicity. The factor $\f1l$ comes from the degree of the induced map on $H_2(X)$.

Then the diagram commutes:
\[\xymatrix{
    \CLie_{X,S,v_\infty}\ar[r]^{[l]^*}\ar[dr]_{\Cor_\Hod}&\CLie_{X,S',v_\infty}\ar[d]^{\Cor_\Hod}\\&\Lie_\HT^\vee.
}\]
For example, in weight 1, we have
\[\Cor_\Hod(x,y)=\f12\pq{\Cor_\Hod(\sqrt x,\sqrt y)+\Cor_\Hod(\sqrt x,-\sqrt y)+\Cor_\Hod(-\sqrt x,\sqrt y)+\Cor_\Hod(-\sqrt x,-\sqrt y)},\] where a branch of the square root has been chosen. On the level of periods, this becomes the equality
\[\log\aq{x-y}=\f12\pq{\log\aq{\sqrt x-\sqrt y}+\log\aq{\sqrt x+\sqrt y}+\log\aq{-\sqrt x-\sqrt y}+\log\aq{-\sqrt x+\sqrt y}}.\] 

\subsection{Motivic correlators over a number field}

\subsubsection{Mixed motives}

Let $F$ be a number field. There is a semisimple abelian category $\PM_F$ of Grothendieck pure motives over $F$ and a functor $H:\mathbf{SmProj}_F\to\PM_F$ assigning to every smooth projective variety over $F$ the sum of its motivic cohomology objects:
\begin{equation*}
H(X)=\bigoplus_{i=0}^{2\dim(X)}H^i(X).
\end{equation*}
Every Weil cohomology theory $\mathbf{SmProj}_F\to\mathbf{Vect}$ factors through $H$ and a realization functor $\PM_F\tto r\mathbf{Vect}$:
\begin{align*}
    r_{\text{\rm Betti}}H(X)&=\bigoplus_{i=0}^{2\dim(X)}H^i_{\text{\rm Betti}}(X_\CC,\ZZ)\otimes\QQ&\text{(Betti)},\\
    r_\Hod H(X)&=\bigoplus_{i=0}^{2\dim(X)}\bigoplus_{p+q=i}H^{p,q}_{\text{\rm Hod}}(X_\CC,\RR)&\text{(real de Rham (Hodge))},\\
    r_\ell H(X)&=\bigoplus_{i=0}^{2\dim(X)}H^i_{\text{\rm\'et}}(X_{\ol F},\ZZ_\ell)\otimes\QQ_\ell&\text{($\ell$-adic \'etale)}. 
\end{align*}
This category is graded by the \emph{weight}, where the weight of $H^i(X)$ is $i$. There is an invertible Tate object $\QQ(1)$ of weight $-2$; we write $M(n)$ for the Tate twist $M\otimes\QQ(1)^{\otimes n}$. The various realization functors respect the weight. For example, for $X$ a variety over $F$ and a fixed embedding $F\to\CC$, the $r_\Hod H^i(X)$ carries a pure Hodge structure of weight $i$. For $X$ with good reduction modulo $p$, the Frobenius automorphism acts on $r_\ell H^i(X)$ with eigenvalues of norm $p^{i/2}$.

There is a conjectural category of \emph{mixed motives} $\MiM_F$ that should extend this construction to arbitrary varieties over $F$. The desired properties of $\MiM_F$ were conjectured by Beilinson \cite{beilinson-height-pairing}, see also Deligne \cite{deligne-motifs}. It is expected to be an abelian tensor category, in which every object has a canonical weight filtration $W_\bullet$. There should be a fiber functor $\gr^W:\MiM_F\to\PM_F$ such that $\gr^W_iX$ is pure of weight $i$. 

The Hodge realization of a mixed motive should be a mixed Hodge structure. Deligne \cite{deligne-hodge3} showed that for any complex variety $X$, there is a mixed Hodge structure on $\bigoplus H^i_\Hod(X,\RR)$. In this way, $\gr^W$ is a motivic lift of the associated graded functor from mixed to pure real Hodge structures: $\gr^W:\MHS_\RR\to\HS_\RR$.

The full tensor subcategory of $\MiM_F$ generated by $\QQ(1)$ is the category of \emph{mixed Tate motives} $\MTM_F$. Such a category with desirable properties has been constructed by \cite{deligne-goncharov}. If $X$ is a rational curve, then $H(X)$ is a mixed Tate motive. The simple objects of $\MTM_F$ are $\QQ(n)=\QQ(1)^{\otimes n}$, $n\in\ZZ$, and every object of $\MTM_F$ is an iterated extension of these objects. They satisfy
\begin{align*}
\Hom(\QQ(m),\QQ(n))&=0,\quad m<n;\\
\Ext^1(\QQ(0),\QQ(n))&=\begin{cases}0&n\leq 0\\K_{2n-1}(F)\otimes\QQ&n>0\end{cases},\\
\Ext^i(\QQ(0),\QQ(n))&=0,\quad i>1.
\end{align*}
The real Hodge realizations of mixed Tate motives are mixed Hodge-Tate structures. The images of the $\QQ(n)$, the real mixed Hodge-Tate structures $\RR(n)$ generate the subcategory $\MHTS_\RR$ in $\MHS_\RR$.

We will consider only the mixed Tate motives. As in the Hodge realization, the associated graded objects of the weight filtration are trivial in odd weight, so we reindex the filtration by semiweight (so $\QQ(1)$ has weight $-1$, rather than $-2$).

\subsubsection{Fundamental Lie algebra and period map}
 
Assume the mixed motivic formalism above. The Tannakian reconstruction theorem implies that there would be a negatively graded Lie algebra $\Lie_{\MTF}$ in the category $\PM_F$, the \emph{fundamental (motivic Tate) Lie algebra}, such that $\MTM_F$ is canonically equivalent to the category of finite-dimensional graded $\Lie_\MTF$-modules in $\PM_F$. That is, for any $X\in\MTM_F$, there is an action by derivations $\Lie_\MTF\to\Der(\gr^WX)$. We prefer to study its graded dual $\Lie_\MTF^\vee$.

This Lie coalgebra breaks into isotypical components over the isomorphism classes of simple Tate objects of $\PM_F$:
\begin{equation*}
\Lie_\MTF^\vee=\bigoplus_{[M]\in\PM_F}\pq{\Lie_\MTF^\vee}_M\boxtimes_{\End(M)}M\du.
\end{equation*}
As a consequence, the cohomology of $\Lie_\MTF^\vee$ can be expressed as Ext-groups in the category of mixed motives:
\begin{equation*}
H^i\pq{\Lie_\MTF^\vee\tto\delta\wedge^2\Lie_\MTF^\vee\tto\delta\dots}=\bigoplus_{[M]}\Ext^i_{\MTF}(\QQ(0),M)\boxtimes_{\End(M)}M\du.
\end{equation*}
For $F$ a number field, the $\Ext^i$ in $\MTM_F$ are trivial for $i>1$; equivalently, $\Lie_\Mot^\vee$ is free on the generators $\Ext^1_{\MTM_F}(\QQ(0),M)$.

Fix an embedding $r:F\to\CC$. The Hodge realization functor induces a Lie coalgebra morphism $r_\Hod:\Lie_\MTF^\vee\to\Lie_\HT^\vee$. This means that there is a period map $p\circ r_\Hod:\Lie_\MTF^\vee\to\RR$.

For every integer $n>0$, there is the Beilinson regulator map
\[
\reg:\Ext^1_{\MTM_F}(\QQ(0),\QQ(n))\to\bigoplus_{F\to\CC/\sigma}\Ext^1_{\MHTS_\RR}(\RR(0),\RR(n)),
\]
where $\sigma$ is complex conjugation. By Beilinson's theorem (\cite{beilinson-regulators}) it coincides for $n>1$ with the Borel regulator on $K_{2n-1}(F)$, i.e., the diagram commutes:
\[
\xymatrix{
    \Ext^1_{\MTM_F}(\QQ(0),\QQ(n))\ar[r]^{\reg}\ar[d]&\bigoplus_{F\to\CC/\sigma}\Ext^1_{\MHTS_\RR}(\RR(0),\RR(n))\ar[d]\\K_{2n-1}(F)\otimes\QQ\ar[r]^\reg&\RR^{d_n(F)}
},\]
\[d_n(F)=\begin{cases}r_1(F)+r_2(F)&\text{$n$ odd},\\r_2(F)&\text{$n$ even}\end{cases}.\]

Borel's theorem states that this regulator map -- the second row in the diagram -- is injective \cite{borel-stable-cohomology}. So there is an injective map on the first cohomology of the fundamental Lie coalgebras
\[
    \ker(\Lie_\MTF^\vee\tto\delta\wedge^2\Lie_\MTF^\vee)
    \to
    \bigoplus_{F\to\CC/\sigma}\ker(\Lie_{\HT}^\vee\tto\delta\wedge^2\Lie_{\HT}^\vee).
\]
In particular, we get the following basic theorem, which plays a crucial role in this paper:
\begin{thm}
    If $x\in\Lie_\MTF^\vee$ is of weight at least 2 with $\delta(x)=0$ and $p(r_\Hod(x))=0$ for every embedding $r:F\to\CC$, then $x=0$.
    \label{thm:d0h0m0}
\end{thm}

Specifically, we obtain Lemma~\ref{lma:d0h0_rational}:
\begin{lma*}
    Let $X\sm S$ be a rational curve over $F$. Suppose $x\in\pq{\CLie_{X,S,v_\infty}^\Mot}^\vee$ has weight $>1$, $\delta\Cor_\Mot(x)=0$, and $\Cor_\HHH(r(x))=0$ for every embedding $r:F\to\CC$. Then $\Cor_\Mot(x)=0$.
\end{lma*}
\begin{proof}
    $\Cor_\Mot(x)$ is an element of $\Lie_\MTF^\vee$ with coproduct 0. The canonical period of its Hodge realization in $\Lie_{\HT}^\vee$ coincides with the correlator period $\Cor_\HHH(r(x))=0$. By Theorem~\ref{thm:d0h0m0}, it is 0.
\end{proof}
This does not hold in weight 1. For example, choose $z$ to be an element of $F$ that is not a root of unity, but has norm 1 under every complex embedding (e.g., $F=\QQ(i)$ and $z=\f15\pq{3+4i}$). Then $((0)\otimes(z))(1)$ has coproduct 0 and period $\log\aq{\sigma(z)}=0$ under both of the embeddings $\QQ(i)\tto\sigma\CC$. However, the object $\Cor_\Mot(0,z)$ is not 0 as an element of $\Ext^1_{\MTM/F}(\QQ(0),\QQ(1))\cong F^\times\otimes\QQ$.

\subsubsection{Distribution relations}

Suppose $x_i\in F$ are such that $x^l-x_i$ splits in $F$ for all $i$. Then the distribution relations from \S\ref{sec:distr_rel} hold:
\[\Cor_\Mot(x_0,\dots,x_n)=\f1l\sum_{y_i^l=x_i}\Cor_\Mot(y_0,\dots,y_n),\]
where $y_i=0$ is taken with multiplicity $l$ if $x_i=0$.

\section{Construction of the quasidihedral Lie coalgebra}

\subsection{Definitions}

For an abelian group $G$, we defined the Lie coalgebra $\widetilde\DDD(G)$ as the quotient of the tensor algebra of $\QQ[G\cup\cq0]$ by cyclic symmetry, first shuffle, distribution, and scaling relations.

Recall Theorem~\ref{thm:main_qdih_const}:
\begin{thm*}
    The second shuffles form a coideal in $\widetilde\DDD(G)$. The coproduct on $\widetilde\DDD(G)$ descends to a well-defined coproduct on $\DDD(G)$.
\end{thm*}
The proof of this theorem is the goal of this section. 

The extra term in the scaling relation in weight 1, and the presence of terms of lower depth in the coproduct formula (\ref{eqn:cyc_coproduct}), makes the proof more difficult than that in \cite{goncharov-dihedral}'s construction of the dihedral Lie coalgebra. We find Theorem~\ref{thm:main_qdih_const} to be a small combinatorial miracle. Unfortunately, we do not know a simpler proof.

\subsubsection{Generating functions}

The second shuffle relations can be expressed in a compact form in terms of generating functions. This simplifies their proof.

We package the elements of $\widetilde\DDD(G)$ into a generating function as follows:
\begin{equation}
    \GF\pg{w_0,\dots,w_k}{t_0,\dots,t_k}:=\sum_{n_i\geq0} C^*(w_0|n_0,\dots,w_k|n_k)\prod_{i=0}^kt_i^{n_i}, \label{eqn:gf_def_simple}
\end{equation} where $\prod_{i=0}^kw_i=1$ and the $t_i$ are formal variables. 

\begin{figure}[ht]  
    \centering
\includegraphics{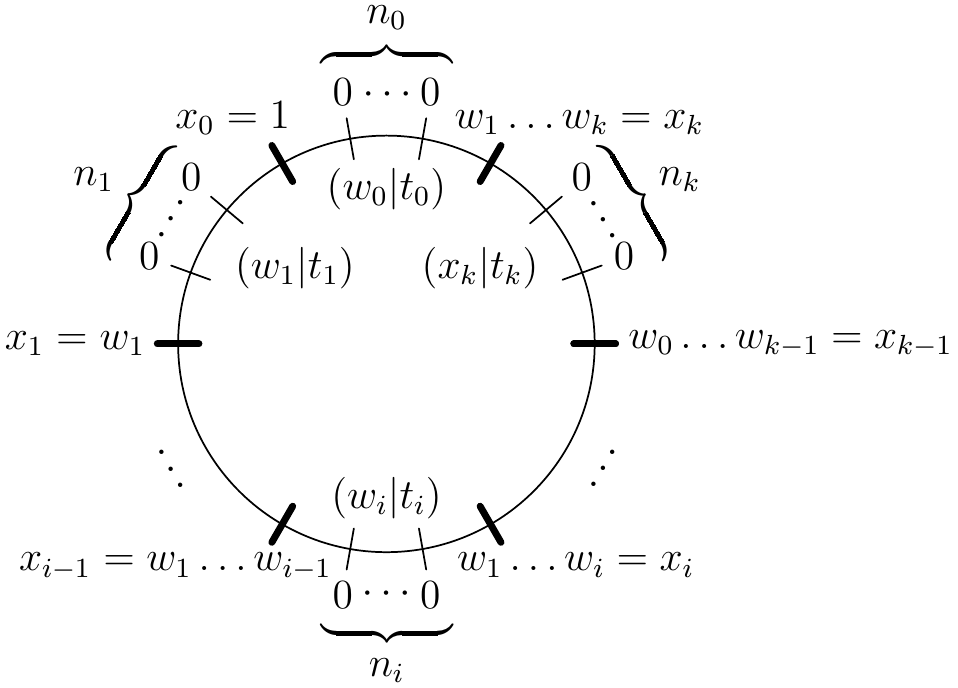}

\end{figure}

We allow multisets of variables to appear in place of the $t_i$: if $S_i=\cq{t_{i,1},\dots,t_{i,d_i}}$, then
\begin{align}
    \GF\pg{w_0,\dots,w_k}{S_1,\dots,S_k}
    &=\sum_{n_i\geq0}\sum_{\substack{n_{i,j}\geq0\\\sum_{j=1}^{d_i}n_{i,j}=n_i-d_i+1}}C\du(w_0|n_0,\dots,w_k|n_k)\prod_{i=0}^k\prod_{j=1}^{d_i}t_{i,j}^{n_{i,j}}\nonumber\\
    &=\sum_{n_{i,j}\geq0}C\du(x_0|N_0,\dots,x_k|N_k)\prod_{i=0}^k\prod_{j=1}^{d_i}t_{i,j}^{n_{i,j}}, \label{eqn:gf_multiset}
\end{align}
where in the last expression $N_i=n_{i,1}+1+n_{i,2}+1+\dots+1+n_{i,d_j}$. The corresponding operation on the correlator coefficients is combining adjacent segments of 0s, with additional 0s being inserted between them, such as \[\substack{\displaystyle(\text{$n_{i,1}$ 0s indexed by $t_{i,1}$})\\\displaystyle(\text{$n_{i,2}$ 0s indexed by $t_{i,2}$})}\to(\text{$n_{i,1}+1+n_{i,2}$ 0s indexed by $\cq{t_{i,1},t_{i,2}}$}).\]

\begin{figure}[ht]
    \centering
\includegraphics{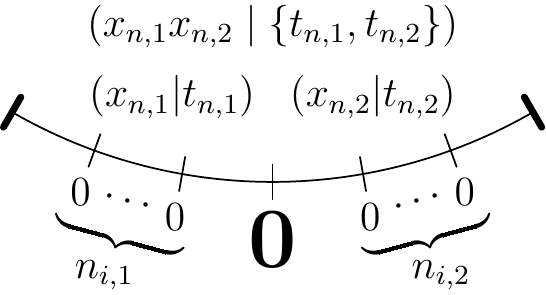}

\end{figure}

There is a useful identity
\begin{lma}
\begin{align}
    \GF\pg{\dots,w,\dots}{\dots,\cq t\sqcup T,\dots}&-\GF\pg{\dots,w,\dots}{\dots,\cq u\sqcup T,\dots}\nonumber\\&=(t-u)\GF\pg{\dots,w,\dots}{\dots,\cq{t,u}\sqcup T,\dots}.
\label{eqn:multi_identity}
\end{align}
\end{lma}
\begin{proof}
    Clear by comparing the coefficients of $t^ru^s$.

\end{proof}

Theorem~\ref{thm:main_qdih_const} can then be expressed in terms of the generating functions:
\begin{thm*}
    The subspace of $\widetilde\DDD(G)\bbq{t_1,\dots,t_k}$ generated by elements of the form
    \begin{align*}
        &\sum_{\sigma\in\ol\Sigma_{r,s}}(-1)^{r+s-M_\sigma}\GF\pg{w_{\sigma\inv(1)},\dots,w_{\sigma\inv(M_\sigma)},w_0}{S_{\sigma\inv(1)},\dots,S_{\sigma\inv(M_\sigma)},S_0}\\
        &-\GF\pg{w_1,\dots,w_r,w_{\cq{r+1,\dots,r+s,0}}}{S_1,\dots,S_r,S_{\cq{r+1,\dots,r+s,0}}}\\
        &-\GF\pg{w_{r+1},\dots,w_{r+s},w_{\cq{1,\dots,r,0}}}{S_{r+1},\dots,S_{r+s},S_{\cq{1,\dots,r,0}}}&=0,
    \end{align*}
    where
    \[S_I=\bigsqcup_{i\in I}S_i,\quad w_I=\prod_{i\in I}w_i\]
    forms a coideal.
    
\end{thm*}

\subsubsection{Coproduct}

Let us write down the formula defining the coproduct (\ref{eqn:cyc_coproduct}) in terms of the elements $C\du$.

\begin{lma}
    Let $C= C^*(w_0|n_0,\dots,w_k|n_k)$ and suppose $\text{\rm wt}(C)>2$. Then
\begin{align}
    \delta C\nonumber
    =&\sum_{\rm cyc}\bigg(\sum_{i=0}^{k}\sum_{n_i'+n_i''=n_i} C^*(\underbracket{w_i\dots w_k}|n_i',w_0|n_0,\dots,w_{i-1}|n_{i-1})\wedge\nonumber\\
    &\hspace{1.04in}\wedge C^*(w_{i+1}|n_{i+1},\dots,w_k|n_k,\underbracket{w_0w_1\dots w_i}|n_i'')\bigg)\label{eqn:mult_cop_1}\\
    +&\sum_{\rm cyc}\bigg(\sum_{i=1}^k\sum_{\substack{n_i'+n_i''=n_i\\n_0'+n_0''=n_0+1}} C^*(w_1|n_1,\dots,w_{i-1}|n_{i-1},\underbracket{w_i\dots w_k w_0}|n_i'+n_0'')\wedge\nonumber\\
    &\hspace{1.19in}\wedge C^*(\underbracket{w_0\dots w_i}|n_0'+n_i'',w_{i+1}|n_{i+1},\dots,w_k|n_k)\bigg)\label{eqn:mult_cop_2}\\
    +&\sum_{i=0}^kL_i\wedge C(0,w_i),\label{eqn:mult_cop_3}
\end{align}
    where
    \begin{equation}
        L_i=
        \begin{cases}
             C^*(w_0|n_0,\dots,w_i|n_i-1,\dots,w_k|n_k),&n_i>0,\\

            \substack{\displaystyle\quad C^*(w_0|n_0,\dots,\underbracket{w_{i-1}w_i}|n_{i-1},w_{i+1}|n_{i+1},\dots,w_k|n_k)\\\displaystyle+ C^*(w_0|n_0,\dots,w_{i-1}|n_{i-1},\underbracket{w_iw_{i+1}}|n_{i+1},\dots,w_k|n_k)},&n_i=0
        \end{cases},\label{eqn:log_cop}
    \end{equation}
    and the sums are taken over cyclic permutations of the indices $0,\dots,k$.
    
    If $\text{\rm wt}(C)=2$, this formula holds modulo terms of the form $C(0,a)\wedge C(0,b)$.
    \label{lma:mult_cop}
    \end{lma}
    \begin{proof}

    Classify the terms $C'\wedge C''$ of $\delta C$ by the common point of the two resulting parts $C'$ and $C''$. Let $x_i=w_1\dots w_i$ be the point counterclockwise from the segment $w_i$. Up to cyclic symmetry, any cut is either:
    \begin{enumerate}[(a)]
    \item a cut from $x_0$ to the segment $w_i$ (between $x_{i-1}$ and $x_i$) (Fig.~\ref{fig:cut}(a));
    \item cut from a 0 on the segment $w_0$ (between $x_k$ and $x_0$) to the segment $w_i$ (Fig.~\ref{fig:cut}(b)).
    \end{enumerate}
    
    \begin{figure}[ht]
        \centering 
        \begin{tabular}{cc}
    \includegraphics{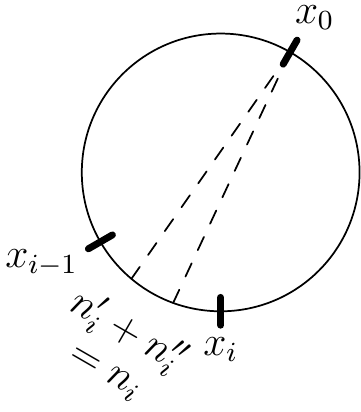}
        &
    \includegraphics{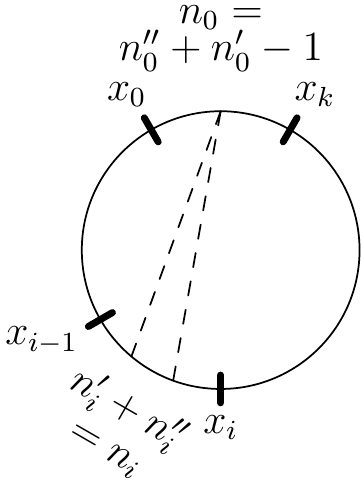} 
        \\\ \\
        (a)&(b)
    \end{tabular}
    \simplecap{fig:cut}{}
    \end{figure}
    
    We first write the terms arising from these cuts modulo elements of form $C(0,x)$.

    Case (a) contributes the terms (\ref{eqn:mult_cop_1}) and case (b) contributes the terms (\ref{eqn:mult_cop_2}), noting that modulo elements of the form $C(0,a)$ the $C\du$ have cyclic symmetry.
    
    Now we handle the terms (\ref{eqn:mult_cop_3}). Let $w=n_0+n_1+\dots+n_k+k$ be the weight. Consider the $(\text{weight $w-1$})\wedge(\text{weight 1})$ terms of the coproduct. 
    
    Such elements, of form $C_{w-1}\wedge C$, fall into two cases, depending on which point is present in $C$ but not in $C_{w-1}$.
    \begin{enumerate}[(1)]
    \item 0 on the segment $w_i$ (from $x_{i-1}$ to $x_i$).
    \item $x_i$.
    \end{enumerate}
    If $w>2$, the $C_{w-1}$ are invariant under scaling.
    If $w=2$, then the cyclic permutation of the arguments $w_0,\dots,w_{i-1}w_i,w_{i+1},\dots,w_k$ and $w_0,\dots,w_{i-1},w_iw_{i+1},\dots,w_k$ in (\ref{eqn:log_cop}) modifies those terms by an element of weight 1, so the expressions in (\ref{eqn:mult_cop_3}) are determined up to $(\text{weight 1})\wedge(\text{weight 1})$.
    
    In case (1), we have \[C_{w-1}=(w_0|n_0,\dots,x_i|n_i-1,\dots,x_k|n_k).\] The only nonzero terms that appear are $(C_{w-1}\wedge(-C(0,x_i))$ (cut clockwise of $x_i$) and $C_{w-1}\wedge C(0,x_{i-1})$ (cut counterclockwise of $x_i$).
    
    On the other hand, (\ref{eqn:mult_cop_1}) produces no terms for these two cuts (they correspond to to $i=1$ and $i=k$). Thus this case contributes the terms \[C_{w-1}\wedge(C(0,x_i)-C(0,x_{i-1}))=C_{w-1}\wedge C(0,w_i),\] which are the $n_i>0$ terms in (\ref{eqn:mult_cop_3}).
    
    In case (2), 
    \[
        C_{w-1}= C^*(x_0|n_0,\dots,x_ix_{i+1}|n_i+n_{i+1},\dots,x_k|n_k). 
    \]
    Let $C_1'$ and $C_1''$ be the elements formed by $x_i$ and the point clockwise and counterclockwise from $x_i$, respectively. Then the resulting terms are $-C_{w-1}\wedge C_1'$ and $C_{w-1}\wedge C_1''$.

    If $n_i=0$, then $C_1'=C(x_i,x_{i-1})=C\du(w_i|0,w_i\inv|0)+C(0,x_i)$, while (\ref{eqn:mult_cop_1}) contributes $C\du(w_i\inv|0,w_i|0)\wedge C_{w-1}$. Thus we get an added term \[-C_{w-1}\wedge(C(0,x_i)-C(0,w_i)).\]
    
    If $n_i\neq0$, then $C_1'=C(0,x_i)$, while (\ref{eqn:mult_cop_2}) contributes 0. Thus we get a term $-C_{w-1}\wedge C(0,x_i)$.
    
    Similarly, we get terms $C_{w-1}\wedge(C(0,x_i)+C(0,w_{i+1}))$ if $n_{i+1}=0$ and $C_{w-1}\wedge C(0,x_i)$ if $n_{i+1}>0$.

    Collecting terms, the total contribution from this case is $C_{w-1}\wedge(M_i'+M_i'')$, where
     \begin{equation}
     M_i'=\begin{cases}C(0,x_i)&n_i=0,\\0&n_i\neq0\end{cases},\quad
     M_i''=\begin{cases}C(0,x_{i+1})&n_{i+1}=0,\\0&n_{i+1}\neq0\end{cases}.
    \label{eqn:log_cop_mi}
     \end{equation}
     Reindexing, we get exactly the $n_i=0$ terms of (\ref{eqn:mult_cop_3}).
\end{proof}

We remark that if a cyclic permutation is applied to the arguments in (\ref{eqn:mult_cop_1}), so that it is written
\begin{align*}
    &C^*(w_0|n_0,\dots,w_{i-1}|n_{i-1},\underbracket{w_i\dots w_k}|n_i')\wedge\nonumber\\
    &\wedge C^*(\underbracket{w_0w_1\dots w_i}|n_i'',w_{i+1}|n_{i+1},\dots,w_k|n_k)\bigg),
\end{align*}
then the $n_i=0$ terms in (\ref{eqn:mult_cop_3}) disappear.
        
Then there is the following formula for the coproduct of generating functions:

\begin{lma}
    Suppose $k>2$ and let $X=\GF\pg{w_0,\dots,w_k}{t_0,\dots,t_k}$. Then
    \begin{align}
        \delta X=\sum_{\rm cyc}\bigg(\sum_{i=0}^{k}&\GF\pg{\underbracket{w_i\dots w_k},w_0,\dots,w_{i-1},}{t_i,t_1,\dots,t_{i-1}}\nonumber\\
        &\wedge \GF\pg{w_{i+1},\dots,w_k,\underbracket{w_0\dots w_i}}{t_{i+1},\dots,t_k,t_i}\bigg)\label{eqn:gf_cop_1}\\
        +\sum_{\rm cyc}\bigg(\sum_{i=1}^k&t_0\GF\pg{w_1,\dots,w_{i-1},\underbracket{w_i\dots w_kw_0}}{t_1,\dots,t_{i-1},\cq{t_i,t_0}}\nonumber\\
        &\wedge \GF\pg{\underbracket{w_0\dots w_i},w_{i+1},\dots,w_k}{\cq{t_i,t_0},t_{i+1},\dots,t_k}\label{eqn:gf_cop_2}\\
        +\sum_{i=1}^k&L_i\wedge\log w_i,\label{eqn:gf_cop_3}
    \end{align}
    where
    \begin{align}
        L_i=&t_i\GF\pg{w_0,\dots,w_k}{t_0,\dots,t_k}\label{eqn:log_cop_gf_lr}\\
        &+\GF\pg{w_0,\dots,\underbracket{w_{i-1}w_i},w_{i+1},\dots,w_k}{t_1,\dots,t_{i-1},t_{i+1},\dots,t_k}\label{eqn:log_cop_gf_l}\\
        &+\GF\pg{w_0,\dots,w_{i-1},\underbracket{w_iw_{i+1}},\dots,w_k}{t_1,\dots,t_{i-1},t_{i+1},\dots,t_k}.\label{eqn:log_cop_gf_r}
    \end{align}
    If $k=2$, this formula holds modulo terms of the form $C(0,a)\wedge C(0,b)$.
    \label{lma:cop_of_gf}
\end{lma}
\begin{proof}
Directly reinterpret Lemma~\ref{lma:mult_cop} via the definition (\ref{eqn:gf_def_simple}) by summing the expressions (\ref{eqn:mult_cop_1}), (\ref{eqn:mult_cop_2}), (\ref{eqn:mult_cop_3}) over choices of $\cq{n_i}_{i=0}^k$ taken with a monomial $\prod_it_i^{n_i}$.

The expressions (\ref{eqn:mult_cop_1}) and (\ref{eqn:mult_cop_2}) yield (\ref{eqn:gf_cop_1}) and (\ref{eqn:gf_cop_2}) in an obvious manner.

The $n_i>0$ cases in (\ref{eqn:log_cop}) give the terms with (\ref{eqn:log_cop_gf_lr}), and the $n_i=0$ cases give (\ref{eqn:log_cop_gf_l})-(\ref{eqn:log_cop_gf_r}).
\end{proof}

We also remark that if a cyclic permutation is applied to the arguments in (\ref{eqn:gf_cop_1}), so that it is written
\begin{align*}
    &\GF\pg{w_0,\dots,w_{i-1},\underbracket{w_i\dots w_k}}{t_1,\dots,t_{i-1},t_i}\nonumber\\
    &\wedge \GF\pg{\underbracket{w_0\dots w_i},w_{i+1},\dots,w_k}{t_i,t_{i+1},\dots,t_k}\bigg)
\end{align*}
then the terms (\ref{eqn:log_cop_gf_l}) and (\ref{eqn:log_cop_gf_r}) disappear.

\subsubsection{Dual generating function and homogeneity}

For a more complete analogy with the generating functions $L, L\du$ for multiple polylogarithms (\S\ref{sec:zeta_shuffle}), we define a dual generating function $\HGF$:
\begin{align}
    \HGF\pg{x_0,\dots,x_k}{t_0,\dots,t_k}&:=\sum_{n_i\geq0} C(x_0,\underbrace{0,\dots,0}_{n_0},x_1\dots,x_k,\underbrace{0,\dots,0}_{n_k})\prod_{i=0}^k(t_0+\dots+t_i)^{n_i},\label{eqn:hgf_def_simple}
\end{align} where the formal variables $t_i$ satisfy the relation $\sum_{i=0}^kt_i=0$. The pair of generating functions $\GF,\HGF$ resemble those used by \cite{goncharov-dihedral} in the definition of the dihedral Lie coalgebra. 

The duality is made clear by the following statement:
\begin{lma}
\begin{enumerate}[(a)]
\item The generating functions are related by
\begin{equation}
    \GF\pg{w_0,\dots,w_k}{t_0,\dots,t_k}=\HGF\pg{1,w_0,\dots,w_0\dots w_{k-1}}{t_0,t_1-t_0,\dots,t_k-t_{k-1}}.
\end{equation}
\item For $k>1$, the generating functions $\GF$ are homogeneous in the $t_i$ (invariant under a shift $t_i\mapsto t_i+t$), and the $\HGF$ are homogeneous in the $x_i$ (invariant under a shift $x_i\mapsto x_i\cdot x$). 
\item Both generating functions are invariant under cyclic permutation of the indices.
\end{enumerate}
\end{lma}
\begin{proof}
Part (a) is clear from the definitions.

For $\GF$, (c) is clear from the scaling relations imposed in $\widetilde\DDD(G)$. For $\HGF$, (b) is also immediate. Part (c) for $\HGF$ would follow easily from (a) and (b,c) for $\GF$, recalling that $t_1+\dots+t_k=0$.

The nontrivial part is (b) for $\GF$. We must show
\[
\GF\pg{w_0,\dots,w_k}{t_0+t,\dots,t_k+t}
=
\GF\pg{w_0,\dots,w_k}{t_0,\dots,t_k}.
\]
Consider the coefficient of $t^n\cdot\prod_it_i^{n_i}$ on each side. If $k=0$, the coefficients on both sides are equal. If $k>0$, the coefficient on the left side is precisely a first shuffle relation (where the $n$ 0s indexed by the variable $t$ are shuffled with all other points, with the point 1 remainining fixed), while the right side is 0.
\end{proof}

The first shuffle relation imposed in $\widetilde\DDD(G)$ can be expressed in terms of the $\HGF$:
\begin{lma}
    The generating functions $\HGF$ obey a shuffle relation for $r,s>1$:
\begin{equation}
    \sum_{\sigma\in\Sigma_{r,s}}\HGF\pg{x_{\sigma\inv(1)},\dots,x_{\sigma\inv(r+s)},x_0}{t_{\sigma\inv(1)},\dots,t_{\sigma\inv(r+s)},t_0}=0.
    \label{eqn:first_shuffle_corr}
\end{equation}
    \label{lma:first_shuf_gf}
\end{lma}
\begin{proof}
    Similar to the previous lemma. It follows from the shuffle relation on the coefficients, where we fix $x_0$ and shuffle the $x_1,\dots,x_r$ and the zeros indexed by $t_1,\dots,t_r$ with the other points.
\end{proof}

\subsection{Proof of Theorem~\ref{thm:main_qdih_const}}

\subsubsection{Summary of the proof}

The proof of the Theorem~\ref{thm:main_qdih_const} will be by induction on the depth of the second shuffles.

Define
\begin{align*}
\QSh^{r,s}(w_1|S_1,\dots,w_n|S_n,w_0|S_0)=\\
=\sum_{\sigma\in\ol\Sigma_{r,s}}(-1)^{r+s-M_\sigma}\GF\pg{w_{\sigma\inv(1)},\dots,w_{\sigma\inv(M_\sigma)},w_0}{S_{\sigma\inv(1)},\dots,S_{\sigma\inv(M_\sigma)},S_0},
\end{align*}
where $w_i\in G$ with $\prod_iw_i=1$, and
\begin{align}
    \ol\QSh^{r,s}(w_1|S_1,\dots,w_n|S_n,w_0|S_0)&=
    \QSh^{r,s}(w_1|S_1,\dots,w_n|S_n,w_0|S_0)\label{eqn:def_olqsh_q}\\
    &-\GF\pg{w_1,\dots,w_r,w_{\cq{r+1,\dots,r+s,0}}}{S_1,\dots,S_r,S_{\cq{r+1,\dots,r+s,0}}}\label{eqn:def_olqsh_ra}\\
    &-\GF\pg{w_{r+1},\dots,w_{r+s},w_{\cq{1,\dots,r,0}}}{S_{r+1},\dots,S_{r+s},S_{\cq{1,\dots,r,0}}}.\label{eqn:def_olqsh_rb}
\end{align}
We must show that the elements $\ol\QSh$ form a coideal, i.e., their coproducts vanish modulo other elements of this form. 

To make the notation more transparent, when $r$ and $s$ are fixed, we will relabel 
\begin{align*}
T_1,\dots,T_r&=S_1,\dots,S_r,\\
U_1,\dots,U_s&=S_{r+1},\dots,S_{r+s},\\
V&=S_0,\\
a_1,\dots,a_r&=w_1,\dots,w_r\\
b_1,\dots,b_s&=w_1,\dots,w_{r+s},\\
c&=w_0,
\end{align*}
so that we consider elements
\[\ol\QSh^{r,s}(a_1|T_1,\dots,a_r|T_r,b_1|U_1,\dots,b_s|U_s,c|V).\]

The main steps will be the following:
\begin{enumerate}[\it Step 1.]
\item[\it Step 0.] Fix the $a_i$ and $b_j$. Show that it suffices to assume $\aq{T_i}=\aq{U_i}=\aq{V}=1$. Denote the three terms (\ref{eqn:def_olqsh_q}), (\ref{eqn:def_olqsh_ra}), (\ref{eqn:def_olqsh_rb}) by $Q$, $R_A$, and $R_B$, respectively. 
\item Show that $\delta(Q-R_A-R_B)$ is zero modulo shuffle relations of lower depth and elements of the form $C(0,x)$ (Lemma~\ref{lma:pf_lwsr}).
\begin{enumerate}[(a)]
\item Group the terms of $\delta Q$ according to a combinatorial classification and reduce them using shuffle relations of lower depth (Lemma~\ref{lma:pf_dq}).
\item Group the terms of $\delta(R_A)$ and $\delta(R_B)$ in the same way and show that they coincide with the terms found in (a) (Lemma~\ref{lma:pf_dr}).
\end{enumerate}
\item Show that the (weight 1)$\wedge$(weight $\geq 1$) component of $\delta(Q-R_A-R_B)$ is 0, modulo shuffle relations of lower depth (Lemma~\ref{lma:pf_log}).

\end{enumerate}

Throughout the proof, in a term $\GF\pg{w_1,\dots,w_k,w_0}{s_1,\dots,s_k,s_0}$ appearing in the definition of $\ol\QSh$, call the segment $\pg{w_0}{s_0}$ the \emph{distinguished segment} (i.e., $\pg cv$ in (\ref{eqn:def_olqsh_q}) and the collapsed segments in (\ref{eqn:def_olqsh_ra}) and (\ref{eqn:def_olqsh_rb})). In the following lemmas, we will always use the following classification of terms of the coproduct of a generating function (see Fig.~\ref{fig:cut_class}). 
\begin{enumerate}[(1)]
\item Terms $g\wedge h$ where one of the parts $g$ or $h$ contains the distinguished segment (i.e., the distinguished segment is not cut). In this case, we always write the term in the form $\pm g\wedge h$, where $g$ contains the distinguished segment.
\begin{enumerate}[(a)]
\item Cut from a point $x_i$ to the segment $\pg{w_j}{s_j}$ ($0\leq i<j\leq k$).
\item Cut from a point $x_j$ to the segment $\pg{w_{i+1}}{s_{i+1}}$ ($0\leq i<j\leq k$).
\item Cut from a 0 on the segment $\pg{x_{i+1}}{t_{s+1}}$ to the segment $\pg{w_j}{t_j}$ ($0\leq i<j\leq k$).
\item Cut from a 0 on the segment $\pg{x_j}{s_j}$ to the segment $\pg{w_{i+1}}{t_{i+1}}$ ($0\leq i<j\leq k$).
\end{enumerate}
\item Terms $g\wedge h$ where the distinguished segment is cut. In this case, we always write $\pm g\wedge h$, where $g$ contains the point $x_0$ and $h$ the point $x_k$.
\begin{enumerate}[(a)]
\item Cut from a point $x_i$ to the distinguished segment.
\item Cut from a 0 on the segment $\pg{w_i}{s_i}$ to the distinguished segment ($0<i<k$).
\item Cut from a 0 on the distinguished segment to the segment $\pg{s_i}{t_i}$ ($0<i<k$).
\end{enumerate}
\end{enumerate}

\begin{figure}[t]
    \centering 
\begin{tabular}{cccc}
    \hspace{-0.5in}
\includegraphics{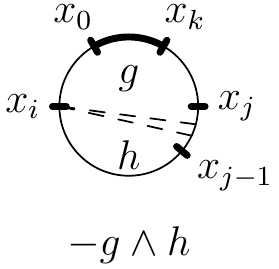}
    &
\includegraphics{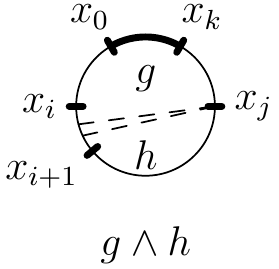}
    &
\includegraphics{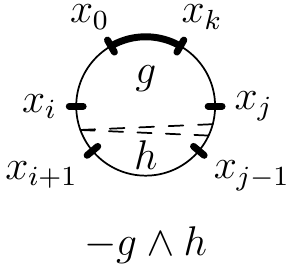}
    &
\includegraphics{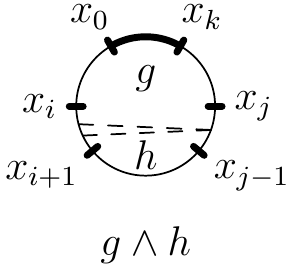}
    \\\ \\
    (1a)&(1b)&(1c)&(1d)
\end{tabular}

\vspace{1em}

\begin{tabular}{ccc}
\includegraphics{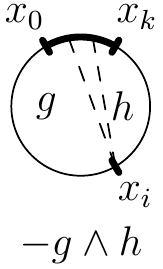}
    &
\includegraphics{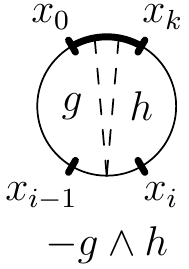}
    &
\includegraphics{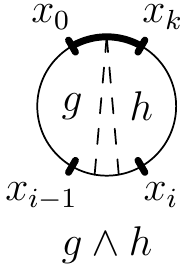}
\\\ \\
    (2a)&(2b)&(2c)
\end{tabular}
\simplecap{fig:cut_class}{}
\end{figure}

\subsubsection{Step 0}

As stated in Step 0 above, we fix $m>0$ and $n>0$, the $a_i$, $b_j$, $c$, and the $T_i$, $U_j$, $V$, and let $Q,R_A,R_B$ be the three terms of the expression defining $\ol\QSh$: (\ref{eqn:def_olqsh_q}), (\ref{eqn:def_olqsh_ra}), and (\ref{eqn:def_olqsh_rb}), respectively.

We may assume $T_i=\cq{t_i}$, $U_j=\cq{u_j}$, and $V=\cq{v}$, by the following:
\begin{lma}[Step 0]
\label{lma:pf_red_multi}
The shuffle relations for $\aq{T_i}=\aq{U_j}=\aq{V}=1$ imply the shuffle relations for general index sets.
\end{lma}
\begin{proof}
Obvious by induction using (\ref{eqn:multi_identity}).
\end{proof}

\begin{lma}[Step 1]
\label{lma:pf_lwsr}
Modulo shuffle relations of lower depth and elements $C(0,x)$, $\delta(Q-R_A-R_B)=0$.
\end{lma}

\begin{lma}[Step 2]
     Modulo lower-depth shuffle relations and terms $C(0,x)\wedge C(0,y)$, 
    \begin{equation}
        \delta(Q-R_A-R_B)=\bq{\sum_{i=1}^mC(0,a_i)(t_i-v)+\sum_{j=1}^nC(0,b_j)(u_j-v)}\wedge(Q-R_A-R_B).
        \label{eqn:step2}
    \end{equation}
    \label{lma:pf_log}
\end{lma}

\subsubsection{Proof of Step 1}

\begin{lmacomp}[Step 1(a)] 
    \label{lma:pf_dq}
    Modulo shuffle relations of lower depth and elements $C(0,x)$, 
    $\delta Q$ is given by the sum of expressions
    
    (\ref{eqn:dq_2_total2})-(\ref{eqn:dq_sum_ffgg})
    below.
\end{lmacomp}
    
Group all terms of $\delta Q$ by the type of cut as defined in the outline above. Some computational lemmas will simplify the contributions to $\delta Q$ coming from the cuts of each type. The contribution of cuts (1a/b/c/d) is computed in Lemma~\ref{lma:pf_dq_type1}, and cuts (2a/b/c) are dealt with in Lemma~\ref{lma:pf_dq_type2}.

\begin{lmacomp} 
\label{lma:pf_dq_type1}
The contribution of cuts of type (1a/b/c/d)
to $\delta Q$, modulo shuffle relations of lower depth and elements $C(0,x)$,
is given by (\ref{eqn:dq_1_total}) below.
\end{lmacomp}

The cuts of types (1a) and (1b) contribute terms of the form (\ref{eqn:gf_cop_1}), while cuts of types (1c) and (1d) contribute terms of the form (\ref{eqn:gf_cop_2}) below.

Consider the upper parts $g$ of terms $\pm g\wedge h$ as shown in Fig.~\ref{fig:cut_class}; by cyclic invariance modulo $C(0,x)$ we may write 
\[g=\GF\pg{w_1,\dots,w_l,c}{S_1,\dots,S_l,V}.\] Let $\pg{w_p}{S_p}$ be the new segment arising from the cut (that is, the bracketed segment in (\ref{eqn:gf_cop_1}) or (\ref{eqn:gf_cop_2})). 

We say that $a_i$ appears in $g$ if either the segment $\pg{a_i}{t_i}$ or some $\pg{a_ib_j}{\cq{t_i,u_j}}$ is present in $g$ as one of the $\pg{w_l}{S_l}$ ($l\neq i$), and similarly for $b_j$. Then the set of segments that \emph{do not} appear in $g$ (``appear below g'') is determined by the $w_1,\dots,\widehat{w_p},\dots,w_l$ and consists of consecutively indexed elements $a_i$ and $b_j$, i.e., $a_{i_0},\dots,a_{i_1}$ and $b_{j_0},\dots,b_{j_1}$, where by convention $i_0=i_1+1$ if no $a_i$ appear, and likewise for $j_0,j_1$. 

Group the terms $g\wedge h$ by the sequence of segments $w_1,\dots,\widehat{w_p},\dots,w_l$. To shorten notation, write 
\newcommand{\tg}[1]{\widetilde{g}\pq{#1}}
\[\tg{S}=\GF\pg{w_1,\dots,w_p,\dots,w_l,c}{S_1,\dots,S_p=S,\dots,S_l,V}.\]

There are three cases:
\begin{enumerate}[(1)]
\item $i_1-i_0>0$ and $j_1-j_0>0$: at least two $a_i$ and two $b_j$ appear below $g$ (Lemma-Computation~\ref{lma:dq_1_case1}).
\item $i_1-i_0=-1$ or $j_1-j_0=-1$: only $a_i$'s or only $b_j$'s appear below $g$ (Lemma-Computation~\ref{lma:dq_1_case2}).
\item $i_1-i_0=0$ or $j_1-j_0=0$: only one $a_i$ or only one $b_j$ appear below $g$ (Lemma-Computation~\ref{lma:dq_1_case3}).
\end{enumerate}
We compute the contribution of each case in the next three lemmas.

\begin{lmacomp}
\textit{Case 1} ($i_1-i_0>0$ and $j_1-j_0>0$) contributes 0 to $\delta Q$. 
\label{lma:dq_1_case1}
\end{lmacomp}
\begin{proof}
Consider a term $g\wedge h$ coming from a cut in Case 1. 

Let $i_0'\geq i_0$ be minimal such that $a_{i_0'}$ appears in $h$, and $i_1'\leq i_1$ be maximal such that $a_{i_1'}$ appears in $h$. Define $j_0',j_1'$ in the analogous way. For example, for cuts of type (1a), $i_0'=i_0$; for cuts of type (1c),
\[
    i_0'=\begin{cases} 
        i_0&\text{if $\pg{w}{S}$ is $\pg{b_{i_0}}{u_{i_0}}$}\\
        i_0+1&\text{if $\pg{w}{S}$ is $\pg{a_{i_0}}{t_{i_0}}$ or $\pg{a_{i_0}b_{i_0}}{\cq{t_{i_0},u_{j_0}}}$}
    \end{cases}
,\]
where $\pg wS$ is the segment that contains the vertex of the cut.

Notice that $i_0'-i_0\leq 1$ and $j_0-j_0'\leq 1$, and $i_1-i_0>0$ implies $i_1'-i_0'\geq-1$.

Group all terms of $\delta Q$ coming from Case 1 by the type of cut and by $i_0',j_0',i_1',j_1'$. These groups can be expressed in terms of 
\begin{align*}
\tg{S_1}\wedge\QSh\pg{&a_{i_0'},\dots,a_{i_1'},b_{j_0'},\dots,b_{j_1'},\pq{a_{i_0'}\dots a_{i_1'}\cdot b_{j_0'}\dots b_{j_1'}}\inv}{\\&t_{i_0'},\dots,t_{i_1'},u_{j_0'},\dots,u_{j_1'},S_2}
\end{align*}
for some $S_1,S_2$. Indeed, the arrangements of segments that may occur in the lower part of the cut, given $i_0,j_0$ and $i_1',j_1'$, are precisely the quasishuffles. Applying the lower-weight shuffle relations, this expression becomes
\begin{align}
\tg{S_1}&\wedge\pq{\GF\pg{a_{i_0'},\dots,a_{i_1'},\pq{a_{i_0'}\dots a_{i_1'}}\inv}{t_{i_0'},\dots,t_{i_1'},\cq{u_{j_0'},\dots,u_{j_1'}}\sqcup S_2}}\nonumber\\
+\,\tg{S_1}&\wedge\pq{\GF\pg{b_{j_0'},\dots,b_{j_1'},\pq{b_{j_0'}\dots b_{j_1'}}\inv}{u_{j_0'},\dots,u_{j_1'},\cq{t_{i_0'},\dots,t_{i_1'}}\sqcup S_2}}\label{eqn:fafb}.
\end{align}
\newcommand{\tff}{\widetilde{f}}
Fix $i_0',i_1',j_0',j_1'$, and introduce the notation 
\begin{align*}
    \tff_A(i_0',i_1',S)&=\GF\pg{a_{i_0'},\dots,a_{i_1'},\pq{a_{i_0'}\dots a_{i_1'}}\inv}{t_{i_0'},\dots,t_{i_1'},\cq{u_{j_0+1},\dots,u_{j_1-1}}\sqcup S},\\
    \tff_B(j_0',j_1',S)&=\GF\pg{b_{j_0'},\dots,b_{j_1'},\pq{b_{j_0'}\dots b_{j_1'}}\inv}{u_{j_0'},\dots,u_{j_1'},\cq{t_{i_0+1},\dots,t_{i_1-1}}\sqcup S}.
\end{align*}
The expressions in (\ref{eqn:fafb}) can be rewritten with $\tff_A$ and $\tff_B$.

Now let us collect these terms coming from different cuts and show that they yield 0. By symmetry, it suffices to show this for three kinds of terms $\tff_A(i_0',i_1',j_0',j_1',S_2)$: where $i_0'=i_0$ and $i_1'=i_1$; where $i_0'=i_0$ and $i_1'=i_1-1$; and where $i_0'=i_0+1$ and $i_1'=i_1-1$. 

Look at the terms with $i_0'=i_0$ and $i_1'=i_1$ (all $a_i$ that are not in $g$ are in $f_A$). They arise from cuts (1a) and (1b) where the cut segment is $b_{i_0+1}$ or $b_{i_1-1}$ and from cuts (1c) and (1d) where the cut segment and the segment containing the vertex are $b_{i_0+1}$ and $b_{i_1-1}$, or vice versa. These cases give:
\begin{align*}
-\tg{u_{j_1}}&\wedge\tff_A(i_0,i_1,\cq{u_{j_0}}\sqcup\cq{u_{j_1}}),\\
\tg{u_{j_0}}&\wedge\tff_A(i_0,i_1,\cq{u_{j_1}}\sqcup\cq{u_{j_0}}),\\
(u_{j_1}-u_{j_0})\tg{\cq{u_{j_0},u_{j_1}}}&\wedge\tff_A(i_0,i_1,\cq{u_{j_0},u_{j_1}}),
\end{align*}
the sum of which is 0 by (\ref{eqn:multi_identity}).

The terms with $i_0'=i_0$ and $i_1'=i_1-1$ (all $a_i$ that are not in $g$, except the last, are in $f_A$) come from three sources: 
\begin{itemize}
    \item cuts of type (1a) where the cut segment $x_2$ is either $a_{i_1}$ or $a_{i_1}b_{j_1}$; 
    \item cuts of type (1c) and (1d) where the segment $x_1$ containing the vertex and the segment $x_2$ that is cut are $b_{j_0}$ and $a_{i_1}$, or vice versa;
    \item cuts of type (1c) and (1d) where the segment $x_1$ containing the vertex and the segment $x_2$ that is cut are $b_{j_0}$ and $a_{i_1}b_{j_1}$, or vice versa.
\end{itemize}
A similar computation shows their total contribution is 0.

Finally, consider terms with $i_0'=i_0+1$ and $i_1'=i_1-1$ (all $a_i$ not in $g$ except the first and last are in $f_A$). They arise from cuts of type (1c) and (1d), where the segment $x_1$ is either $a_{i_0}$ or $a_{i_0}b_{j_0}$ and the segment $x_2$ is either $a_{i_1}$ or $a_{i_1}b_{j_1}$, yielding four cases: \[(x_1,x_2)=(a_{i_0},a_{j_1}),(a_{i_0}b_{j_0},a_{j_1}),(a_{i_0},a_{j_1}b_{j_1}),(a_{i_0}b_{j_0},a_{i_1}b_{j_1}).\] The sum of their contributions is also 0.

\end{proof}

\newcommand{\thhh}[1]{\widetilde{f}_B(#1)}
\begin{lmacomp}
\label{lma:dq_1_case2}
The contribution of Case 2 ($i_1-i_0=-1$) to the $\delta Q$ is given by the sum of expressions (\ref{eqn:dq_1a_allb})-(\ref{eqn:dq_1d_allb}) below.
\end{lmacomp}
\begin{proof}
Suppose that $i_1-i_0=-1$.  The cuts of types (1a), (1b), (1c), and (1d) contribute
\begin{align}
-\tg{u_{j_1}}&\wedge\thhh{j_0,j_1-1,u_{j_1}}\label{eqn:dq_1a_allb},\\
\tg{u_{j_0}}&\wedge\thhh{j_0+1,j_1,u_{j_0}}\label{eqn:dq_1b_allb},\\
-u_{j_0}\tg{\cq{u_{j_0},u_{j_1}}}&\wedge\thhh{j_0+1,j_1-1,\cq{u_{j_0},u_{j_1}}}\label{eqn:dq_1c_allb},\\
u_{j_1}\tg{\cq{u_{j_0},u_{j_1}}}&\wedge\thhh{j_0+1,j_1-1,\cq{u_{j_0},u_{j_1}}},\label{eqn:dq_1d_allb}
\end{align}

respectively.
\end{proof}
By symmetry, analogous expressions will result if $j_1-j_0=-1$. 

\begin{lmacomp}
\label{lma:dq_1_case3}
The contribution of Case 3 ($i_1-i_0=0$) to the $\delta Q$ is given by the sum of expressions (\ref{eqn:onea_contrib3}) and (\ref{eqn:onea_contrib}) below.
\end{lmacomp}
\begin{proof}
Suppose $i_1-i_0=0$, so only one segment $a_i$ occurs below $g$. 

If $j_1-j_0=0$, then it is easy to see that only cuts of type (1a) and (1b) contribute nonzero terms, and that the (1a) terms cancel with the (1b) terms. So assume $j_1-j_0>0$.

The cuts of type (1a) fall into three classes depending on which segment is cut: (i) $a_{i_0}$, (ii) $b_{j_1}$, or (iii) $a_{i_0}b_{j_1}$. The first two contribute
\begin{align}
-\tg{t_{i_0}}&\wedge\thhh{j_0,j_1,t_{i_0}}\label{eqn:dq_1a_onea_ca},\\
-\tg{u_{j_1}}&\wedge\QSh^{1,j_1-j_0}\pg{a_{i_0},b_{j_0},\dots,b_{j_1-1},\pq{a_{i_0}b_{j_0}\dots b_{j_1-1}}\inv}{\nonumber\\&\quad\quad\quad\quad\quad\quad t_{i_0},u_{j_0},\dots,u_{j_1-1},u_{j_1}}\nonumber\\
\equiv-\tg{u_{j_1}}&\wedge\biggl(\thhh{j_0,j_1-1,\cq{t_{i_0},u_{j_1}}}
\nonumber\\&\quad\quad+\GF\pg{a_{i_0},a_{i_0}\inv}{t_{i_0},\cq{u_{j_0},\dots,u_{j_1}}}\biggr)\label{eqn:dq_1a_onea_cb}, 
\end{align}

respectively, where we have used that the sequences that may occur in the lower part of the cut are precisely the shuffles of $a_i$ and $b_j$ appearing below $g$, except the cut segment $b_{j_1}$. Finally, the third class gives
\begin{align}
\f{1}{t_{i_0}-u_{j_1}}\biggl(
\tg{t_{i_0}}&\wedge\thhh{j_0,j_1-1,t_{i_0}}\nonumber\\
-\tg{u_{j_1}}&\wedge\thhh{j_0,j_1-1,u_{j_1}}\biggr)\label{eqn:dq_1a_onea_cab},
\end{align}

where we have applied (\ref{eqn:multi_identity}) to break the generating functions with $\cq{t_{i_0},u_{j_1}}$ into ones with only $t_{i_0}$ or $u_{j_1}$.

The cuts of type (1c) fall into five classes, depending on the segment where the vertex of the cut lies and the segment that is cut: (i) vertex on $a_{i_0}$ and $b_{j_1}$ is cut, (ii) vertex on $b_{j_0}$ and $b_{j_1}$ is cut, (iii) vertex on $b_{j_0}$ and $a_{i_0}$ is cut, (iv) vertex on $b_{j_0}$ and $a_{i_0}b_{j_1}$ is cut, (v) vertex on $a_{i_0}b_{j_0}$ and $b_{j_1}$ is cut. They contribute the following terms:
\begin{align}
-t_{i_0}\tg{\cq{t_{i_0},u_{j_1}}}&\wedge\thhh{j_0,j_1-1,\cq{t_{i_0},u_{j_1}}}\label{eqn:dq_1c_onea_zacb},\\
-u_{j_0}\tg{\cq{u_{j_0},u_{j_1}}}&\wedge \QSh^{1,j_1-j_0-1}\pg{a_{i_0},b_{j_0+1},\dots,b_{j_1-1},\pq{a_{i_0}b_{j_0+1}\dots b_{j_1-1}}\inv}{\nonumber\\&\quad\quad\quad\quad\quad\quad\quad t_{i_0},u_{j_0+1},\dots,u_{j_1-1},\cq{u_{j_0},u_{j_1}}}\nonumber\\
\equiv-u_{j_0}\tg{\cq{u_{j_0},u_{j_1}}}&\wedge\biggl(\thhh{j_0+1,j_1-1,\cq{t_{i_0},u_{j_0},u_{j_1}}}\nonumber\\&\quad\quad+\GF\pg{a_{i_0},a_{i_0}\inv}{t_{i_0},\cq{u_{j_0},\dots,u_{j_1}}}\biggr)
\label{eqn:dq_1c_onea_zbcb},\\
-u_{j_0}\tg{\cq{u_{j_0},t_{i_0}}}&\wedge\thhh{j_0+1,j_1,\cq{u_{j_0},t_{j_0}}}\label{eqn:dq_1c_onea_zbca},\\
\f{1}{t_{i_0}-u_{j_1}}\biggl(u_{j_0}\tg{\cq{u_{j_0},t_{i_0}}}&\wedge\thhh{j_0+1,j_1-1,\cq{u_{j_0},t_{i_0}}}\nonumber\\
-u_{j_0}\tg{\cq{u_{j_0},u_{j_1}}}&\wedge\thhh{j_0+1,j_1-1,\cq{u_{j_0};u_{j_1}}}\label{eqn:dq_1c_onea_zbcab}\biggr),\\
\f{1}{t_{i_0}-u_{j_0}}\biggl(t_{i_0}\tg{\cq{t_{i_0},u_{j_1}}}&\wedge\thhh{j_0+1,j_1-1,\cq{t_{i_0},u_{j_1}}}\nonumber\\
-u_{j_0}\tg{\cq{u_{j_0},u_{j_1}}}&\wedge\thhh{j_0+1,j_1-1,\cq{u_{j_0},u_{j_1}}}\label{eqn:dq_1c_onea_zabcb}\biggr).
\end{align}

The cuts (1b) and (1d) contribute antisymmetric terms, i.e., $u_{j_0}$ and $u_{j_1}$ are exchanged and $\thhh{j_0+d_0,j_1-d_1,S}$ becomes $-\thhh{j_0+d_1,j_1-d_0,S}$.
The entire contribution of case 3 is then the 

 symmetrization of the sum of expressions (\ref{eqn:dq_1a_onea_ca})-(\ref{eqn:dq_1c_onea_zabcb}).

The expression (\ref{eqn:dq_1a_onea_ca}) with its symmetrization cancels to 0.

The remaining terms form the contribution of Case 3, and are simplified to
\begin{align}
    \tg{\cq{t_{i_0},u_{j_1}}}&\wedge\thhh{j_0,j_1-1,u_{j_1}}
    -\tg{\cq{t_{i_0},u_{j_0}}}&\wedge\thhh{j_0+1,j_1,u_{j_0}}
    \label{eqn:onea_contrib3}\\
    -(u_{j_1}-u_{j_0})\tg{\cq{t,u_{j_0},u_{j_1}}}&\wedge\thhh{j_0+1,j_1-1,\cq{u_{j_0},u_{j_1}}}.\label{eqn:onea_contrib}
\end{align}

Analogous expressions result if $j_1-j_0=0$.
\end{proof}

\begin{proof}[Proof of Lemma~\ref{lma:pf_dq_type1}]
Let us collect the terms obtained from cases 2 and 3: (\ref{eqn:dq_1a_allb})-(\ref{eqn:dq_1d_allb}), (\ref{eqn:onea_contrib3}), and (\ref{eqn:onea_contrib}).

Consider first the expressions of the form $\thhh{j_0,j_1-1,u_{j_1}}$, arising from (\ref{eqn:dq_1a_allb}) and (\ref{eqn:onea_contrib3}). (The notation $\thhh$, which by definition depends on $i_0$ and $i_1$, is unambiguous here since no $a_i$ appear in the expression for $\thhh$ when $i_1-i_0\leq0$.) We claim that for fixed $j_0$ and $j_1$, the sum of these terms over all $g$ is precisely
\begin{align}
-\QSh^{m,n-(j_1-j_0)}\pg{&a_1,\dots,a_r,b_1,\dots,\underbracket{b_{j_0}\dots b_{j_1}},\dots,b_s,c}{\nonumber\\&t_1,\dots,t_r,u_1,\dots,u_{j_1},\dots,u_s,v}
\quad\quad\quad\wedge\,\thhh{j_0,j_1-1,u_{j_1}}.
\label{eqn:qsh_red_1a}
\end{align}
Indeed, the term that appears on the left side for a fixed $g$ is $-\tg{u_{j_1}}$ if $i_1-i_0=-1$ and $\tg{\cq{t_{i_0},u_{j_1}}}$ if $i_1-i_0=0$ 

. The quasishuffles for which the underlined segment collides with no $a_i$ provides the terms with $i_1-i_0=-1$, while the quasishuffles for which the underlined segment collides with some $a_i$ provide the terms with $i_0=i_1=i$.

In a similar way, the expressions with $u_{j_0}\thhh{j_0+1,j_1-1,\cq{u_{j_0},u_{j_1}}}$, coming from (\ref{eqn:onea_contrib}) and (\ref{eqn:dq_1c_allb}), yield
\begin{align}
-u_0\QSh^{m,n-(j_1-j_0)}\pg{&a_1,\dots,a_r,b_1,\dots,\underbracket{b_{j_0}\dots b_{j_1}},\dots,b_s,c}{\nonumber\\&t_1,\dots,t_r,u_1,\dots,\cq{u_{j_0},u_{j_1}},\dots,u_s,v}
\quad\quad\quad\wedge\,\thhh{j_0+1,j_1-1,\cq{u_{j_0},u_{j_1}}}.
\label{eqn:qsh_red_1c}
\end{align}
The expressions with $\thhh{j_0+1,j_1,u_{j_1}}$ and $u_{j_1}\thhh{j_0+1,j_1-1,\cq{u_{j_0},u_{j_1}}}$ give the antisymmetric terms.

Applying the shuffle relations of lower depth to (\ref{eqn:qsh_red_1a}) and (\ref{eqn:qsh_red_1c}), we get the total contribution of cases 2 and 3 for fixed $j_0$ and $j_1$:
\begin{align} 
&-\biggl(\GF\pg{a_1,\dots,a_r,b_1\dots b_s\cdot c}{t_1,\dots,t_r,\cq{u_1,\dots,u_{j_0-1},u_{j_1},u_{j_1+1},\dots,u_s,v}}\nonumber\\
&\quad\quad+\GF\pg{b_1,\dots,b_{j_0}\dots b_{j_1},\dots,b_s,a_1\dots a_r\cdot c}{u_1,\dots,u_{j_1},\dots,u_s,\cq{t_1,\dots,t_r,v}}\biggr)\nonumber\\
&\quad\wedge\,\GF\pg{b_{j_0},b_{j_0+1},\dots,b_{j_1-1},\pq{b_{j_0}\dots b_{j_1-1}}\inv}{u_{j_0},u_{j_0+1},\dots,u_{j_1-1},u_{j_1}}\label{eqn:qsh_red_a}\\
&+\biggl(\GF\pg{a_1,\dots,a_r,b_1\dots b_s\cdot c}{t_1,\dots,t_r,\cq{u_1,\dots,u_{j_0-1},u_{j_0},u_{j_1+1},\dots,u_s,v}}\nonumber\\
&\quad\quad+\GF\pg{b_1,\dots,b_{j_0}\dots b_{j_1},\dots,b_s,a_1\dots a_r\cdot c}{u_1,\dots,u_{j_1},\dots,u_s,\cq{t_1,\dots,t_r,v}}\biggr)\nonumber\\
&\quad\wedge\,\GF\pg{b_{j_0+1},\dots,b_{j_1-1},b_{j_1},\pq{b_{j_0+1}\dots b_{j_1}}\inv}{u_{j_0+1},\dots,u_{j_1-1},u_{j_1},u_{j_0}}\label{eqn:qsh_red_b}\\
&+(u_{j_1}-u_{j_0})\biggl(\GF\pg{a_1,\dots,a_r,b_1\dots b_s\cdot c}{t_1,\dots,t_r,\cq{u_1,\dots,u_{j_0-1},u_{j_0},u_{j_1},u_{j_1+1},\dots,u_s,v}}\nonumber\\
&\quad\quad+\GF\pg{b_1,\dots,b_{j_0}\dots b_{j_1},\dots,b_s,a_1\dots a_r\cdot c}{u_1,\dots,\cq{u_{j_0},u_{j_1}},\dots,u_s,\cq{t_1,\dots,t_r,v}}\biggr)\nonumber\\
&\quad\wedge\,\GF\pg{b_{j_0+1},\dots,b_{j_1-1},\pq{b_{j_1}\dots b_{j_1-1}}\inv}{u_{j_0+1},\dots,u_{j_1-1},\cq{u_{j_0},u_{j_1}}}\label{eqn:qsh_red_cd}.
\end{align} 
Notice that this expression does not depend on $i_0,i_1$, and all but one of the segments in each generating function $f$ depends only on the $a_i$ or only on the $b_j$.

Reindexing leads to cancelation of all terms $f(a_1,\dots,a_r,\dots)$ except the term in (\ref{eqn:qsh_red_a}) where $j_0=1$ and the term in (\ref{eqn:qsh_red_b}) where $j_1=n$. That is, if $j_0\neq 1$ and $j_1\neq n$, then this expression becomes
\begin{align} 
    F(j_0,j_1):=
    &-\biggl(\GF\pg{b_1,\dots,b_{j_0}\dots b_{j_1},\dots,b_s,a_1\dots a_r\cdot c}{u_1,\dots,u_{j_1},\dots,u_s,\cq{t_1,\dots,t_r,v}}\biggr)\nonumber\\
    &\quad\wedge\,\GF\pg{b_{j_0},b_{j_0+1},\dots,b_{j_1-1},\pq{b_{j_0}\dots b_{j_1-1}}\inv}{u_{j_0},u_{j_0+1},\dots,u_{j_1-1},u_{j_1}}\label{eqn:qsh_red_a_noa}\\
    &+\biggl(\GF\pg{b_1,\dots,b_{j_0}\dots b_{j_1},\dots,b_s,a_1\dots a_r\cdot c}{u_1,\dots,u_{j_0},\dots,u_s,\cq{t_1,\dots,t_r,v}}\biggr)\nonumber\\
    &\quad\wedge\,\GF\pg{b_{j_0+1},\dots,b_{j_1-1},b_{j_1},\pq{b_{j_0+1}\dots b_{j_1}}\inv}{u_{j_0+1},\dots,u_{j_1-1},u_{j_1},u_{j_0}}\label{eqn:qsh_red_b_soa}\\
    &+(u_{j_1}-u_{j_0})\biggl(\GF\pg{b_1,\dots,b_{j_0}\dots b_{j_1},\dots,b_s,a_1\dots a_r\cdot c}{u_1,\dots,\cq{u_{j_0},u_{j_1}},\dots,u_s,\cq{t_1,\dots,t_r,v}}\biggr)\nonumber\\
    &\quad\wedge\,\GF\pg{b_{j_0+1},\dots,b_{j_1-1},\pq{b_{j_1}\dots b_{j_1-1}}\inv}{u_{j_0+1},\dots,u_{j_1-1},\cq{u_{j_0},u_{j_1}}}\label{eqn:qsh_red_cd_noa}.
\end{align} 
If $j_0=1$ or $j_1=s$, the following terms remain, respectively:
\begin{align}
F_L(j_1):=
&-\GF\pg{a_1,\dots,a_r,b_1\dots b_s\cdot c}{t_1,\dots,t_r,\cq{u_{j_1},u_{j_1+1},\dots,u_s,v}}\nonumber\\
&\quad\wedge\,\GF\pg{b_1,\dots,b_{j_1-1},\pq{b_1\dots b_{j_1-1}}\inv}{u_1,\dots,u_{j_1-1},u_{j_1}}\label{eqn:qsh_red_a_l},\\
F_R(j_0):=\,
&\GF\pg{a_1,\dots,a_r,b_1\dots b_s\cdot c}{t_1,\dots,t_r,\cq{u_1,\dots,u_{j_0},v}}\nonumber\\
&\wedge\,\GF\pg{b_{j_0+1},\dots,b_s,\pq{b_{j_0+1}\dots b_s}\inv}{u_{j_0+1},\dots,u_s,u_{j_0}}.\label{eqn:qsh_red_b_r}
\end{align}

Identical terms $G(i_0,i_1)$, $G_L(i_1)$, $G_R(i_0)$ with the $\pg{a_i}{t_i}$ and $\pg{b_j}{u_j}$ exchanged appear in the cases $j_1-j_0=\text{$0$ or $-1$}$. 

So the total contribution of cuts of type 1 is 
\begin{align}
\sum_{\substack{1\leq j_0,j_1\leq s\\j_1-j_0\geq-1}}F(j_0,j_1)&+\sum_{\substack{1\leq i_0,i_1\leq r\\i_1-i_0\geq-1}}G(i_0,i_1)\nonumber\\
+\sum_{1\leq j_1\leq s}F_L(j_1)+\sum_{1\leq j_0\leq s}F_R(j_0)
&+\sum_{1\leq i_1\leq r}G_L(i_1)+\sum_{1\leq i_0\leq r}G_R(i_0)
\label{eqn:dq_1_total}
\end{align}
finishing the computation.
\end{proof}

\begin{lmacomp}
\label{lma:pf_dq_type2}
Computation of cuts of type (2).
\end{lmacomp}
\begin{proof}
A cut of type (2a/b/c) divides the circle into a left part $g$ and a right part $h$ (see Fig.~\ref{fig:cut_class}). Let $i_0$ be maximal such that $a_{i_0}$ appears in $g$ and $i_1$ minimal such that $a_{i_1}$ appears in $h$, with $i_0=-1$ or $i_1=m+1$ if the corresponding segments do not appear. Define $j_0,j_1$ in the same manner, for the $b_j$.

Let
\begin{align*}
f_L(i_0,j_0,S)&=\GF\pg{a_1,\dots,a_{i_0},\pq{a_1\dots a_{i_0}}\inv}{t_1,\dots,t_{i_0},\cq{u_1,\dots,u_{j_0}}\sqcup S},\\
f_R(i_1,j_1,S)&=\GF\pg{a_{i_1},\dots,a_r,\pq{a_{i_1}\dots a_r}\inv}{t_{i_1},\dots,t_r,\cq{u_{j_1},\dots,u_m}\sqcup S},
\end{align*}
and define $g_L(i_0,j_0,S)$ and $g_R(i_1,j_1,S)$ in a similar way for the $\pg{b_j}{u_j}$.
(As usual, one interprets these expressions as 0 if the index set is empty.) Also let
\begin{align*}
q_L(i_0,j_0,S)&=\QSh^{i_0,j_0}\pg{a_1,\dots,a_{i_0},b_1,\dots,b_{j_0},\pq{a_1\dots a_{i_0}\cdot b_1\dots b_{j_0}}\inv}{t_1,\dots,t_{i_0},u_1,\dots,u_{j_0},S},\\
&=f_L(i_0,j_0,S)+g_L(i_0,j_0,S)\\
q_R(i_1,j_1,S)&=\QSh^{r-i_1+1,s-j_1+1}\pg{a_{i_1},\dots,a_r,b_{j_1},\dots,b_s,\pq{a_{i_1}\dots a_r\cdot b_{j_1}\dots b_s}\inv}{t_{i_1},\dots,t_r,u_{j_1},\dots,u_s,S}\\
&=f_R(i_1,j_1,S)+g_R(i_1,j_1,S).
\end{align*}

Consider cuts (2a) for fixed $i_0,i_1,j_0,j_1$. For such cuts, \[i_1-i_0=j_1-j_0=1,
\quad-1\leq i_0\leq r,\quad-1\leq j_0\leq s.\] The $g$ that occur in the resulting terms are exactly the quasishuffles of $\cq{a_i:i\leq i_0}$ and $\cq{b_j:j\leq j_0}$. The analogous statement holds for $h$. The contribution of cuts (2a) is
\begin{align}
-q_L(i_0,j_0,v)\wedge q_R(i_0+1,j_0+1,v).\label{eqn:dq_2a}
\end{align}

Now look at cuts (2b) and (2c). The non-distinguished segment containing the vertex or the cut is either $a_{i_0+1}$ ($i_0<r$), $b_{j_0+1}$ ($j_0<s$), or $a_{i_0+1}b_{j_0+1}$ ($i_0<r,j_0<s$). 
The terms coming from the sum of (2b) and (2c) are, for these three cases respectively,
\begin{align}
(v-t_{i_0+1})q_L(i_0,j_0,\cq{t_{i_0+1},v})&\wedge q_R(i_0+2,j_0+1,\cq{t_{i_0+1},v}),\label{eqn:dq_2bc_a}\\
(v-u_{j_0+1})q_L(i_0,j_0,\cq{u_{j_0+1},v})&\wedge q_R(i_0+1,j_0+2,\cq{u_{j_0+1},v}),\label{eqn:dq_2bc_b}\\
\f{-1}{t_{i_0+1}-u_{j_0+1}}\biggl(
(v-t_{i_0+1})q_L(i_0,j_0,\cq{t_{i_0+1},v})&\wedge q_R(i_0+2,j_0+2,\cq{t_{i_0+1},v})\nonumber\\
-(v-u_{i_0+1})q_L(i_0,j_0,\cq{u_{j_0+1},v})&\wedge q_R(i_0+2,j_0+2,\cq{u_{j_0+1},v})\biggr).\label{eqn:dq_2bc_ab}
\end{align}

Let us assemble the terms of the form $f_L\wedge g_R$ and $g_L\wedge g_R$ coming from application of the shuffle relations to the $q_L$ and $q_R$. (The terms $g_L\wedge f_R$ and $f_L\wedge f_R$ are symmetrical.)

The terms $f_L\wedge g_R$, for $-1\leq i_0<r$ and $-1\leq j_0<s$, are:
\begin{align}
-f_L(i_0,j_0,v)&\wedge g_R(i_0+1,j_0+1,v)\nonumber\\
+(v-t_{i_0+1})f_L(i_0,j_0,\cq{t_{i_0+1},v})&\wedge g_R(i_0+2,j_0+1,\cq{t_{i_0+1},v}),\nonumber\\
+(v-u_{j_0+1})f_L(i_0,j_0,\cq{u_{j_0+1},v})&\wedge g_R(i_0+1,j_0+2,\cq{u_{j_0+1},v}),\nonumber\\
-\f{1}{t_{i_0+1}-u_{j_0+1}}\biggl(
(v-t_{i_0+1})f_L(i_0,j_0,\cq{t_{i_0+1},v})&\wedge g_R(i_0+2,j_0+2,\cq{t_{i_0+1},v})\nonumber\\
-(v-u_{i_0+1})f_L(i_0,j_0,\cq{u_{j_0+1},v})&\wedge g_R(i_0+2,j_0+2,\cq{u_{j_0+1},v})\biggr)\nonumber\\
=f_L(i_0,j_0,t_{i_0+1})&\wedge g_R(i_0+1,j_0+1,v)\nonumber\\
-f_L(i_0,j_0+1,t_{i_0+1})&\wedge g_R(i_0+1,j_0+2,v)\nonumber.
\end{align}
Summing this over $j_0$ leaves
\begin{align}
    f_L(i_0,0,t_{i_0+1})\wedge g_R(i_0+1,1,v)&-f_L(i_0,s,\cq{u_s,v})\wedge g_R(i_0+1,s+1,v)\\
    &=f_L(i_0,0,t_{i_0+1})\wedge g_R(i_0+1,1,v)&=-G_L(i_0+1).
    \label{eqn:dq_2}
\end{align}
If $i_0=r,j_0<s$, then from (\ref{eqn:dq_2a}) and (\ref{eqn:dq_2bc_b}) we also have the terms
\begin{align}
-f_L(r,j_0,v)&\wedge g_R(r+1,j_0+1,v),\label{eqn:dq_2_t1}\\
(v-u_{j_0+1})f_L(r,j_0,\cq{u_{j_0+1},v})&\wedge g_R(r+1,j_0+2,\cq{u_{j_0+1},v})\nonumber\\
=f_L(r,j_0+1,v)&\wedge\pq{g_R(r+1,j_0+2,v)-g_R(r+1,j_0+2,u_{j_0+1})}\label{eqn:dq_2_t2}.
\end{align}
The last term $f_L(r,j_0+1,v)\wedge g_R(r+1,j_0+2,u_{j_0+1})$ is $F_R(j_0+1)$. The remaining term and (\ref{eqn:dq_2_t1}) mostly cancel when summed over $j_0$, leaving only
\begin{align}
    Z:=-f_L(r,0,v)\wedge g_R(r+1,1,v)+f_L(r,s,v)\wedge g_R(r+1,s+1,v)\nonumber\\
    =-\GF\pg{a_1,\dots,a_r,\prod_jb_j\cdot c}{t_1,\dots,t_r,v}\wedge \GF\pg{b_1,\dots,b_s,\prod_ia_i\cdot c}{u_1,\dots,u_s,v}\label{eqn:dq_2_nice}.
\end{align}

If $j_0=s,i_0<r$, there are terms
\begin{align}
-f_L(i_0,s,v)&\wedge g_R(i_0+1,s+1,v)&=0,\nonumber\\
(v-t_{i_0+1})f_L(i_0,s,\cq{t_{i_0+1},v})&\wedge g_R(i_0+2,s+1,\cq{t_{i_0+1},v})&=0.
\end{align}
Finally, $i_0=r,j_0=s$ also produces 0.

Thus the sum of terms $f_L\wedge g_R$ is 
\begin{equation}
    Z-\sum_{i_0=0}^{r-1}G_L(i_0+1)-\sum_{j_0=0}^{s-1}F_R(j_0+1)\label{eqn:dq_all_2_fg}.
\end{equation}

Similarly, terms of the form $g_L\wedge f_R$ give
\begin{equation}
    -Z-\sum_{j_0=0}^{s-1}F_L(j_0+1)-\sum_{i_0=0}^{r-1}G_R(i_0+1)\label{eqn:dq_all_2_gf}.
\end{equation}

The terms $g_L\wedge g_R$ where $i_0<r$, $j_0<s$ are, similarly:
\begin{align}
    -g_L(i_0,j_0,t_{i_0+1})&\wedge g_R(i_0+2,j_0+1,\cq{t_{i_0+1},v})\nonumber\\
    +g_L(i_0,j_0,\cq{t_{i_0+1},u_{j_0+1}})&\wedge g_R(i_0+2,j_0+2,\cq{t_{i_0+1},v}).\label{eqn:dq_all_2_gg}
\end{align}
If $i_0=r,j_0<s$, we get the terms
\begin{align}
-g_L(r,j_0,v)&\wedge g_R(r+1,j_0+1,v),\nonumber\\
(v-u_{j_0+1})g_L(r,j_0,\cq{u_{j_0+1},v})&\wedge g_R(r+1,j_0+2,\cq{u_{j_0+1},v}).
\label{eqn:dq_all_2_ggm}
\end{align}
If $j_0=s,i_0<r$, there are terms
\begin{align}
-g_L(i_0,s,v)&\wedge g_R(i_0+1,s+1,v)&=0\nonumber,\\
(v-t_{i_0+1})g_L(i_0,s,\cq{t_{i_0+1},v})&\wedge g_R(i_0+2,s+1,\cq{t_{i_0+1},v})&=0.
\end{align}
The case $i_0=r$, $j_0=s$ again contributes 0.

The terms $f_L\wedge f_R$ are symmetrical.

Assembling (\ref{eqn:dq_all_2_gf})-(\ref{eqn:dq_all_2_ggm}), the total contribution of cuts (2a/b/c) is
\begin{align}
-\sum_{i=1}^rG_L(i)&-\sum_{j=1}^sF_R(j)\label{eqn:dq_2_total1}\\
+\sum_{i_0=0}^{r-1}\sum_{j_0=0}^{s-1}\biggl(
    -g_L(i_0,j_0,t_{i_0+1})&\wedge g_R(i_0+2,j_0+1,\cq{t_{i_0+1},v})\label{eqn:dq_2_total2}\\
    +g_L(i_0,j_0,\cq{t_{i_0+1},u_{j_0+1}})&\wedge g_R(i_0+2,j_0+2,\cq{t_{i_0+1},v})\biggr)\label{eqn:dq_2_total2v}\\
    +\sum_{j_0=0}^{s-1}\biggl(
        -g_L(r,j_0,v)&\wedge g_R(r+1,j_0+1,v)\label{eqn:dq_2_total3v}\\
+(v-u_{j_0+1})g_L(r,j_0,\cq{u_{j_0+1},v})&\wedge g_R(r+1,j_0+2,\cq{u_{j_0+1},v})\biggr),\label{eqn:dq_2_total3}
\end{align}
plus symmetrical terms.
\end{proof}

\begin{proof}[Proof of Lemma~\ref{lma:pf_dq}]
    Cancellation of (\ref{eqn:dq_1_total}) with (\ref{eqn:dq_2_total1}) leaves
    \begin{equation}
    \sum_{\substack{1\leq j_0,j_1\leq s\\j_1-j_0\geq-1}}F(j_0,j_1)
    \label{eqn:dq_sum_ffgg}
    \end{equation}
plus the symmetrical term.

Thus $\delta Q$ is the symmetrized sum of expressions (\ref{eqn:dq_2_total2})-(\ref{eqn:dq_sum_ffgg}).
\end{proof}

\begin{lmacomp}[Step 1(b)]
\label{lma:pf_dr} 
Modulo elements $C(0,x)$, $\delta R_B$ and $\delta R_A$ are given by expression (\ref{eqn:drb_all}) below and its symmetric expression, respectively.
\end{lmacomp}

\begin{proof}[Proof of Lemma~\ref{lma:pf_dr}]
We compute $\delta R_B$.

Recall that the distinguished segment of $R_B$ is $\prod_ja_i\cdot c$. We use the above classification of cuts of type (1a/b/c/d) and (2a/b/c).

Consider first the terms $f\wedge g$ coming from cuts of type (1). For each such term, let $j_0,j_1$ be the minimal and maximal indices of $b_j$ that do not appear in $g$. For fixed $j_0,j_1$, the cuts of type (1a), (1b), and (1c/d) produce precisely the expressions (\ref{eqn:qsh_red_a_noa}), (\ref{eqn:qsh_red_b_soa}), and (\ref{eqn:qsh_red_cd_noa}) above. Thus the contribution of cuts of type (1) is $F(j_0,j_1)$, and the total contribution is 
\begin{equation}
    \sum_{\substack{0\leq j_0,j_1\leq s\\j_1-j_0\geq-1}}F(j_0,j_1).
    \label{eqn:dr_1}
\end{equation}

Next, we look at cuts of type (2). We will need a simplified formula for terms where either the vertex or the cut are on a segment indexed with $S=\cq{s_1,\dots,s_k}$. If $k=1$ and the vertex is at a nonzero point, we get terms of the form \[\GF\pg{\dots}{\dots,s_1}\wedge \GF\pg{\dots}{s_1,\dots}.\] Applying (\ref{eqn:multi_identity}), it is easy to show by induction that, for general $k$, the resulting terms are
\begin{equation}
    \sum_{i=1}^k\GF\pg{\dots}{\dots,\cq{s_1,\dots,s_i}}\wedge \GF\pg{\dots}{\cq{s_i,\dots,s_k},\dots}.\label{eqn:multi_cop_id_1}
\end{equation}
For example, if $k=2$, this becomes
\begin{align}
    \GF\pg{\dots}{\dots,s_1}\wedge \GF\pg{\dots}{\cq{s_1,s_2},\dots}&+\GF\pg{\dots}{\dots,\cq{s_1,s_2}}\wedge \GF\pg{\dots}{s_2,\dots}\nonumber
    \\=\f{1}{s_1-s_2}\biggl(\GF\pg{\dots}{\dots,s_1}\wedge \GF\pg{\dots}{s_1,\dots}&+\GF\pg{\dots}{\dots,s_2}\wedge \GF\pg{\dots}{s_2,\dots}\biggr)\nonumber,
\end{align}
agreeing with the formula following directly from (\ref{eqn:multi_identity}) that has been used in the previous computations.

Similarly, if the vertex is on some segment $s'$, the term for $k=1$, \[s'\GF\pg{\dots}{\dots,\cq{s_1,s'}}\wedge \GF\pg{\dots}{\cq{s_1,s'},\dots},\] expands into
\begin{equation}
    s'\sum_{i=1}^k\GF\pg{\dots}{\dots,\cq{s_1,\dots,s_i,s'}}\wedge \GF\pg{\dots}{\cq{s_i,\dots,s_k,s'},\dots}.\label{eqn:multi_cop_id_2}
\end{equation}

Finally, if the vertex is at a 0 on the segment $s_i$ and the cut is on the segment $s'$, we get terms
\begin{align}
    \sum_{i=1}^ks_i\GF\pg{\dots}{\dots,\cq{s_1,\dots,s_i,s'}}&\wedge \GF\pg{\dots}{\cq{s_i,\dots,s_k,s'},\dots}\nonumber\\
    +\sum_{i=1}^{k-1}\GF\pg{\dots}{\dots,\cq{s_1,\dots,s_i,s'}}&\wedge \GF\pg{\dots}{\cq{s_{i+1},\dots,s_k,s'},\dots}.
    \label{eqn:multi_cop_id_3}
\end{align}

These identities can also be shown combinatorially, by interpreting the definition of the multiple generating functions in terms of collapsing segments.

For a term $f\wedge g$ coming from a cut of type (2), let $j_0$ be the maximal index of $b_j$ appearing in $f$ and $j_1$ the minimal index in $g$, so $j_1-j_0=1$ for cuts (2a) and $j_1-j_0=2$ for cuts (2b/c). By (\ref{eqn:multi_cop_id_1}), for fixed $j_0$, the cuts of type (2a) contribute
\begin{align}
-\sum_{i=1}^mg_L(i,j_0,\emptyset)&\wedge g_R(i,j_0+1,v)\label{eqn:dr_2a}
\\+\,g_L(r,j_0,v)&\wedge g_R(r+1,j_0+1,v). \label{eqn:dr_2av}
\end{align}
By (\ref{eqn:multi_cop_id_2}), the cuts of type (2b) contribute
\begin{align}
-\sum_{i=1}^mu_{j_0+1}g_L(i,j_0,u_{j_0+1})&\wedge g_R(i,j_0+2,\cq{u_{j_0+1},v})\label{eqn:dr_2b}\\
-\,u_{j_0+1}g_L(r,j_0,\cq{u_{j_0+1},v})&\wedge g_R(r+1,j_0+2,\cq{u_{j_0+1},v}).\label{eqn:dr_2bv}
\end{align}
By (\ref{eqn:multi_cop_id_3}), the cuts of type (2c) contribute
\begin{align}
\sum_{i=1}^mt_ig_L(i,j_0,u_{j_0+1})&\wedge g_R(i,j_0+2,\cq{u_{j_0+1},v})\label{eqn:dr_2c1}\\
+\,vg_L(r,j_0,\cq{u_{j_0+1},v})&\wedge g_R(r+1,j_0+2,\cq{u_{j_0+1},v})\label{eqn:dr_2cv}\\
+\sum_{i=1}^mg_L(i,j_0,u_{j_0+1})&\wedge g_R(i+1,j_0+2,\cq{u_{j_0+1},v}).
\label{eqn:dr_2c2}
\end{align}
The sum of expressions (\ref{eqn:dr_2b}), (\ref{eqn:dr_2c1}), and (\ref{eqn:dr_2c2}) simplifies to
\begin{equation}
    g_L(i,j_0,u_{j_0+1})\wedge g_R(i,j_0+2,v).
\label{eqn:dr_2bc}
\end{equation}
Then, letting $H_B(j_0)$ be the sum of expressions (\ref{eqn:dr_2a}), (\ref{eqn:dr_2av}), (\ref{eqn:dr_2bv}), (\ref{eqn:dr_2cv}), and (\ref{eqn:dr_2bc}), the coproduct of $R_B$ is 
\begin{equation}
    \sum_{\substack{0\leq j_0,j_1\leq s\\j_1-j_0\geq-1}}F(j_0,j_1)
    +\sum_{j=0}^{s-1}H_B(j).
\label{eqn:drb_all}
\end{equation}

The coproduct of $R_A$ is the symmetric expression.
\end{proof}

\begin{proof}[Proof of Lemma~\ref{lma:pf_lwsr}]
We now compare the results of the computations in Lemmas~\ref{lma:pf_dq} and \ref{lma:pf_dr}.

We have computed that $\delta Q$ is the symmetrization of 
\[(\ref{eqn:dq_2_total2})
+(\ref{eqn:dq_2_total2v})
+(\ref{eqn:dq_2_total3v})
+(\ref{eqn:dq_2_total3})
+(\ref{eqn:dq_sum_ffgg})\] 
and $\delta R_B+\delta R_A$ is the symmerization of 
\[(\ref{eqn:dr_1})
+\sum_{j=0}^{n-1}\bigl[
(\ref{eqn:dr_2a})
+(\ref{eqn:dr_2av})
+(\ref{eqn:dr_2bv})
+(\ref{eqn:dr_2cv})
+(\ref{eqn:dr_2bc})\bigr].\]

Obviously $(\ref{eqn:dr_1})=(\ref{eqn:dq_sum_ffgg})$. Now
\[(\ref{eqn:dq_2_total2})=\sum(\ref{eqn:dr_2a}),\quad
(\ref{eqn:dq_2_total2v})=\sum(\ref{eqn:dr_2bc}),\quad
(\ref{eqn:dq_2_total3v})=\sum(\ref{eqn:dr_2av}),\quad
(\ref{eqn:dq_2_total3})=\sum\bq{(\ref{eqn:dr_2bv})+(\ref{eqn:dr_2cv})},
\]
which finishes the proof.
\end{proof}

\subsubsection{Proof of Step 2}

Here we show the terms of weight $(1)\wedge(w-1)$ coming from $\delta Q-\delta R_A-\delta R_B$ are 0.

\begin{proof}

We first examine the relevant terms of $\delta Q$. Let us compute the coefficient $L_i^A$ occurring with $C(0,a_i)$. These come from shuffles containing segment $\pg{a_i}{t_i}$, a segment $\pg{a_ib_j}{\cq{t_i,u_j}}$, and the segment $\pg{c}{v}$ (where we write $c=\prod_ia_i\inv\prod_jb_j\inv$). 

Inspect the generating functions of depth 1 $\GF\pg{w,w\inv}{s_1,s_2}$ that appeared in the proof of Lemma~\ref{lma:pf_lwsr}. All generating functions in the lower half of cuts (1a/b/c/d) were written in a form where $\pg{w}{s_1}$ is the first segment counterclockwise of the distinguished segment, rather than with the segment counterclockwise of the vertex of the cut as in (\ref{eqn:gf_cop_1})).

So, by the remark following Lemma~\ref{lma:cop_of_gf},

the terms (\ref{eqn:log_cop_gf_r}) vanish in the coproduct, so the terms arising from these cuts are canceled by the lower-depth shuffle relations in Lemma~\ref{lma:pf_lwsr}. Similarly, for cuts of type (2), we only have terms (\ref{eqn:log_cop_gf_lr}) contributing the coefficient of $C(0,a_i\inv)$.

For quasishuffles in which $\pg{a_i}{t_i}$ appears, the terms (\ref{eqn:log_cop_gf_l}) where some $b_j$ appears immediately clockwise of $a_i$ gives terms
\begin{align}
\QSh^{r,s}_{(ij)}\pg{a_1,\dots,a_i,\dots,a_r,b_1,\dots,b_j,c}{&t_1,\dots,\emptyset,\dots,t_r,u_1,\dots,u_j,\dots,u_s,v},\label{eqn:dq_log_ab}
\end{align}
where $\QSh^{r,s}_{(ij)}$ denotes the sum over only those quasishuffles where $a_i$ collapses with $b_j$.

The terms (\ref{eqn:log_cop_gf_l}) where either $a_{i-1}$ or some $a_{i-1}b_j$ appears immediately clockwise of $a_i$ sum to
\begin{align}
    \QSh^{r-1,s}\pg{a_1,\dots,a_{i-1}a_i,a_{i+1},\dots,a_r,b_1,\dots,b_s,c}{t_1,\dots,t_{i-1},t_{i+1},\dots,t_r,u_1,\dots,u_s,v}.\label{eqn:dq_log_al}
\end{align}

Finally, the terms (\ref{eqn:log_cop_gf_lr}) contribute to $L^A_i$ the terms 
\begin{align}
    t_i\QSh^{r,s}_{(i)}\pg{a_1,\dots,a_r,b_1,\dots,b_s,c}{t_1,\dots,t_r,u_1,\dots,u_s,v}
\label{eqn:dq_log_alr}
\end{align}
where $\QSh_{(i)}$ denotes the the quasishuffles in which $a_i$ does not collapse with any $b_j$.

For quasishuffles in which some $\pg{a_ib_j}{\cq{t_i,u_j}}$ appears, the terms (\ref{eqn:log_cop_gf_l}) contribute 0, since they arise from cuts of segments containing no 0s.
The terms (\ref{eqn:log_cop_gf_lr}) give
\begin{align}
\f{-1}{t_i-u_j}\biggl(
t_i\QSh^{r,s}_{(ij)}\pg{a_1,\dots,a_i,\dots,a_r,b_1,\dots,b_j,\dots,b_s,c}{&t_1,\dots,t_i,\dots,t_r,u_1,\dots,\emptyset,\dots,u_s,v}\nonumber\\
-u_j\QSh^{r,s}_{(ij)}\pg{a_1,\dots,a_i,\dots,a_r,u_1,\dots,u_j,\dots,u_s,v}{&t_1,\dots,\emptyset,\dots,t_r,u_1,\dots,u_j,\dots,u_s,v}\biggr)\nonumber\\
=-t_i\QSh^{r,s}_{(ij)}\pg{a_1,\dots,a_i,\dots,a_r,b_1,\dots,b_s,c}{&t_1,\dots,t_i,\dots,t_r,u_1,\dots,u_j,\dots,u_s,v}\label{eqn:dq_log_ablr1}\\
-\QSh^{r,s}_{(ij)}\pg{a_1,\dots,a_i,\dots,a_r,b_1,\dots,b_s,c}{&t_1,\dots,\emptyset,\dots,t_r,u_1,\dots,u_j,\dots,u_s,v}.
\label{eqn:dq_log_ablr2}
\end{align}

For the segment $\pg{c}{v}$, which includes a factor of $a_i\inv$, we get a contribution of
\begin{align}
-v\QSh^{r,s}\pg{a_1,\dots,a_r,b_1,\dots,b_s,c}{t_1,\dots,t_r,u_1,\dots,u_s,v}=-vQ.\label{eqn:dq_log_vq}
\end{align}
from (\ref{eqn:log_cop_gf_lr}) and
\begin{align*}
-\QSh^{r-1,s}\pg{a_1,\dots,a_{m-1},b_1,\dots,b_s,a_rc}{&t_1,\dots,t_{r-1},u_1,\dots,u_s,t_r}\\
-\QSh^{r,s-1}\pg{a_1,\dots,a_r,b_1,\dots,b_{n-1},b_sc}{&t_1,\dots,t_r,u_1,\dots,u_{s-1},u_s}\\
+\QSh^{r-1,s-1}\pg{a_1,\dots,a_{m-1},b_1,\dots,b_{n-1},a_rb_sc}{&t_1,\dots,t_{r-1},u_1,\dots,u_{s-1},\cq{t_r,u_s}},
\end{align*}
from (\ref{eqn:log_cop_gf_l}), with three terms, depending on which segment ($a_r$, $b_s$, or $a_rb_s$) appears clockwise of $c$. By the lower-depth shuffle relations, this simplifies to
\begin{align}
-\GF\pg{a_1,\dots,a_r,\prod_jb_j\cdot c}{&t_1,\dots,t_r,\cq{u_1,\dots,u_s}}\label{eqn:dq_log_ca}\\
-\GF\pg{b_1,\dots,b_s,\prod_ia_i\cdot c}{&u_1,\dots,u_s,\cq{t_1,\dots,t_r}}.\label{eqn:dq_cb}
\end{align}

The terms (\ref{eqn:dq_log_ab}) cancel with (\ref{eqn:dq_log_ablr2}). Summing (\ref{eqn:dq_log_ablr1}) over $j$ and adding to (\ref{eqn:dq_log_alr}) results in
\begin{equation}
    t_i\QSh^{r,s}\pg{a_1,\dots,a_r,b_1,\dots,b_s,c}{t_1,\dots,t_r,u_1,\dots,u_s,v}=t_iQ.
\label{eqn:dq_log_x}
\end{equation} 
Thus $L_i^A$ is the sum of (\ref{eqn:dq_log_al}), (\ref{eqn:dq_log_vq}), (\ref{eqn:dq_log_ca}), (\ref{eqn:dq_cb}) and (\ref{eqn:dq_log_x}). Applying lower-depth shuffle relations and (\ref{eqn:multi_identity}), this sum simplifies to
\begin{align}
L_i^A=\GF\pg{a_1,\dots,a_{i-1}a_i,a_{i+1},\dots,a_r,\prod_jb_j\cdot c}{&t_1,\dots,t_{i-1},t_{i+1},\dots,t_r,\cq{u_1,\dots,u_s}\sqcup\cq v}\label{eqn:dq_log_qa}\\
-\GF\pg{a_1,\dots,a_r,\prod_jb_j\cdot c}{&t_1,\dots,t_r,\cq{u_1,\dots,u_s}}\label{eqn:dq_log_qab}\\
+(t_i-v)(Q-R_B).\label{eqn:dq_log_qb}
\end{align}

Now let us compute the coefficient $M_i^A$ occuring with $C(0,a_i)$ in $\delta(R_A)$. For the segment $\pg{a_i}{t_i}$ in $R_A$, (\ref{eqn:log_cop_gf_lr}) and (\ref{eqn:log_cop_gf_l}) contribute the terms
\begin{align}
t_i\GF\pg{a_1,\dots,a_i,\dots,a_r,\prod_jb_j\cdot c}{&t_1,\dots,t_i,\dots,t_r,\cq{u_1,\dots,u_s}\sqcup\cq v}&=t_iR_A,\label{eqn:dr_log_x}\\
\GF\pg{a_1,\dots,a_{i-1}a_i,a_{i+1},\dots,a_r,\prod_jb_j\cdot c}{&t_1,\dots,t_{i-1},t_{i+1},\dots,t_r,\cq{u_1,\dots,u_s}\sqcup\cq v},\label{eqn:dr_log_al}
\end{align}
where the second term appears only if $i>1$.

The distinguished segment $\pg{\prod_ia_i\inv}{\cq{u_j}\sqcup\cq v}$ contributes only a term (\ref{eqn:log_cop_gf_lr}). By an argument similar to that in Lemma~\ref{lma:pf_dr}, this term can be written
\begin{align}
-vR_A-\GF\pg{a_1,\dots,a_r,\prod_jb_j\cdot c}{t_1,\dots,t_r,\cq{u_1,\dots,u_s}}.
\label{eqn:dr_log_v}
\end{align}

Combining (\ref{eqn:dq_log_qa})-(\ref{eqn:dr_log_v}), we find that \[(L_i^A-M_i^A)=(t_i-v)(Q-R_A-R_B).\] Therefore, adding the symmetric terms for the $C(0,b_j)$,
\[\delta(Q-R_A-R_B)=\bq{\sum_{i=1}^rC(0,a_i)(t_i-v)+\sum_{j=1}^sC(0,b_j)(u_j-v)}\wedge(Q-R_A-R_B)\]
modulo lower-depth shuffle relations and elements $(\text{weight 1})\wedge(\text{weight 1})$.

\end{proof}    

\subsubsection{Conclusion}

We are ready to use the coproduct we have computed to reduce the proof of the relations to a simple base case.

\begin{proof}[Proof of Theorem~\ref{thm:main_qdih_const}(a)]
We induct on the depth $r+s$. When $r=0$ or $s=0$, $\ol\QSh^{r,s}$ is identically 0.

If $r,s>0$, taking coproduct on both sides of (\ref{eqn:step2}) and using that $\delta^2=0$, one deduces that $\delta(Q-R_A-R_B)=0$ modulo shuffle relations of depth $<r+s$ and terms $\text{(weight 1)}\wedge\text{(weight 1)}$. 

When no terms $C(0,x)\wedge C(0,y)$ are present in the coproduct, Lemma~\ref{lma:pf_red_multi} and Lemma~\ref{lma:pf_log} imply that $\delta(Q-R_A-R_B)$ lies in the ideal generated by lower-depth relations. 

These terms appear only in a base case: the constant term of the shuffle relation for $r=s=1$. Showing the coproduct of this term is 0 amounts to proving the identity
\begin{equation}
    \delta\pq{\bq{ C^*(a|0,b|0,c|0)+ C^*(b|0,a|0,c|0)- C^*(ab|1,c|0)}- C^*(a|0,bc|1)- C^*(b|0,ac|1)}=0.\label{eqn:dilog_identity}
\end{equation}
We compute directly that the left side of (\ref{eqn:dilog_identity}) is 
\begin{align*}
\quad C(1,a)\wedge C(1,ab)+C(1,ab)\wedge (C(1,b)+C(0,a))+(C(1,b)+C(0,a))&\wedge C(1,a)\\
+\,C(1,b)\wedge C(1,ab)+C(1,ab)\wedge (C(1,a)+C(0,b))+(C(1,a)+C(0,b))&\wedge C(1,b)\\
-\,C(1,ab)\wedge C(0,ab)+C(1,a)\wedge C(0,a)+C(1,b)&\wedge C(0,b)\\
=\quad C(1,ab)\wedge C(0,a)+C(0,a)&\wedge C(1,a)\\+\,C(1,ab)\wedge C(0,b)+C(0,b)&\wedge C(1,b)\\-\,C(1,ab)\wedge C(0,ab)+C(1,a)\wedge C(0,a)+C(1,b)&\wedge C(0,b)&=0.
\end{align*}

The theorem is proved.
\end{proof}

\section{Specialization theorem for Hodge correlators}

\label{sec:nodal}

We now study how the Hodge correlators over a base $B$ behave when the sections collide. This will require extending the theory of Hodge correlators to nodal curves.

\subsubsection{The correlator Lie coalgebra for nodal curves} 

Recall the moduli space $\MMM_{0,n}'$ of $n$ distinct points and a distinguished tangent vector on $\PP^1$. Its Deligne-Mumford compactification $\ol\MMM_{0,n}'$ consists of the nodal curves of genus 0, i.e., those whose dual graph is a tree and in which every component is a punctured projective line. with $n$ marked points and a distinguished tangent vector $v_\infty$.

Let $X=\bigcup_iX_i$ be a genus 0 nodal curve with a set of punctures $S$. Let $T$ be the dual tree of $X$, with vertices indexed by $i$ corresponding to $X_i$, rooted at the component $0$ with the base point $s_0\in X_0$, oriented away from the root (write $i\to j$ if $(i,j)$ is an edge). Choose a coordinate $z_i$ on each $X_i$ such that the point joining the component to its parent $X_j$ is $(z_i=\infty,z_j=v_{ji})$, and the base point on $X_0$ is at $z_0=\infty$ with tangent vector $v_\infty$. Let $S_i$ be the set of punctures on $X_i$. Let $N_i=\cq{v_{ij}:i\to j}$.

We define the correlator Lie coalgebra for the nodal curve $X$ by
\begin{equation}
    \CLie^\vee_{X,S,v_\infty} =
\bigoplus_{i}\CLie^\vee_{X_i,S_i\cup N_i,v_i}
,\label{eqn:def_nodal_clie}
\end{equation}
where $v_i$ is the tangent vector $\f{-1}{z_i^2}\f{\del}{\del z_i}$ at $z_i=\infty$.

It coincides with the usual definition if $X$ is smooth, justifying the notation. If $X$ is not smooth, it is different from $\widetilde{\CLie}^\vee_{X,S,s_0}$, the coalgebra naively defined as the tensor algebra of $S$ modulo cyclic symmetry and shuffle relations with a $H_2(X)$ coefficient. They are related in the following way. For each $i$, there is a surjective coalgebra morphism to the component of the direct sum corresponding to $X_i$:
\[\widetilde{\CLie}^\vee_{X,S,s_0}\tto{\pi_i}\CLie^\vee_{X_i,S_i\cup N_i,v_i}.\]
To define it on a generator $(x_1\otimes\dots\otimes x_n)\otimes[X_i]$, let $p$ be the common parent of the components containing the $x_j$. If $p\neq i$, 
the $i$-th component of the map is 0. Otherwise, set
\begin{align*}
\pi_i(x)=\begin{cases}x,&x\in X_i\\z_i=v_{ij}\in N_i,&\text{$x\in X_k$, where $\exists$ path $i\to j\to\dots\to k$}\end{cases},
\end{align*}
extended to preserve the tensor product. That is, points in $X_i$ remain, while points in components below $X_i$ collapse to the nearest node on $X_i$. Evidently this map preserves the coproduct and defining relations. Taking the direct sum of the maps $\pi_i$, we have produced a coalgebra morphism:
\begin{equation*}
\pi:\widetilde{\CLie}^\vee_{X,S,v_\infty}\to\CLie^\vee_{X,S,v_\infty}.
\end{equation*}
It preserves the decomposition of the domain by $H_2(X)=\bigoplus_iH_2(X_i)$. 

In particular, if $(X,S,v_0)$ vary over a base $B\to\MMM_{0,n}'$, and the variation extends to $\ol B\to\ol\MMM_{0,n}'$, with $D=\ol B\sm B$, then we have a \emph{degeneration map} 
\begin{equation}
    \pi_D:\CLie^\vee_{X/B,S,v_\infty}\to\widetilde{\CLie}^\vee_{X/D,S,v_\infty}\to\CLie^\vee_{X/D,S,v_\infty},
\end{equation}
where the first map simply applies the induced map on $H_2$ and the second map is the quotient defined above. The composition forgets the way in which the sections in $S$ collided at boundary of $B$.

\subsubsection{Specialization theorem}

Recall that an element of $\CLie^\vee_{X/B,S,v_\infty}$ over a base $B\to\MMM_{0,n}'$ determines, by the map $\Cor_\Hod$, a variation of Hodge structures over $B$, and, by the period map $p$, a smooth function on $B$. The maps $\Cor_\Hod$ and $p$ also exist for $X$ a nodal curve, extended by linearity from the definition (\ref{eqn:def_nodal_clie}).

\begin{thm}
Suppose $B\to\MMM_{0,n}'$ is a family of curves $(X,S,v_\infty)$ extending to $\ol B\to\ol\MMM_{0,n}'$, with $D=\ol B\sm B$ a normal crossings divisor, and suppose $x\in\CLie^\vee_{X/B,S,s_0}$ of weight $n>1$.
\begin{enumerate}[(a)]
\item The Deligne's canonical extension to $D$ of the variation of framed mixed Hodge structures determined by $\Cor_\Hod(x)$ is independent of the normal vector to $D$. Thus there is a specialized map $\Spec_D\Cor_\Hod:\pq{\CLie_{X/B,S,v_\infty}\to\Lie_{\HT/D}^\vee}_{w>1}$.
\item This specialized map coincides with the Hodge correlator of the degeneration map:
\[
 \xymatrix{
    \pq{\CLie^\vee_{X/B,S,v_\infty}}_{w>1}\ar[r]^{\pi_D}\ar[d]_{\Cor_\Hod}
    &\pq{\CLie^\vee_{X/D,S,v_\infty}}_{w>1}\ar[d]^{\Cor_\Hod}\\(\Lie_{\HT/B}^\vee)\ar[r]^{\Spec_D}&(\Lie_{\HT/D}^\vee).
 }
.\]
\item Let $t=0$ be a local equation for $D$. Then
\[\lim_{t\to0}p(\Cor_\Hod(x_t))=p(Cor_Hod(x_{t=0})).\]
\end{enumerate}
\label{thm:degen_nodal}
\end{thm}

\begin{proof}
Let $x\in\CLie^\vee_{X/B,S,v_\infty}$ be a generator of weight $w>1$. For any $v$ a normal vector to $D$, we get the specialized framed mixed Hodge-Tate structure $\Spec_D^v\Cor_\Hod(x)$.
    
We must show that:
\begin{enumerate}[(1)]
    \item The periods of $\Cor_\Hod(x)$ extend continuously to $D$.
    \item The coproduct of $\Spec_D^v\Cor_\Hod(x)$ does not depend on the direction of specialization $v$ at any smooth point of $D$.
    \item The periods of the specializations (i.e., the limits of the periods at $D$) coincide with the periods of the degeneration to $D$.
\end{enumerate}

We will prove (1-3) by induction on the weight. First, let us see how they imply the result.

Assuming (2), the coproduct of $\Spec_D^v\Cor_\Hod(x)$ is independent of $v$. Because the coproduct commutes with $\Spec_D^v$, this element is independent of $v$ up to $\Ext^1(\RR(0),\RR(n))$, which is 1-dimensional and controlled by the period. By (1), the period is independent of the direction of specialization, which gives (a). By (3), it coincides with the period of the degeneration, which gives (b). Then (c) follows by the definitions from (b).

To show (1), we let $\cq{\eps_i=0}$ be a set of smooth local equations for $D$ and prove that $p(\Cor_\Hod(x))$ can be represented locally as a polynomial in the $\log\eps_i$ such that the terms with $\log\eps_i$ appearing in positive degree have coefficients vanishing along $\cq{\eps_i=0}$ (\emph{tame logarithmic singularities}). This will follow from the differential equations on the periods. Note that in weight 1, the period of $C(x,y)$ has a (not tame) logarithmic singularity along $x=y$. In weight $>1$, we proceed by induction.

Consider a simple element $x=x_0\otimes\dots\otimes x_n\in\CLie^\vee_{X/B,S,v_\infty}$ ($n>1$). Suppose that not all $x_i$ collide on $D$, so we must only consider the summand of the nodal $\CLie^\vee_{X/D,S,s_0}$ corresponding to the component containing the base point. The terms of $\delta(x)$ can be grouped into those of two forms:
\begin{enumerate}[(i)]
\item $x'\wedge x''$, where not all sections in $x'$ and in $x''$ collide to the same section on $D$;
\item $x'\wedge(x''_1-x''_2)$, where not all sections in $x'$ collapse on $D$, but $x''_1$ and $x''_2$ coincide on $D$.
\end{enumerate}
By the inductive hypothesis, the specialization of $\delta(x)$ does not depend on the direction of specialization: for terms (i), $x'$ and $x''$ satisfy (2), while in terms (ii) the $x''_1-x''_2$ vanish under specialization to $D$. This gives (2).

For (1), from the differential equations on the periods~(\ref{eqn:diff_eq_p1}), we see that $d_Bp(\Cor_\Hod(x))$ is a sum of terms that are smooth over $B$ with logarithmic singularities along $D$ (from type (i)) and terms that vanish along $D$ by the inductive hypothesis (from type (ii)). We conclude that $p^{v_\infty}(x)$ has tame logarithmic singularities along $D$.

If all $x_i$ collide on $D$, we simply pass to their common parent component and apply the same argument.

We conclude with (3). We have shown that the specializations of $\Cor_\Hod^{v_0}$ and its coproduct to $D$ exist at every point and their periods are independent of $v$, and thus the specialized period map $p\circ\Cor_\Hod$ is equal to the period of the degeneration up to adding a constant for each smooth component of the smooth locus of $D$. We must show the constant 0.

It is enough to show this for $D$ a lowest-codimension boundary stratum in $\ol\MMM_{g,n}'$. We are done by the next lemma.

\end{proof}

\begin{lma}
Let $I$ be a proper subset of $\cq{0,1,\dots,n}$ ($n>1$) and $x_0,\dots,x_n\in\CC^*$ with $x_i\neq x_j$ if $i\neq j$ and either $i,j\in I$ or $i,j\notin I$. Let \[x_i(t)=\begin{cases}tx_i&i\in I,\\x_i&i\notin I\end{cases}.\] Then \[\Cor_\HHH(x_0(t),x_1(t),\dots,x_n(t))\] is continuous at $t=0$.

\label{lma:circ_spec}
\end{lma}
\begin{proof}
For $n=2$, this amounts to continuity of $\LLL_2$ at 1.

In the proof of Theorem~\ref{thm:degen_nodal} it was established that 
\[\lim_{t\to0}\Cor_\HHH(x_0(t),x_1(t),\dots,x_n(t))-\Cor_\HHH(x_0(0),x_1(0),\dots,x_n(0))\]
is independent of the $x_i$, for generic $x_i$. Let us integrate this difference over $(x_0,\dots,x_n)\in(S^1)^{n+1}$, with respect to the standard measures $\mu(x_i)$ of volume 1 on $S^1=\cq{\aq z=1}\subset\CC$. 

The limit is uniform in the directions $x_i$ ($i\in I$), and so
\[\int\lim_{t\to0}\Cor_\HHH(x_0(t),x_1(t),\dots,x_n(t))\,\prod d\mu(x_i)=\lim_{t\to0}\int\Cor_\HHH(x_0(t),x_1(t),\dots,x_n(t))\,\prod d\mu(x_i).\]
To conclude, it suffices to show that
\begin{equation}
    \int\Cor_\HHH(x_0(t),x_1(t),\dots,x_n(t))\,\prod d\mu(x_i)=0.\label{eqn:circ_norm_int}
\end{equation}
for all $t$.

For any tree $T$ entering into the Feynman integral expression for (\ref{eqn:circ_norm_int}), choose a pair of boundary vertices (without loss of generality, labeled $x_0$ and $x_1$) incident to a common internal vertex $v$ with corresponding variable $x_v$, and let $x_w$ be variable corresponding to the third vertex incident to $v$. Then the integral over the $x_i$ contains the term
\[\int_{x_0,x_1}\pq{\int\LLL_2\pq{\f{x_w-x_0(t)}{x_w-x_1(t)}}\wedge(\text{terms independent of $x_0(t),x_1(t)$})}\,d\mu(x_0)\,d\mu(x_1).\]
Exchanging the two integrals and noting that $\LLL_2(\f{z-a}{z-b})$ changes sign under the involution \[a\mapsto\ol a\f{z^2}{\aq z^2},\quad b\mapsto\ol b\f{z^2}{\aq z^2},\]
we conclude that this expression is 0.
\end{proof}

The specialization theorem states is that when the punctures labeling an element of $\CLie^\vee$ collide, only the nearest possible to the base point component of the resulting nodal curve determines the limit Hodge correlator. We obtain as a corollary Theorem~\ref{thm:corrh_cts}:
\begin{thm*}
    The Hodge correlators $\Cor_\HHH(z_0,\dots,z_n)$ are continuous on $\CC^{n+1}\sm\cq{z_0=\dots=z_n}$.
\end{thm*}

For example, $\LLL_2$ is continuous with a tame logarithmic singularity at 1, but $\LLL_2\pq{\f{a-c}{b-c}}$ has no limit as $a,b,c\to0$.

\section{The second shuffle relations}

\subsection{Proofs of Theorems~\ref{thm:hodge_main} and \ref{thm:mot_main}}

In this section we will prove the second shuffle relations for Hodge and motivic correlators.

\subsubsection{Proof for Hodge correlators}

Recall Theorem~\ref{thm:hodge_main}:
\begin{thm*}
    \begin{enumerate}[(a)]
    \item Restricted to the subspace of $\CLie_{X,S,v_\infty}^\vee$ generated by elements $(x_0\otimes\dots\otimes x_n)(1)$ with not all $x_i$ equal, the map $\Cor_\Hod$ factors through $\DDD^\circ(\CC\du)$.
    \item
    Suppose that $r,s>1$ and that not all $n_i=0$ or not all $w_i=1$. Then the Hodge correlators satisfy the relation:
    \begin{align*}
        &\sum_{\sigma\in\ol\Sigma_{r,s}}(-1)^{r+s-M_\sigma}\Cor_\Hod\du(w_{\sigma\inv(1)}|n_{\sigma\inv(1)},\dots,w_{\sigma\inv(M_\sigma)}|n_{\sigma\inv(M_\sigma)},w_0|n_0)\\
        &-\Cor_\Hod\du(w_1|n_1,\dots,w_r|n_r,w_{\cq{r+1,\dots,r+s,0}}|n_{\cq{r+1,\dots,r+s,0}})\\
        &-\Cor_\Hod\du(w_{r+1}|n_{r+1},\dots,w_{r+s}|n_{r+s},w_{\cq{1,\dots,r,0}}|n_{\cq{1,\dots,r,0}})&=0,
    \end{align*}
    where
    \[n_S=\sum_{i\in S}(n_i+1)-1,\quad w_S=\prod_{i\in S}w_i.\]
    \item The Hodge correlators satisfy all specializations of this relation as any subset of the $w_i$ $(1\leq i\leq n)$ approaches 0.
\end{enumerate}
\end{thm*}

\begin{proof}
    For fixed $r$, $s$, and $n_i$, consider the $(r,s)$-second shuffle relation in (b). It is a family of framed mixed Hodge-Tate structures over
    \[S=\cq{(w_0,\dots,w_n)\in(\CC\du)^{n+1}:w_0\dots w_n=1}.\]
    To show (b), it suffices to show the family is trivial as an element of $\Lie_\HT^\vee$ over every point of $S$, except at $(1,\dots,1)$ if all $n_i=0$. This is equivalent to (a) by the definitions, as the Hodge correlators are already known to satisfy the defining relations in $\widetilde\DDD^\circ(\CC\du)$.

    Each term of this relation is an element
    \[\Cor_\Hod(1,z_1,\dots,z_n),\] where each $z_k$ is either 0 or monomial in the $w_i$. By Theorem~\ref{thm:hc_main}, it is a variation $\mathbf V$ of framed mixed Hodge-Tate structures over \[T=\cq{(z_1,\dots,z_n)\in(\CC^*)^n}\sm\pq{\text{\rm diagonals}}.\] We first show by induction on the weight $n$ that all such variations is trivial.

    In the base case $n=1$, there are no second shuffle relations.

    For the induction hypothesis, suppose $n>1$ and (b) holds in weights $1<w<n$. Fix $r$, $s$, and $n_i$ and let $\mathbf V$ be the variation defined above. By the induction hypothesis, $\delta\Cor_\Hod(\mathbf{V})$ vanishes, and thus, by rigidity, $\mathbf V$ is a constant variation, determined pointwise as an element of $\Ext^1(\RR(0),\RR(n))$ by the period. We show the period is 0.

    The specialization theorem (\S\ref{sec:nodal}) implies that the period of $\mathbf{V}$ is continuous away from the main diagonal in $\CC^{n+1}$. Unless all $n_i=0$ or  all $w_i=1$, in no term of the relation (b) do all points collide to the main diagonal. By Corollary~\ref{thm:corrh_cts}, the specialization of the period at $w_1,\dots,w_n=0$ is equal to the substitution $w_i=0$. Under this substitution, the period of each term of the relation becomes \[\Cor_\HHH(1,0,\dots,0)=0.\] Therefore, $\mathbf V$ is trivial over $T$.

    Because $T$ is dense in $\CC^n$, the relation at all points -- except $w_1=\dots=w_n=1$ if all $n_i=0$ -- follows by the specialization theorem. This completes the proof of (b) and (c).
\end{proof}

Applying the period map, we immediately obtain Theorem~\ref{thm:period_main}:
\begin{thm*}
    \begin{enumerate}[(a)]
        \item
        Suppose that $r,s>1$ and that not all $n_i=0$ or not all $w_i=1$. Then the Hodge correlators satisfy the relation:
        \begin{align*}
            &\sum_{\sigma\in\ol\Sigma_{r,s}}(-1)^{r+s-M_\sigma}\Cor_\HHH\du(w_{\sigma\inv(1)}|n_{\sigma\inv(1)},\dots,w_{\sigma\inv(M_\sigma)}|n_{\sigma\inv(M_\sigma)},w_0|n_0)\\
            &-\Cor_\HHH\du(w_1|n_1,\dots,w_r|n_r,w_{\cq{r+1,\dots,r+s,0}}|n_{\cq{r+1,\dots,r+s,0}})\\
            &-\Cor_\HHH\du(w_{r+1}|n_{r+1},\dots,w_{r+s}|n_{r+s},w_{\cq{1,\dots,r,0}}|n_{\cq{1,\dots,r,0}})&=0,
        \end{align*}
        where
        \[n_S=\sum_{i\in S}(n_i+1)-1,\quad w_S=\prod_{i\in S}w_i.\]
        \item The Hodge correlators satisfy all specializations of this relation as any subset of the $w_i$ $(1\leq i\leq n)$ approaches 0.
    \end{enumerate}
\end{thm*}

\subsubsection{Proof for motivic correlators}

Recall Theorem~\ref{thm:mot_main}:
\begin{thm*}
    Let $F$ be a number field.
    \begin{enumerate}[(a)]
        \item Restricted to the subspace of $\pq{\CLie_{X,S,v_\infty}^\Mot}^\vee$ generated by elements $(x_0\otimes\dots\otimes x_n)(1)$ with not all $x_i$ equal, the map $\Cor_\Mot$ factors through $\DDD^\circ(F^\times)$.
        \item
        Suppose that $r,s>1$ and that not all $n_i=0$ or not all $w_i=1$. Then the motivic correlators satisfy the same relation as in Theorem~\ref{thm:hodge_main}, with $\Cor_\Hod\du$ replaced by $\Cor_\Mot\du$.
        \item The motivic correlators satisfy all specializations of this relation as any subset of the $w_i$ $(1\leq i\leq n)$ approaches 0.
    \end{enumerate}
\end{thm*}

\begin{proof}
    Fix an embedding $F\tto r\CC$. It induces a map $\DDD^\circ(F^\times)\to\DDD^\circ(F^\times)$, which we also denote by $r$.

    Denoting by $\CLie^{\vee\circ}$ the subalgebras generated by elements $(x_1\otimes\dots\otimes x_n)(1)$ where not all $x_i$ are equal, we have the diagram
    \[
    \xymatrix{
    \pq{\CLie_{X,S,v_\infty}^\Mot}^{\vee\circ} 
    \ar[rr]^{\Cor_\Mot}\ar[dd]^r\ar[dr]
    &&\Lie_{\MTF}^\vee\ar[dd]^r
    \\&\DDD^\circ(F^\times)\ar@{-->}[ur]\ar[dd]^{r}\\
    \CLie_{X,S,v_\infty}^{\vee\circ}\ar[dr]
    \ar[rr]^{\Cor_\Hod\quad\quad}
    &&\Lie_{\HT}^\vee\ar[r]^p&\RR
    \\&\DDD^\circ(\CC\du)\ar[ur]\\
    }    
    ,\]
    where the lower half commutes by Theorem~\ref{thm:hodge_main} and the vertical maps are induced by $r$.
    
    It is necessary to show the dashed arrow is well-defined, i.e., that $\Cor_\Mot$ vanishes on the kernel of the map $\pq{\CLie_{X,S,v_\infty}^\Mot}^{\vee\circ}\to\DDD^\circ(F^\times)$. 
    
    Commutativity of the diagram for every embedding $r$ implies the result. Precisely, we argue by induction. 
    
    In weight 1, then there are no first or second shuffles, and the shuffle relations are mapped to 0 by $\Cor_\Mot$. Indeed, we have $\Cor_\Mot(0,0)=0$ and $\Cor_\Mot(ab,ac)=\Cor_\Mot(0,a)+\Cor_\Mot(b,c)$, since \[\Cor_\Mot(a,b)=(a-b)\in(\Lie_\MTF)^\vee_{w=1}\cong F^\times\otimes\QQ.\]
    
    For the inductive step, if $x\in\pq{\CLie_{X,S,v_\infty}^\Mot}^{\vee\circ}$, homogeneous of weight $>1$, vanishes in $\DDD^\circ(F^\times)$, then $\Cor_\Hod(r(x))=0\in\Lie_\HT^\vee$ under every embedding $r$, and $\del\Cor_\Mot(x)=0$ by the inductive hypothesis. By Lemma~\ref{lma:d0h0_rational}, $\Cor_\Mot(x)=0$.
\end{proof}

\subsection{Applications}

\subsubsection{Additive shuffle relation}

Specializing all $w_i$ to 1 in the second shuffle relation, where all $n_i=0$, we extract an \emph{additive} second shuffle relation, which does not have lower-depth terms:
\begin{cly}
Let $m,n>0$. The additive shuffle
\[\sum_{\sigma\in\Sigma_{m,n}}\Cor_\HHH(\eps_{\sigma\inv(1)},\eps_{\sigma\inv(1)}+\eps_{\sigma\inv(2)},\dots,\eps_{\sigma\inv(1)}+\dots+\eps_{\sigma\inv(m+n)},0).\]
is a constant independent of $\eps_1,\dots,\eps_n\in\CC^n\sm0$.
\end{cly}

It is easy to see that this constant is 0 if $m+n$ is even. If $m+n$ is odd, it is equal, in particular, to a sum of Hodge correlators at roots of unity.

\subsubsection{Proofs of Corollaries \ref{cly:gr28} and \ref{cly:gr2729}}

Recall Corollary~\ref{cly:gr28}
\begin{cly*}[\cite{goncharov-rudenko}, Proposition 2.8]
    For $n>2$, every Hodge correlator of weight $n$ is a linear combination of Hodge correlators of weight $n$ and depth at most $n-2$.

    Precisely, for $z_1,\dots,z_n\in\CC\du$, we have
    \begin{align}
        \Cor_\HHH(z_1,\dots,z_n,0)
        &=\sum_{i=1}^n\Cor_\HHH\pq{z_1,\dots,z_{i-1},z_i,z_i\f{z_1}{z_n},\dots,z_{n-1}\f{z_1}{z_n},z_n\f{z_1}{z_n}}\nonumber\\
        &\quad-\sum_{i=2}^n\Cor_\HHH\pq{z_1,\dots,z_{i-1},0,z_i\f{z_1}{z_n},\dots,z_{n-1}\f{z_1}{z_n},z_n\f{z_1}{z_n}}\nonumber\\
        &\quad-\Cor_\HHH\pq{z_1,z_1\cdot\f{z_1}{z_n},0,\dots,0}.\label{eqn:lower_depth_reduction_app}
    \end{align}
\end{cly*}
\begin{proof}
By multiplicative invariance, we may assume $z_1=1$. Then this is precisely the $(n-1,1)$-second shuffle relation applied to the segments
\newcommand{\ffff}[2]{#1/#2}
\[\pqg{\ffff{z_2}{z_1}}{0},\pqg{\ffff{z_3}{z_2}}{0},\dots,\pqg{\ffff{z_n}{z_{n-1}}}{0}\]
and
\[\pqg{\ffff{z_1}{z_n}}{0},\]
where the segment $(1|0)$ is left fixed. Indeed, the two summations come from the $n$ shuffles and the $n-1$ additional quasishuffles, with the remaining terms giving the left side and the last summand. 

All terms on the right side have at least two coinciding arguments. After an additive shift, they have at least two arguments equal to 0, so they are equal to those of depth at most $n-2$.
\end{proof}

Recall Corollary~\ref{cly:gr2729}:
\begin{cly*}
    The Hodge correlators in weight 3 satisfy the relations:
\begin{align}
    \Cor_\HHH(1,0,0,x)&+\Cor_\HHH(1,0,0,1-x)+\Cor_\HHH(1,0,0,1-x\inv)=\Cor_\HHH(1,0,0,1),\label{eqn:corr_27_app}\\
\Cor_\HHH(0,x,1,y)&=
-\Cor_\HHH(1,0,0,1-x\inv)
-\Cor_\HHH(1,0,0,1-y\inv)
-\Cor_\HHH\pq{1,0,0,\f yx}\nonumber\\\label{eqn:corr_29_app}
&\quad-\Cor_\HHH\pq{1,0,0,\f{1-y}{1-x}}
+\Cor_\HHH\pq{1,0,0,\f{1-y\inv}{1-x\inv}}
+\Cor_\HHH(1,0,0,1).
\end{align}
\end{cly*}
\begin{proof}
Apply the $(1,1)$-second shuffle relation to the segments $\pqg{x}{0}$ and $\pqg{x\inv}{1}$, keeping the segment $\pqg{1}{1}$ fixed:
\begin{align*}
\Cor_\HHH(1,x,0,0)&+\Cor_\HHH(1,0,x\inv,1)-\Cor_\HHH(1,0,0,1)\\&-\Cor_\HHH(1,x,0,0)-\Cor_\HHH(1,0,x\inv,0)&=0.
\end{align*}
Multiplicative invariance and the first shuffle relation imply
\[-\Cor_\HHH(1,x,0,0)-\Cor_\HHH(1,0,x\inv,0)=\Cor(1,0,0,x).\]
Rearranging terms and applying additive invariance gives (\ref{eqn:corr_27_app}).

Now apply (\ref{eqn:lower_depth_reduction_app}) to $\Cor_\HHH(x,1,y,0)$ and apply the dihedral symmetry and additive invariance to change all terms to the form $\Cor_\HHH(1,0,0,z)$:
\begin{align*}
\Cor_\HHH(0,x,1,y)&=
\Cor_\HHH\pq{1,0,0,\f{1-y}{x-y}}
+\Cor_\HHH\pq{1,0,0,\f{1-x\inv}{1-y\inv}}
+\Cor_\HHH\pq{1,0,0,\f{x-1}{x-y}}\\
&\quad-\Cor_\HHH\pq{1,0,0,1-y\inv}
-\Cor_\HHH\pq{1,0,0,x\inv}
-\Cor_\HHH\pq{1,0,0,\f xy}.
\end{align*}
Finally, by (\ref{eqn:corr_27_app}),
\[\Cor_\HHH\pq{1,0,0,\f{1-y}{x-y}}+\Cor_\HHH\pq{1,0,0,\f{x-1}{x-y}}=\Cor_\HHH\pq{1,0,0,1}-\Cor_\HHH\pq{1,0,0,\f{1-x}{1-y}},\]
which gives the result.
\end{proof}

\section{Appendix: Multiple polylogarithms}

We review the properties of multiple polylogarithms (\cite{goncharov-polylogs-modular}).

It is well known that these functions obey a family of double shuffle relations similar to our relations for the Hodge correlators. However, they do not enjoy some of their other properties. They are multi-valued and do not satisfy dihedral symmetry relations. The shuffle relations between multiple polylogarithms involve products, while for Hodge correlators they are linear.

\label{sec:zeta_shuffle}

\subsubsection{Multiple polylogarithms}

The multiple polylogarithms are defined by
\begin{equation}
\Li_{n_1,\dots,n_r}(z_1,\dots,z_r)=\sum_{0<k_1<\dots<k_r}\f{z_1^{k_1}\dots z_r^{k_r}}{k_1^{n_1}\dots k_r^{n_r}},\quad{n_1,\dots,n_r>0}.
\label{eqn:def_mult_plog}
\end{equation}
(The \emph{depth} of this formal expression is $r$ and the \emph{weight} is $w:=n_1+\dots+n_r$.) These series converge for $\aq{z_i}<1$ and have analytic continuations to multivalued functions with singularities on $\CC^r$. The multivalued structure is encoded by a smooth variation of mixed Hodge-Tate structures of weight $w$ over a dense open subset of $\CC^r$.

When $r=1$, the multiple polylogarithms are the classical polylogarithms $\Li_n(z)$. Their monodromy and associated mixed Hodge-Tate structures are well understood (\cite{hain-polylogs}).

We can form an algebra $L$ generated over $\QQ$ by the multiple polylogarithms, filtered by the weight and the depth. The expression (\ref{eqn:def_mult_plog}) yields expansions for products of polylogarithms, which shows that $L$ has a well-defined multiplication. For example,
\begin{align*}
    \Li_{n_1}(z_1)\Li_{n_2}(z_2)
    &=\pq{\sum_{0<k_1}\f{z_1^{k_1}}{k_1^{n_1}}}\pq{\sum_{0<k_2}\f{z_2^{k_2}}{k_2^{n_2}}}
    =\bq{\sum_{0<k_1<k_2}+\sum_{0<k_2<k_1}+\sum_{0<k_1=k_2}}\f{z_1^{k_1}z_1^{k_2}}{k_1^{n_1}k_2^{n_2}}\\
    \vphantom{\sum}
    &=\Li_{n_1,n_2}(z_1,z_2)+\Li_{n_2,n_1}(z_2,z_1)+\Li_{n_1+n_2}(z_1z_2).
\end{align*}
Notice that the left side and all terms on the right side have weight $n_1+n_2$; however, the left side and the first two terms on the right side have depth 2, while $\Li_{n_1+n_2}(z_1z_2)$ has depth 1. 

The general relation is:
\begin{align}
    &\Li_{n_1,\dots,n_r}(z_1,\dots,z_r)\Li_{n_{r+1},\dots,n_{r+s}}(z_{r+1},\dots,z_{r+s})\nonumber\\
    &=
    \sum_{\sigma\in\Sigma_{r,s}}\Li_{n_{\sigma\inv(1)},\dots,n_{\sigma\inv(r+s)}}(z_{\sigma\inv(1)},\dots,z_{\sigma\inv(r+s)})+\text{lower-depth terms},
\label{eqn:first_shuffle_polylog}
\end{align}
Expressions (\ref{eqn:first_shuffle_polylog}) are called \emph{first shuffle relations} for multiple polylogarithms. It is convenient to express them with generating functions. Let
\begin{equation*}
L\pg{z_1,\dots,z_r}{t_1:\dots:t_r}\sum_{n_r>0}\Li_{n_1,\dots,n_r}(z_1,\dots,z_r)\prod_it_i^{n_i-1};
\end{equation*}
then
\begin{align}
&L\pg{z_1,\dots,z_r}{t_1:\dots:t_r}L\pg{z_{r+1},\dots,z_{r+s}}{t_{r+1},\dots,t_{r+s}}
=\nonumber\\&=
\sum_{\sigma\in\Sigma_{r,s}}L\pg{z_{\sigma\inv(1)},\dots,z_{\sigma\inv(r+s)}}{t_{\sigma\inv(1)}:\dots:t_{\sigma\inv(r+s)}}+\text{lower-depth terms}.\label{eqn:l_shuf_inhom}
\end{align}

To describe the lower-depth terms in the right side of (\ref{eqn:first_shuffle_polylog}), one needs to work with the set of \emph{quasishuffles} $\widetilde\Sigma_{r,s}$.
Then
\begin{align}
    \Li_{n_1,\dots,n_r}(z_1,\dots,z_r)\Li_{n_{r+1},\dots,n_{r+s}}(z_{r+1},\dots,z_{r+s})
    =\nonumber\\=
    \sum_{\sigma\in\widetilde\Sigma_{r,s}}\Li_{\tilde n_{\sigma\inv(1)},\dots,\tilde n_{\sigma\inv(M_\sigma)}}(\tilde z_{\sigma\inv(1)},\dots,\tilde z_{\sigma\inv(M_\sigma)}),
    \label{eqn:polylog_quasishuf}
\end{align}
where
\[\tilde n_{\sigma\inv(i)}=\sum_{\sigma(j)=i}n_j,\quad\tilde z_{\sigma\inv(i)}=\prod_{\sigma(j)=i}z_j.\]
Such relations are easily proved by interpreting the terms as the simplicial decomposition of the product of an $r$-simplex and an $s$-simplex.

\subsubsection{Iterated integrals}

The analytic continuation of the multiple polylogarithms has a presentation in terms of iterated integrals. Let
\begin{equation*}
I_{n_1,\dots,n_r}(z_1:z_2:\dots:z_{r+1})=\int_{\gamma}\underbrace{\f{dt}{z_1-t}\circ\f{dt}{t}\circ\dots\circ\f{dt}{t}}_{n_1}\circ\dots\circ\underbrace{\f{dt}{z_r-t}\circ\f{dt}{t}\circ\dots\circ\f{dt}{t}}_{n_r},
\end{equation*}
where $\gamma:\bq{0,1}\to\CC$ is a path from 0 to $z_{r+1}$. Here, for 1-forms $\omega_1,\dots,\omega_r$,
\begin{equation*}
\int_\gamma\omega_1\circ\dots\circ\omega_r:=\int_{0\leq t_1\leq\dots\leq t_r\leq 1}\bigwedge_{i=1}^m\gamma^*\omega_i(t_i)
\end{equation*}
is Chen's iterated path integral (\cite{chen-iterated-integrals}). Then (\cite{goncharov-polylogs-modular}, Theorem 2.1)
\begin{equation}
\Li_{n_1,\dots,n_r}(z_1,\dots,z_r)=I_{n_1,\dots,n_r}(1:z_1:z_1z_2:\dots:z_1\dots z_r).
\label{eqn:li_iterint}
\end{equation}
Iterated path integrals also satisfy a shuffle product formula, whose terms correspond to the top-dimensional cells of a decomposition of the product of two simplices:
\begin{equation*}
\int_\gamma\omega_1\circ\dots\circ\omega_r\int_\gamma\omega_{m+1}\circ\dots\circ\omega_{m+n}=\sum_{\sigma\in\Sigma_{m,n}}\omega_{\sigma\inv(1)}\circ\dots\circ\omega_{\sigma\inv(m+n)}.
\end{equation*}
This gives a different kind of shuffle relations (\emph{second shuffle relations}) on the iterated integrals $I_{n_1,\dots,n_r}$, which can also be expressed in terms of generating functions. Let 
\begin{align}
&L'\pg{z_1:\dots:z_{r+1}}{t_1,\dots,t_r}=\nonumber\\=&\sum_{n_i>0}I_{n_1,\dots,n_r}\pq{z_1:\dots:z_{r+1}}t_1^{n_1-1}(t_1+t_2)^{n_2-1}\dots(t_1+\dots+t_r)^{n_r-1},
\label{eqn:def_hom_polylog}
\end{align}
so
\begin{equation}
L\pg{z_1,\dots,z_r}{t_1:\dots:t_r}=L'\pg{1:z_1:\dots:z_1\dots z_r}{t_1,t_2-t_1,\dots,t_r-t_{r-1}}.
\label{eqn:gf_duality_polylog}
\end{equation}
Then
\begin{align}
L'\pg{z_1:\dots:z_r:1}{t_1,\dots,t_r}L'\pg{z_{r+1}:\dots:z_{r+s}:1}{t_{r+1},\dots,t_{r+s}}
=\nonumber\\=
\sum_{\sigma\in\Sigma_{r,s}}L'\pg{z_{\sigma\inv(1)}:\dots:z_{\sigma\inv(r+s)}:1}{t_{\sigma\inv(1)},\dots,t_{\sigma\inv(r+s)}}.\label{eqn:l_shuf_hom}
\end{align}

\subsubsection{Double shuffle relations}

Note the similarity between (\ref{eqn:l_shuf_inhom}) and (\ref{eqn:l_shuf_hom}). There is a duality between the relations with homogeneous and inhomogeneous $z_i$ and $t_i$ arguments. Together, they form systems of \emph{double shuffle relations}. 

The combinatorics of such relations are studied by \cite{goncharov-polylogs-modular,goncharov-polylogs-tate}, allowing them to describe a connection between an algebra of values of the multiple polylogarithms at roots of unity and the geometry of some locally symmetric spaces for $\GL_n(\ZZ)$ ($n=2,3$; and recently for $n=4$ in \cite{goncharov-motivic-modular}). 

\subsubsection{Relation to Hodge correlators}

\label{sec:rel_hodge_polylog}

In depth 1, the Hodge correlators are related to the multiple polylogarithms. We have seen this in weights 1 and 2. In higher weight, define a single-valued version of the polylogarithm by
\begin{equation*}
\LLL_n(z)=\begin{cases}\Re&\text{$n$ odd}\\\Im&\text{$n$ even}\end{cases}\pq{\sum_{k=0}^{n-1}\beta_k\log^k\aq z\cdot\Li_{n-k}(z)}\quad(n\geq2),
\end{equation*}
where $\beta_k$, close relatives of the Bernoulli numbers, are the coefficients of the Taylor expansion $\f{2x}{e^{2x}-1}=\sum\beta_kx^k$. Then
\begin{equation}
\Cor_\HHH(1,\underbrace{0,\dots,0}_{n-1},z)=-(2\pi i)^{-n}\binom{2n-2}{n-1}\inv\sum_{\stackrel{0\leq k\leq n-2}{\text{$k$ even}}}\binom{2n-k-3}{n-1}\f{2^{k+1}}{(k+1)!}\LLL_{n-k}(z)\log^k\aq z.
\label{eqn:corr_polylog}
\end{equation}
The precise relationship between the multiple polylogarithms and Hodge correlators in depth $>1$ is unknown.

\bibliographystyle{alpham}
\bibliography{bibliography}

\end{document}